\documentclass[reqno,12pt]{amsart}

\newcommand\version{September 2, 2024}

\usepackage{amsmath,amsfonts,amsthm,amssymb,amsxtra}
\usepackage{hyperref}
\usepackage{bbm} % to get \1
\usepackage{bm} % to get \bm
\makeindex % to make index files
\usepackage[autostyle=false, style=english]{csquotes}
\MakeOuterQuote{"}

\newtheorem{theorem}{Theorem}[section]
\newtheorem{proposition}[theorem]{Proposition}
\newtheorem{lemma}[theorem]{Lemma}
\newtheorem{corollary}[theorem]{Corollary}

\theoremstyle{definition}

\newtheorem{definition}[theorem]{Definition}
\newtheorem{example}[theorem]{Example}

\theoremstyle{remark}

\newtheorem{remark}[theorem]{Remark}
\newtheorem{remarks}[theorem]{Remarks}

\numberwithin{equation}{section}

\newcommand{\C}{\mathbb{C}}

\renewcommand{\epsilon}{\varepsilon}

\newcommand{\I}{\mathbb{I}}

\newcommand{\N}{\mathbb{N}}

\renewcommand{\phi}{\varphi}
\newcommand{\R}{\mathbb{R}}

\newcommand{\Z}{\mathbb{Z}}

\def\bs{\mathbb{S}}

\def\cl{\mathcal{L}}

\def\Rd{{\mathbb{R}^d}}

\newcommand{\me}[1]{\mathrm{e}^{#1}}
\newcommand{\one}{\mathbf{1}}
\begin{document}

\title[Heat kernel bounds --- \version]{Hardy perturbations of \\ subordinated Bessel heat kernels}

\author[K. Bogdan]{Krzysztof Bogdan}
\address[Krzysztof Bogdan]{Department of Pure and Applied Mathematics, Wroc\l aw University of Science and Technology, Hoene-Wro\'nskiego 13C, 50-376 Wroc\l aw, Poland}
\email{krzysztof.bogdan@pwr.edu.pl}

\author[T. Jakubowski]{Tomasz Jakubowski}
\address[Tomasz Jakubowski]{Department of Pure and Applied Mathematics, Wroc\l aw University of Science and Technology, Hoene-Wro\'nskiego 13C, 50-376 Wroc\l aw, Poland}
\email{tomasz.jakubowski@pwr.edu.pl}

\author[K. Merz]{Konstantin Merz}
\address[Konstantin Merz]{Institut f\"ur Analysis und Algebra, Technische Universit\"at Braunschweig, Universit\"atsplatz 2, 38106 Braun\-schweig, Germany} % , and Department of Mathematics, Graduate School of Science, Osaka University, Toyonaka, Osaka 560-0043, Japan}
\email{k.merz@tu-bs.de}


\date{\version}

\begin{abstract}
  Motivated by the spectral theory of relativistic atoms,  we prove matching upper and lower bounds for the transition density of Hardy perturbations of subordinated Bessel heat kernels. The analysis is based on suitable supermedian functions, in particular invariant functions.
\end{abstract}

\thanks{This research was supported in part by grant Opus 2023/51/B/ST1/02209 of National Science Center, Poland (K.B.), and the PRIME programme of the German Academic Exchange Service DAAD with funds from the German Federal Ministry of Education and Research BMBF (K.M.).}

\maketitle
\vspace*{-2em}
\tableofcontents

\section{Introduction and main result}
\label{s:intro}

Bessel operators are fundamental for mathematical analysis, theoretical physics, and applied sciences. We are interested in their fractional powers with added critically singular Hardy potential,
\begin{equation}
  \label{e.fBo}
  (-\tfrac{d^2}{dr^2}-\tfrac{2\zeta}{r}\,\tfrac{d}{dr})^{\alpha/2}-\kappa r^{-\alpha}
\end{equation}
in $L^2(\R_+,r^{2\zeta}dr)$ with $\R_+=(0,\infty)$.
Here and in what follows, 
\begin{align}
  \label{eq:globalassump}
  \zeta\in(-1/2,\infty),
  \quad
  \alpha\in(0,2]\cap(0,2\zeta+1),    
\end{align}
and $\kappa\in \R$ is called the \textit{coupling constant} of the Hardy potential.
Below, we call (the potential) $\kappa r^{-\alpha}$ attractive (mass-creating) if $\kappa>0$ and repulsive (mass-killing) if $\kappa<0$.
Our interest in the operators \eqref{e.fBo} comes from the spectral theory of atoms. 
We pursue a program initiated in \cite{BogdanMerz2024} to describe the different rates of spectral dissipation of mass in various \textit{angular momentum channels} in $\Rd$, as parameterized by different values of $\zeta$, $\alpha$, and $\kappa$. From the point of view of theoretical physics, the most important is the three-dimensional pseudo-relativistic case  $\alpha=1$, with $\zeta=1+\ell$,  $\ell=0,1,\ldots$, and Coulomb potential $\kappa r^{-1}$. Here,  the operator \eqref{e.fBo} arises as one of the direct summands of the direct sum decomposition of $(-\Delta)^{1/2}-\kappa/|x|$ in $L^2(\R^3)$ into different angular momentum channels, indexed by $\ell\in\N_0$. Here,
$\kappa$ defines the strength of coupling  between the nucleus and electrons in an atom. The passage from the state space $\R^3$, or, more generally, $\Rd$ with $d\in\N$, to the half-line $\R_+$ in \eqref{e.fBo} is inspired by the intertwining in \cite[Proposition 4.3]{BogdanMerz2024}. In fact, in the present paper, we abstract from the setting of \cite{BogdanMerz2024} and study, in maximal generality, a semigroup on $\R_+$ associated with \eqref{e.fBo}.
The spectral analysis of \cite{BogdanMerz2024} in $\Rd$ will be picked up again in the forthcoming paper \cite{BogdanMerz2024H}, where we also clarify the definition of \eqref{e.fBo} as a self-adjoint operator.
The main result of the present paper, Theorem~\ref{mainresultgen} below, provides a crucial technical ingredient to be used in \cite{BogdanMerz2024H}, namely non-explosion and sharp heat kernel bounds for a semigroup associated with \eqref{e.fBo}.

To state the result, we briefly introduce necessary notation, but defer precise definitions to the ensuing subsections. In particular, here and below, we use a specific parameterization of the coupling constant $\kappa$ in \eqref{e.fBo}. Thus, for $\zeta\in(-1/2,\infty)$, we let 
\begin{align}\index{$\Psi_\zeta(\eta)$}
\label{eq:defpsietazetaOLD}
  \begin{split}
    \kappa=\Psi_\zeta(\eta) := 
    \begin{cases}
      \displaystyle \frac{2^\alpha\Gamma\left(\frac{2\zeta+1-\eta}{2}\right)\Gamma\left(\frac{\alpha+\eta}{2}\right)}{\Gamma\left(\frac{\eta}{2}\right)\Gamma\left(\frac{2\zeta+1-\eta-\alpha}{2}\right)} & \ \text{if $\alpha\in(0,2\wedge (2\zeta+1))$ and $\eta\in(-\alpha,2\zeta+1)$},
      \vspace{3pt} \\ 
      (2\zeta-1-\eta)\eta & \ \text{if $\alpha=2<2\zeta+1$ and $\eta\in\R$}.
    \end{cases}
  \end{split}
\end{align}
Here, as usual, $A\wedge B:=\min\{A,B\}$ and $A\vee B:=\max\{A,B\}$\index{$\wedge$}\index{$\vee$}.
We introduce further notation as we proceed; it is summarized in the \hyperref[index]{\indexname} at the end of the paper.
Clearly, $\Psi_\zeta(\eta)$ is symmetric about $\eta=(2\zeta+1-\alpha)/2$.
In view of \cite[Remark~3.3]{BogdanMerz2024},  its maximum is 
$$
\kappa_{\rm c}(\zeta,\alpha):=\Psi_\zeta((2\zeta+1-\alpha)/2).
$$
For instance, in the three-dimensional, pseudorelativistic situation with $\zeta=1+\ell$, 
\[
  \kappa_{\rm c}(1+\ell,1) = \frac{2\Gamma(1+\ell/2)^2}{\Gamma((\ell+1)/2)^2}.
\]
In particular, for \(\ell=0\), we get \(\kappa_{\rm c}(1,1)=2/\pi\) while for \(\ell=1\), we get \(\kappa_{\rm c}(2,1)=\pi/2\).
In the general situation, $\Psi_\zeta(\eta)$ is strictly positive for $\eta\in(0,2\zeta+1-\alpha)$, zero for $\eta\in\{0,2\zeta+1-\alpha\}$, and tends to $-\infty$ as $\eta\to -M$ or $\eta\to2\zeta+1-\alpha+M$. Here $M:=\alpha$ if $\alpha<2$ and $M:=\infty$ for $\alpha=2$. Thus, the parameterization $\kappa=\Psi_\zeta(\eta)$ only produces
$\kappa\in (-\infty,\kappa_{\rm c}(\zeta,\alpha)]$.
For this range,
we may and often do restrict $\eta$ to the interval $(-M,(2\zeta+1-\alpha)/2]$.
Moreover, for $\zeta>\zeta'>-1/2$, we have $\Psi_{\zeta}(\eta)>\Psi_{\zeta'}(\eta)$ for all 
$\eta$ in the intersection of the domains of $\Psi_\zeta$ and $\Psi_{\zeta'}$.
By the Hardy inequality, in fact, by the ground state representation \cite[Theorem~3.2]{BogdanMerz2024}, the quadratic form corresponding to \eqref{e.fBo} on $L^2(\R_+,r^{2\zeta}\,dr)$ is bounded from below if $\kappa=\Psi_\zeta(\eta)$ for some $\eta$
and it is unbounded from below if $\kappa>\kappa_{\rm c}(\zeta,\alpha)$. 
In the former case, the form is bounded from below by zero. This observation motivates to call $\kappa_{\rm c}(\zeta,\alpha)$ the critical coupling constant.

Summarizing, in most of the discussion below, \eqref{e.fBo} is substituted by
\begin{equation}
  \label{e.fBoeta}
  (-\tfrac{d^2}{dr^2}-\tfrac{2\zeta}{r}\,\tfrac{d}{dr})^{\alpha/2}-\Psi_\zeta(\eta) r^{-\alpha},
\end{equation}
for $\eta$ \textit{admissible}, that is, in the domain of $\Psi_\zeta$, see \eqref{eq:defpsietazetaOLD}.
We call $2\zeta+1$ the \textit{effective dimension}.
This is because the reference measure for the Hilbert space in which \eqref{e.fBo} is defined is $r^{2\zeta}dr=r^{2\zeta+1}\frac{dr}{r}$, where $\frac{dr}{r}$ is the Haar measure on the multiplicative group $(0,\infty)$.
As we shall see, $\eta$ conveniently parametrizes Hardy potentials, and, of course, $\alpha/2$ is the order of subordination. All the operators \eqref{e.fBoeta} are homogeneous of order $\alpha$ (with respect to dilations), but
the interplay of the three parameters critically influences their potential-theoretic properties.
We let
$$
p_\zeta^{(\alpha)}(t,r,s), \quad t, r,s>0,
$$
be the transition density associated to \eqref{e.fBo} with $\kappa=0$; see \eqref{eq:defpheatalpha2} and \eqref{eq:defpheatalpha} for actual definitions in the cases $\alpha=2$ and $\alpha<2$, respectively. Here, $t>0$ is \textit{time}, but $r,s\in (0,\infty)$ are positions in \textit{space}.
Moreover, we let
$$
p_{\zeta,\eta}^{(\alpha)}(t,r,s), \quad t, r,s>0,
$$
be the Schr\"odinger perturbation of $p_\zeta^{(\alpha)}$ by the Hardy potential $\Psi_\zeta(\eta)/r^\alpha$, 
as in Definition~\ref{def:pzetaeta}.
In \cite{BogdanMerz2024H}, we will show that $p_{\zeta,\eta}^{(\alpha)}$ equals the heat kernel of \eqref{e.fBoeta} in the sense of quadratic forms.
The main contribution of the present paper are the following \textit{matching} upper and lower, or \textit{sharp}, bounds for $p_{\zeta,\eta}^{(\alpha)}$, which are proved in Section~\ref{s:proofmainresultgen}.
\begin{theorem}
  \label{mainresultgen}
  If $\zeta\in(-1/2,\infty)$, $\alpha\in(0,2)\cap(0,2\zeta+1)$, and $\eta\in(-\alpha,\frac{2\zeta+1-\alpha}{2}]$, then
  \begin{align}
    \label{eq:mainresultgen}
    \begin{split}
      p_{\zeta,\eta}^{(\alpha)}(t,r,s)
      \sim_{\zeta,\alpha,\eta} \left(1\wedge\frac{r}{t^{1/\alpha}}\right)^{-\eta}\, \left(1\wedge\frac{s}{t^{1/\alpha}}\right)^{-\eta}\, p_\zeta^{(\alpha)}(t,r,s),
    \end{split}
  \end{align}
  and $p_{\zeta,\eta}^{(\alpha)}(t,r,s)$ 
is a jointly continuous function of $t,r,s>0$.
\end{theorem}
Here, we write $A \sim B$\index{$\sim$}, or $B \lesssim A \lesssim B$, if $A, B \geq 0$ and $c^{-1}B \leq A \leq cB$ for a constant \(c \in (0,\infty)\). We write $A \sim_\tau B$ if $c = c_\tau$ depends on \(\tau\), etc. The dependence on fixed parameters, like $\zeta,\alpha,\eta$ is usually omitted.

In view of the sharp Hardy inequality in \cite{BogdanMerz2024}, one may duly expect a \textit{blowup} for \textit{supercritical} coupling constants $\kappa$, i.e., $\kappa>\kappa_{\rm c}(\zeta,\alpha)$. 
Indeed, the interpretation and proof of the following corollary to Theorem~\ref{mainresultgen} are given in Subsection~\ref{ss:blowup} below.
\begin{corollary}
  \label{blowupcor}
  Let $\zeta\in(-1/2,\infty)$ and $\alpha\in(0,2)$. Then  the heat kernel associated with \eqref{e.fBo} blows up if $\kappa >\kappa_{\rm c}(\zeta,\alpha)$.
\end{corollary}

\begin{remark} 
  \label{remarksmainresult1}
  For $\alpha=2$, we have the identity
  \begin{align}
    \label{eq:mainresultalpha2}
    p_{\zeta,\eta}^{(2)}(t,r,s) = \left(rs\right)^{-\eta}p_{\zeta-\eta}^{(2)}(t,r,s).    
  \end{align}
  It is proved, e.g., by Metafune, Negro, and Spina \cite[Theorem~4.12, Proposition~4.14]{Metafuneetal2018} by a change variables; see also the beginning of Section~\ref{s:proofmainresultgen} for a sketch of proof. For $\alpha=2$, \eqref{eq:mainresultalpha2} and  \eqref{eq:easybounds2} yield that $p_{\zeta,\eta}^{(2)}(t,r,s)$ is jointly continuous on $(0,\infty)^3$ and obeys \eqref{eq:mainresultgen}, provided $\sim$ is replaced by the notation $\asymp$ explained in the paragraph below \eqref{eq:pzetafact}. Moreover, the blow up result in Corollary~\ref{blowupcor} extends to $\alpha=2$.
\end{remark}

Below in this paper we develop  potential-theoretic techniques, which we refer to as the \textit{integral method}, to study kernels appearing in \cite{BogdanMerz2024} and \cite{BogdanMerz2024H}. This may be considered as a complementary study largely independent of \cite{BogdanMerz2024} and \cite{BogdanMerz2024H}. Namely, we abstract from the setting of $\Rd$ and the angular momentum channels parameterized by $\ell$ in \cite{BogdanMerz2024} by projecting, transferring, or intertwining \cite{PatieSavov2021} onto the half-line $\R_+$. This allows us to use sub-Markovian semigroups to capture delicate oscillations that influence the different rates of dissipation of mass for different spherical harmonics. This method is not only relevant for this context but is also of independent interest. The addition of the Hardy potential has an analogous effect on the heat kernel, changing the dissipation rate but preserving \textit{scaling}.

In fact, additional motivation for our work 
comes from the fact that models with scaling are of fundamental interest for PDEs, functional analysis, probability, and physics; for further connections, see, e.g., 
\cite{Bileretal2001}, \cite{Chaumontetal2013}, \cite{Vazquez2018}, \cite{PatieSavov2021}, \cite{DerezinskiGeorgescu2021}, \cite{Kyprianouetal2021}, \cite{Choetal2020}, \cite{Bogdanetal2023}, \cite{Metafuneetal2023}, \cite{Armstrongetal2023}.

The main methodological inspiration for this work comes from Bogdan, Grzywny, Jakubowski, and Pilarczyk \cite{Bogdanetal2019} and Jakubowski and Wang \cite{JakubowskiWang2020}. However, the present development differs in several important ways and is more complex. Namely, the projection onto \(\mathbb{R}_+\) unfortunately eliminates both the origin and rotational invariance from the setting. Furthermore, we address \(\eta > 0\) and \(\eta < 0\) simultaneously. Moreover, the integral method for negative perturbations involves rather delicate compensation. 
Finally, our resolution is expected to have significant consequences for the multi-channeled spectral resolution of the Laplacian and related operators.

Let us note that Cho, Kim, Song, and Vondra\v{c}ek \cite{Choetal2020} (see also Song, Wu, and Wu \cite{Songetal2022}) and Jakubowski and Maciocha \cite{JakubowskiMaciocha2023} recently proved heat kernel bounds for the fractional Laplacian with Hardy potential on $\R_+$ and Dirichlet boundary condition on $(\R_+)^c$. As in our analysis, \cite{JakubowskiMaciocha2023} use the integral method introduced in  \cite{Bogdanetal2019}, but, in contrast to our approach, they cannot rely on subordination since the Dirichlet fractional Laplacian is 
different from the fractional Dirichlet Laplacian, see, e.g., \cite[Lemma~6.4]{FrankGeisinger2016}. 
Consequently, the results of \cite{JakubowskiMaciocha2023G} and \cite{JakubowskiMaciocha2023} are different than ours, as can be readily seen by comparing the respective parameterizations of the coupling constants, cf. \cite[(1.8)]{JakubowskiMaciocha2023G} and \eqref{eq:defpsietazetaOLD}.

In the following subsections, we make our setting more precise. To this end, we first define the subordinated Bessel heat kernels $p_\zeta^{(\alpha)}$ and recall some of their properties. 
Then  we give a precise definition of $p_{\zeta,\eta}^{(\alpha)}$.
At the end of this section, we outline implications of Theorem~\ref{mainresultgen} and the structure of the rest of the paper.
Recall that in \cite{BogdanMerz2024H}, we will prove that $p_{\zeta,\eta}^{(\alpha)}$ equals the heat kernel of \eqref{e.fBo} in the sense of quadratic forms for $\kappa=\Psi_\zeta(\eta)$, but we mention this fact only for motivation and of course do not use it in the present paper.

\subsection{Subordinated Bessel heat kernels}
\label{s:semigroupproperties}
We define the Bessel heat kernel
\begin{align}\index{$p_\zeta^{(2)}(t,r,s)$}
  \label{eq:defpheatalpha2}
  \begin{split}
    p_\zeta^{(2)}(t,r,s) & : = \frac{(rs)^{1/2-\zeta}}{2t}\exp\left(-\frac{r^2+s^2}{4t}\right)I_{\zeta-1/2}\left(\frac{rs}{2t}\right), \quad r,s,t>0.
  \end{split}
\end{align}
Here, for $z\in\C\setminus(-\infty,0]$, $I_\nu(z)$ denotes the modified Bessel functions of the first kind of order $\nu\in\C$ \cite[(10.25.2)]{NIST:DLMF}.\index{$I_\nu$}
Note that
\begin{align}
  \label{eq:pzetatzero}
  p_\zeta^{(2)}(t,r,0)
  = \frac{2^{-2\zeta}}{\Gamma(\frac{2\zeta+1}{2})} t^{-\frac{2\zeta+1}{2}}\exp\left(-\frac{r^2}{4t}\right)
  \end{align}
by the series expansion \cite[(10.25.2))]{NIST:DLMF} of the modified Bessel function.

The kernel $p_\zeta^{(2)}(t,r,s)$ with the reference (speed) measure $r^{2\zeta}dr$ on $\R_+$ is the transition density of the Bessel process with index $\zeta-1/2$ \emph{reflected at the origin}. We remark that the Bessel process with index $\zeta-1/2$ \emph{killed at the origin} has the transition density~\eqref{eq:defpheatalpha2}, but with $I_{\zeta-1/2}(\cdot)$ replaced with $I_{|\zeta-1/2|}(\cdot)$. When $\zeta\geq1/2$, the Bessel process never hits the origin, i.e., no condition (reflecting or killing) takes effect or needs to be imposed at the origin. See, e.g., \cite{Maleckietal2016}. We also remark that $p_\zeta^{(2)}(t,r,s)$ is the heat kernel of the nonnegative Bessel operator of index $\zeta-1/2$ in the Liouville form,
\begin{align}
  \label{eq:deflzeta}
  \cl_\zeta = -\frac{d^2}{dr^2} - \frac{2\zeta}{r}\frac{d}{dr}
  \quad \text{in} \ L^2(\R_+,r^{2\zeta}dr).
\end{align}
More precisely, we understand $\cl_\zeta$ as the Friedrichs extension of the corresponding quadratic form on $C_c^\infty(\R_+)$ if $\zeta\in[1/2,\infty)$, but we understand $\cl_\zeta$ as the Krein extension of the corresponding quadratic form on $C_c^\infty(\R_+)$ if $\zeta\in(-1/2,1/2]$. In both cases, $\cl_\zeta$ is nonnegative and in particular self-adjoint. We refer, e.g., to \cite[Theorem~4.22]{Bruneauetal2011}, \cite[Theorem~8.4]{DerezinskiGeorgescu2021}, or \cite[Section~4]{Metafuneetal2018} for further details and references regarding the spectral theory of $\cl_\zeta$, in particular for operator and quadratic form domains and (form) cores. In that setting,
\begin{align}
  \label{eq:heatkernellzeta}
  p_\zeta^{(2)}(t,r,s) = \me{-t\cl_\zeta}(r,s), \quad t,r,s>0.
\end{align}
This information is important to prove Theorem~\ref{mainresultgen} when $\alpha=2$.

Next, we define, for $\alpha\in(0,2)$, the $\tfrac\alpha2$-subordinated Bessel heat kernels. To that end, recall that for $\alpha\in(0,2)$ and $t>0$, by Bernstein's theorem, the completely monotone function $[0,\infty)\ni\lambda\mapsto\me{-t\lambda^{\alpha/2}}$  is the Laplace transform of a probability density function $\R_+\ni\tau\mapsto\sigma_t^{(\alpha/2)}(\tau)$. That is,
\begin{align}
  \label{eq:subordination}
  \me{-t\lambda^{\alpha/2}} = \int_0^\infty d\tau\, \sigma_t^{(\alpha/2)}(\tau)\, \me{-\tau\lambda}, \quad t>0,\,\lambda\geq0,
\end{align}
see, e.g., \cite[Chapter~5]{Schillingetal2012}.
We give some useful properties of and sharp estimates for $\sigma_t^{(\alpha/2)}(\tau)$, and references in \cite[Appendix~B]{BogdanMerz2024}. Using \eqref{eq:subordination}, we define the $\tfrac\alpha2$-subordinated Bessel heat kernel with reference measure $r^{2\zeta}dr$ on $\R_+$ as
\begin{align}\index{$p_\zeta^{(\alpha)}(t,r,s)$}
  \label{eq:defpheatalpha}
  \begin{split}
    p_\zeta^{(\alpha)}(t,r,s) & : = \int_0^\infty d\tau\, \sigma_t^{(\alpha/2)}(\tau)\, p_\zeta^{(2)}(\tau,r,s), \quad r,s,t>0.
  \end{split}
\end{align}

Crucially, $p_\zeta^{(\alpha)}(t,r,s)$ is a probability transition density. In fact, it also defines a strongly continuous contraction semigroup on $L^2(\R_+,r^{2\zeta}dr)$. We record this in the following proposition, which is proved, e.g., in \cite{BogdanMerz2024,BogdanMerz2023S}.

\begin{lemma}
  \label{summarypropertiespzeta}
  Let $\zeta\!\in\!(-1/2,\infty)$, $\alpha\!\in\!(0,2]$, and $t,t',r,s\!>\!0$. Then  $p_\zeta^{(\alpha)}(t,r,s)\!>\!0$,
  \begin{align}
    \label{eq:normalizedalpha}
    & \int_0^\infty ds\, s^{2\zeta} p_\zeta^{(\alpha)}(t,r,s) = 1, \\
    \label{eq:chapman}
    & \int_0^\infty dz\, z^{2\zeta} \, p_\zeta^{(\alpha)}(t,r,z) p_\zeta^{(\alpha)}(t',z,s) = p_\zeta^{(\alpha)}(t+t',r,s), \\
    \label{eq:scalingalpha}
    & p_\zeta^{(\alpha)}(t,r,s) = t^{-\frac{2\zeta+1}{\alpha}} p_\zeta^{(\alpha)}\left(1,\frac{r}{t^{1/\alpha}},\frac{s}{t^{1/\alpha}}\right),
  \end{align}
  and $\{p_\zeta(t,\cdot,\cdot)\}_{t>0}$ is a strongly continuous contraction semigroup on $L^2(\R_+,r^{2\zeta}dr)$.
\end{lemma}

\subsection{Hardy-type Schr\"odinger perturbations}
\label{ss:mainresult}

Let $\zeta\in(-1/2,\infty)$, $\alpha\in(0,2]\cap(0,2\zeta+1)$, and recall that $M=\alpha$ when $\alpha\in(0,2)$ and $M=\infty$ when $\alpha=2$. For $\eta\in(-M,2\zeta+1)$, we let
\begin{align}\index{$q(z)$}
  \label{eq:defq}
  q(z) := \frac{\Psi_\zeta(\eta)}{z^\alpha}, \quad z>0.
\end{align}
Because of symmetry of $\Psi_\zeta(\eta)$ about $\eta=(2\zeta+1-\alpha)/2$, as a rule we assume $\eta\in(-M,(2\zeta+1-\alpha)/2]$ when considering $q$.
In  the Appendix~\ref{s:constructionnegschrodperturbation}, for the convenience of the reader, we discuss rather general Schr\"odinger perturbations of transition densities. The focus of this paper is, however, on the following cases.

\begin{definition}
  \label{def:pzetaeta}
  For $\zeta\in(-1/2,\infty)$, $\alpha\in(0,2)\cap(0,2\zeta+1)$, and $\eta\in (-\alpha,\frac{2\zeta+1-\alpha}{2}]$, let
  $p_{\zeta,\eta}^{(\alpha)}$ be
  the Schr\"odinger perturbation of $p_\zeta^{(\alpha)}$ by $q$ in \eqref{eq:defq}, defined in the Appendix~\ref{s:constructionnegschrodperturbation}.
  For $\alpha=2$ with $\alpha<2\zeta+1$, and $\eta\in(-\infty,\frac{2\zeta-1}{2}]$, let $p_{\zeta,\eta}^{(2)}(t,r,s)$ be the heat kernel of the self-adjoint operator $\cl_\zeta-\Psi_\zeta(\eta)/r^2$, with the domain described in \cite[Section~4]{Metafuneetal2018}.
\end{definition}
Thus, in the case $\alpha=2$, we typically refer to the literature. For $\alpha\in (0,2)$, we rely on the discussion
 in the Appendix~\ref{s:constructionnegschrodperturbation}: a Schr\"odinger perturbation adds mass to the semigroup if $q>0$ and decreases mass if $q<0$. Thus, for $\eta=0$, we have $q=0$ and $p_{\zeta,0}^{(\alpha)}=p_{\zeta}^{(\alpha)}$. For $\eta\in(0,\frac{2\zeta+1-\alpha}{2}]$, we have $q> 0$ and
\begin{align}\index{$p_{\zeta,\eta}^{(\alpha)}$}
  \label{eq:feynmankactransformed}
  \begin{split}
    p_{\zeta,\eta}^{(\alpha)}(t,r,s)
    & := \sum_{n\geq0} p_t^{(n,D)}(r,s), \quad \text{where} \quad
      p_t^{(0,D)}(r,s) := p_\zeta^{(\alpha)}(t,r,s), \quad \text{and} \\
    p_t^{(n,D)}(r,s) & := \int_0^t d\tau \int_0^\infty dz\, z^{2\zeta} p_\zeta^{(\alpha)}(t,r,z) q(z) p_{t-\tau}^{(n-1,D)}(z,s), \quad n\in\N.
  \end{split}
\end{align}
Each term $p_t^{(n,D)}(r,s)$ may be understood as an iteration of \emph{Duhamel's} formula below,
hence the superscript $D$.
Of course, we have the following bound
\begin{align}
  \label{eq:trivialupperbound}
  p_{\zeta,\eta}^{(\alpha)}(t,r,s) \geq p_\zeta^{(\alpha)}(t,r,s) \quad \text{for} \ \eta\in\left(0,\frac{2\zeta+1-\alpha}{2}\right].
\end{align}
For $\eta\in(-\alpha,0)$, we have $q<0$ and $p_{\zeta,\eta}^{(\alpha)}\leq p_{\zeta}^{(\alpha)}$; see the Appendix~\ref{s:constructionnegschrodperturbation}, in particular Theorem~\ref{thm:1}, for the  construction and properties of negative Schr\"odinger perturbations of transition densities using the more comfortable and general time-inhomogeneous setting. See the remarks before Theorem~\ref{thm:1}, too.
We also note that to deal with $\eta<0$ in Subsection~\ref{ss:integralanalysisnegativeeta}, 
we use the results on $\eta>0$ from Subsection~\ref{s.pc}. The structure of Section~\ref{s:integralana} reflects this connection.

We also note that $p_{\zeta,\eta}^{(\alpha)}$ may also be defined by the Feynman--Kac formula (for bridges), but the singularity of $z^{-\alpha}$ in \eqref{eq:defq} makes this problematic for large negative $\Psi_\zeta(\eta)$, i.e., $\eta$ close to $-\alpha$. 
Actually, a viable (probabilistic) approach to negative $\eta$, based on the Feynman-Kac formula and approximation by the killed semigroup, is suggested by \cite[p.~6]{ChoSong2024}, but we prefer to keep the paper more analytic, and propose the potential-theoretic approach of  the Appendix~\ref{s:constructionnegschrodperturbation}. 

In the following, we write $h_\beta(r):=r^{-\beta}$\index{$h_\beta(r)$} for $r>0$ and $\beta\in\R$, and abbreviate $$h:=h_\eta.$$\index{$h(r)$}
We record some fundamental properties of $p_{\zeta,\eta}^{(\alpha)}$.

\begin{proposition}
  \label{propertiesschrodheatkernel}
  Let $\zeta\in(-1/2,\infty)$, $\alpha\in(0,2]\cap(0,2\zeta+1)$, $\eta\in(-M,\frac{2\zeta+1-\alpha}{2}]$, and $r,s,t,t'>0$. Then the following statements hold. \\
  \textup{(1)}
  We have the Chapman--Kolmogorov equation,
  \begin{align}
    \int_0^\infty dz\, z^{2\zeta} p_{\zeta,\eta}^{(\alpha)}(t,r,z) \, p_{\zeta,\eta}^{(\alpha)}(t',z,s)
    & = p_{\zeta,\eta}^{(\alpha)}(t+t',r,s).
  \end{align}
  \textup{(2)}
  We have the Duhamel formula
  \begin{align}
    \label{eq:duhamelclassictransformed}
    \begin{split}
      p_{\zeta,\eta}^{(\alpha)}(t,r,s)
      & = p_\zeta^{(\alpha)}(t,r,s) + \int_0^t d\tau \int_0^\infty dz\, z^{2\zeta} p_\zeta^{(\alpha)}(\tau,r,z) \, q(z) \, p_{\zeta,\eta}^{(\alpha)}(t-\tau,z,s) \\
      & = p_\zeta^{(\alpha)}(t,r,s) + \int_0^t d\tau \int_0^\infty dz\, z^{2\zeta} p_{\zeta,\eta}^{(\alpha)}(\tau,r,z) \, q(z) \, p_\zeta^{(\alpha)}(t-\tau,z,s).
    \end{split}
  \end{align}
  \textup{(3)}
  We have the scaling relations $p_t^{(n,D)}(r,s)=t^{-\frac{2\zeta+1}{\alpha}}p_1^{(n,D)}(r/t^{1/\alpha},s/t^{1/\alpha})$ and
  \begin{align}
    \label{eq:scalingalphahardy}
    p_{\zeta,\eta}^{(\alpha)}(t,r,s)
    = t^{-\frac{2\zeta+1}{\alpha}} p_{\zeta,\eta}^{(\alpha)}\left(1,\frac{r}{t^{1/\alpha}},\frac{s}{t^{1/\alpha}}\right).
  \end{align}
  \textup{(4)}
  Let $\eta\in[0,\frac{2\zeta+1-\alpha}{2}]$. Then the function $h(r)=r^{-\eta}$ is supermedian with respect to $p_{\zeta,\eta}^{(\alpha)}$, i.e.,
  \begin{align}    \label{eq:supermediantransformedalphaintro}
    \begin{split}
      \int_0^\infty ds\, s^{2\zeta} p_{\zeta,\eta}^{(\alpha)}(t,r,s) h(s)
      & \leq h(r).
    \end{split}
  \end{align}
  In particular, $p_{\zeta,\eta}^{(\alpha)}(t,r,s)<\infty$ for all $t,r>0$ and almost all $s>0$.
\end{proposition}

\begin{proof}
  For $\alpha=2$, these claims follow by using the explicit expression in \eqref{eq:mainresultalpha2}.
  Thus, suppose $\alpha\in(0,2)$ from now on.
  The scaling \eqref{eq:scalingalphahardy} follows from \eqref{eq:scalingalpha}, \eqref{eq:feynmankactransformed} when $\eta\geq0$ and Theorem~\ref{thm:1} when $\eta<0$, and induction.
  For $\eta\geq0$, the other three properties follow, e.g., from \cite{Bogdanetal2008} and \cite[Theorem~1]{Bogdanetal2016} together with the computation \eqref{eq:defhbetagammaalphatransformed}.
  For $\eta<0$, the Chapman--Kolmogorov equation and the Duhamel formula follow by the construction in Theorem~\ref{thm:1}.
\end{proof}

\subsection{Implications of Theorem~\ref{mainresultgen}}
\label{ss:implications}

In the pioneering paper \cite{Bogdanetal2019}, Bogdan, Grzywny, Jakubowski, and Pilarczyk proved estimates for the heat kernel of the homogeneous Schr\"odinger operator $(-\Delta)^{\alpha/2}-\kappa/|x|^\alpha$ on $\R^d$, sometimes called Hardy operator, with $\alpha\in(0,2\wedge d)$ and $0\le \kappa\le \kappa^*$,
where
\begin{equation}\label{e.ks}
   \kappa^*=\frac{2^{\alpha} \Gamma((d+\alpha)/4)^2 }{\Gamma((d-\alpha)/4)^{2}} 
   = \Psi_{\frac{d-1}{2}}\left(\frac{d-\alpha}{2}\right).
\end{equation}
Among others, their bounds were crucial for the systematic study of the $L^p$-Sobolev spaces generated by powers of $(-\Delta)^{\alpha/2}-\kappa/|x|^\alpha$ in \cite{Franketal2021,Merz2021,BuiDAncona2023,BuiNader2022}---see the founding paper \cite{Killipetal2018} dealing with the case $\alpha=2$---and to prove the strong Scott conjecture concerning the electron distribution of large relativistic atoms \cite{Franketal2020P,Franketal2023,MerzSiedentop2022,Franketal2023T}.

In \cite{BogdanMerz2024}, we made a first step into a more detailed analysis of the Hardy operator by taking its spherical symmetry into account. More precisely, for $d\in\N$, $\ell\in L_d\subseteq\N_0$, and $m\in M_{d,\ell}\subseteq\Z^{d-2}$ (for $d\geq3$)\footnote{Explicit descriptions of $L_d$ and $M_{d,\ell}$ are not important here, but can be found in \cite{BogdanMerz2024}.}, let $Y_{\ell,m}$ denote the orthonormal basis of (hyper)spherical harmonics in $L^2(\bs^{d-1})$. In \cite{BogdanMerz2024}, we considered the function space 
$$
V_{\ell,m}:=\{u(|x|)|x|^\ell Y_{\ell,m}(x/|x|):\, u\in L^2(\R_+,r^{d+2\ell-1}dr)\}
$$
of functions with given angular momentum $\ell$.
For $[u]_{\ell,m}(x):=u(|x|)|x|^\ell Y_{\ell,m}(x/|x|)$ with $u\in L^2(\R_+,r^{d-1+2\ell}dr)$, we showed
\begin{align}
  \label{eq:hardyopform}
  \begin{split}
    & \left\langle[u]_{\ell,m},\left((-\Delta)^{\alpha/2}-\frac{\kappa}{|x|^\alpha}\right)[u]_{\ell,m}\right\rangle_{L^2(\R^d)} \\
    & \quad = \left\langle u,\left(\left(-\frac{d^2}{dr^2}-\frac{2\zeta}{r}\frac{d}{dr}\right)^{\alpha/2}-\frac{\kappa}{r^\alpha}\right)u\right\rangle_{L^2(\R_+,r^{d+2\ell-1}\,dr)},
  \end{split}
\end{align}
whenever $\alpha<d+2\ell$, $\kappa=\Psi_{\frac{d-1+2\ell}{2}}(\eta)$ and all $\eta\in(-\alpha,d+2\ell)$. In particular, the two sides of \eqref{eq:hardyopform} are nonnegative if and only if $\kappa\leq\Psi_{\frac{d-1+2\ell}{2}}((d+2\ell-\alpha)/2)$.

Moreover, we constructed the \textit{generalized ground states} and proved a \textit{ground state representation} or Hardy identity for the quadratic forms \eqref{eq:hardyopform}, see \cite[Theorem~1.5]{BogdanMerz2024}. This refined the ground state representation of $(-\Delta)^{\alpha/2}-\kappa/|x|^\alpha$ in $L^2(\R^d)$, first proved in \cite[Proposition~4.1]{Franketal2008H} by Fourier transform\footnote{For $\alpha=2$, the ground state representation of $-\Delta-\kappa/|x|^2$ was known earlier; see, e.g., \cite[p.~169]{ReedSimon1975}.}; see also \cite[Proposition~5]{Bogdanetal2016} for a more abstract approach.
These generalized ground states are just $h(r)=r^{-\eta}$ and are crucial for our estimates of 
$p_{\zeta,\eta}^{(\alpha)}(t,r,s)$.
The main result of our forthcoming paper \cite{BogdanMerz2024H} states that $p_{\zeta,\eta}^{(\alpha)}(t,r,s)$ is indeed the heat kernel of the restriction of $(-\Delta)^{\alpha/2}-\Psi_{\frac{d+2\ell-1}{2}}(\eta)/|x|^\alpha$ to $V_{\ell,m}$.

Let us now comment on open problems and connections.
The above discussion indicates that the estimates in Theorem~\ref{mainresultgen}
will be important to advance the study of $L^p$-Sobolev spaces generated by powers of Hardy operators.
Our bounds should also lead to limits at the origin and self-similar solutions for the considered kernels, see \cite{Bogdanetal2023}.
Extensions to other homogeneous operators, including those with gradient perturbations, are of interest, too, see Metafune, Negro, and Spina \cite{Metafuneetal2018}
and Kinzebulatov, Sem\"enov, and Szczypkowski \cite{Kinzebulatovetal2021}.

Moreover, Theorem~\ref{mainresultgen} and the results in our forthcoming paper \cite{BogdanMerz2024H} are paramount for a more precise description of the location of electrons in large, relativistically described atoms. Let us explain the connection and future plans in more detail. Frank, Merz, Siedentop, and Simon \cite{Franketal2020P} consider large atoms by accounting for relativistic effects close to the nucleus using the so-called Chandrasekhar operator. The main result of \cite{Franketal2020P} is that, in the limit where the number of electrons and the nuclear charge goes to infinity, the ground state density, i.e., the probability density of finding one electron, converges to the so-called hydrogenic density, i.e., the sum of absolute squares of the eigenfunctions of $\sqrt{-\Delta+1}-1-\kappa/|x|$, where $\kappa$ is the (rescaled) interaction strength between the nucleus and the electrons. In fact, \cite{Franketal2020P} shows a finer result stating that the ground state density of electrons with prescribed angular momentum $\ell\in\N_0$ converges to the sum of absolute squares of the eigenfunctions of the operator
\begin{align}
  \label{e.fBoperturbed}
  (-\tfrac{d^2}{dr^2}-\tfrac{2\zeta}{r}\,\tfrac{d}{dr}+1)^{\alpha/2}-1-\kappa r^{-\alpha} \quad \text{in}\ L^2(\R_+,r^{2\zeta}dr),
\end{align}
with $\alpha=1$ and $\zeta=\ell+1/2$. The operator in \eqref{e.fBoperturbed} is a bounded perturbation of~\eqref{e.fBo}, but has, unlike \eqref{e.fBo}, infinitely many negative eigenvalues\footnote{This is because of the summand $+1$ in $(-\tfrac{d^2}{dr^2}-\tfrac{2\zeta}{r}\,\tfrac{d}{dr}+1)^{\alpha/2}$, which breaks the homogeneity and necessitates a distinction between large and small frequencies, the latter being responsible for the emergence of eigenvalues.}.
Moreover, \cite{Franketal2020P} provides pointwise estimates for the hydrogenic density using the heat kernel estimates for the homogeneous operator $(-\Delta)^{\alpha/2}-\kappa/|x|^\alpha$ established  in \cite{Bogdanetal2019}.
 
In this connection, we note that the estimates in the present paper and in \cite{BogdanMerz2024H}, together with the arguments developed in \cite{Franketal2020P}, will lead to bounds for the sum of absolute squares of eigenfunctions \eqref{e.fBoperturbed}. In particular, we can show that for all $t>0$, there is $c_t>0$ such that
\begin{align}
  \label{eq:boundsumsofsquares}
  \sum_{n\geq1}|\phi_n(r)|^2
  \leq c_t r^{-2\eta}, \quad 0<r\leq t,
\end{align}
where $\eta$ is in the one-to-one correspondence to $\kappa$ given by \eqref{eq:defpsietazetaOLD} and 
$\{\phi_n\}_{n\in\N}$ are the eigenfunctions corresponding to the negative eigenvalues of the operator in \eqref{e.fBoperturbed}. We are optimistic that, for $\zeta$ and $\alpha$ fixed, the upper bound in \eqref{eq:boundsumsofsquares} is optimal. 
In fact, the short-distance upper bound for the hydrogenic density of $\sqrt{1-\Delta}-1-\kappa/|x|$, obtained in \cite{Franketal2020P}, is optimal due to the lower bound on the ground state of $\sqrt{1-\Delta}-1-\kappa/|x|$, which was proved by Jakubowski, Kaleta, and Szczypkowski in \cite[Theorem~1.3]{Jakubowskietal2023} by using lower bounds for the corresponding heat kernel in \cite[Theorem~5.1]{Jakubowskietal2024}. While upper heat kernel bounds for \eqref{e.fBoperturbed} follow from those in the present paper, \cite{BogdanMerz2024H}, and subordination, proving corresponding lower bounds will require different techniques.

\subsection{Structure of the paper}
\label{s:organization}

In Section~\ref{s:comparison3g}, we state further pointwise bounds and a 3G inequality for the subordinated Bessel heat kernel $p_\zeta^{(\alpha)}$.
In Section~\ref{s:integralana}, we prove that the function $h(r)=r^{-\eta}$ is invariant under $p_{\zeta,\eta}^{(\alpha)}$. We also state further integral equalities and inequalities involving $p_{\zeta,\eta}^{(\alpha)}$ and $h(r)$.
In Section~\ref{s:proofmainresultgen}, we prove Theorem~\ref{mainresultgen} and Corollary~\ref{blowupcor}.
Finally, in  the Appendix~\ref{s:constructionnegschrodperturbation}, we provide an abstract discussion of Schr\"odinger perturbations of transition densities, which may be of independent interest.

\subsubsection*{Acknowledgments}
K.M.~thanks Haruya Mizutani for his hospitality at Osaka University, where parts of this research were carried out.
We also thank Volker Bach, Rupert Frank, Wolfhard Hansen, Jacek Małecki, Haruya Mizutani, and Grzegorz Serafin for helpful discussions.

\section{Bounds for subordinated Bessel heat kernels}
\label{s:comparison3g}

\subsection{Pointwise bounds and explicit expressions}

For $s,r \in \mathbb{R}$, $t>0$, and $\alpha\in(0,2]$, let $p^{(\alpha)}(t,r,s)$ be the stable density on the real line, i.e., 
\begin{align}
  \label{eq:stabledensity}
  p^{(\alpha)}(t,r,s) 
  := \frac{1}{2\pi}\int_{-\infty}^\infty \me{-i(s-r)z} \me{-t|z|^\alpha}\,dz.
\end{align}
Note that
\begin{align*}
  p^{(2)}(t,r,s) = \frac{1}{\sqrt{4\pi t}}\exp\left(-\frac{(r-s)^2}{4t}\right)
\end{align*}
is the Gau\ss--Weierstra\ss\, kernel and, by \cite{BlumenthalGetoor1960}, 
\begin{align*}
  p^{(\alpha)}(t,r,s) \sim \frac{t}{(t^{1/\alpha} + |r-s|)^{1+\alpha}} \quad \text{for} \ \alpha\in(0,2).
\end{align*}

From \cite[Corollary~3.8]{GrzywnyTrojan2021} and \cite[Theorem~2.1]{BogdanMerz2023S}, we record the following sharp upper and lower bounds for $p_\zeta^{(\alpha)}(t,r,s)$. 

\begin{proposition}[{\cite[Theorem~2.1]{BogdanMerz2023S}}]
  \label{heatkernelalpha1subordinatedboundsfinal}
  Let $\zeta\in(-1/2,\infty)$. Then, there are $c,c'>0$ such that for all $r,s,t>0$,
  \begin{subequations}
    \label{eq:easybounds2}
    \begin{align}
      p_\zeta^{(2)}(t,r,s)
      \label{eq:easybounds2a}
      & \asymp_\zeta t^{-\frac12}\frac{\exp\left(-\frac{(r-s)^2}{c t}\right)}{(rs+t)^{\zeta}} \\
      \label{eq:easybounds2b}
      & \asymp_\zeta \left(1\wedge\frac{r}{t^{1/2}}\right)^{\zeta}\, \left(1\wedge\frac{s}{t^{1/2}}\right)^{\zeta}\,\left(\frac{1}{rs}\right)^{\zeta} \cdot t^{-\frac12} \cdot \exp\left(-\frac{(r-s)^2}{c' t}\right).
    \end{align}
  \end{subequations}\index{$\asymp$}Moreover, for all $\alpha\in(0,2)$ and all $r,s,t>0$,
  \begin{align}
    \label{eq:heatkernelalpha1weightedsubordinatedboundsfinal}
    p_\zeta^{(\alpha)}(t,r,s)
    & \sim_{\zeta,\alpha} \frac{t}{|r-s|^{1+\alpha}(r+s)^{2\zeta} + t^{\frac{1+\alpha}{\alpha}}(t^{\frac1\alpha}+r+s)^{2\zeta}} \\
    \label{eq:pzetafact}
    & \sim_{\zeta,\alpha} p^{(\alpha)}(t,r,s) (t^{1/\alpha} + r+s)^{-2\zeta}.
  \end{align}
\end{proposition}

Here and below the notation $\asymp_\zeta$ combines an upper bound and a lower bound similarly as $\sim_\zeta$, but the constants in exponential factors (e.g., the constants $c$ and $c'$ in \eqref{eq:easybounds2a}--\eqref{eq:easybounds2b}) may be different in the upper and the lower bounds. Note that we allow the constants in the exponential factors to depend on $\zeta$. Then, for instance, \eqref{eq:easybounds2a} is equivalent to the statement that there are positive $c_{j,\zeta}$, $j\in\{1,2,3,4\}$, with
\begin{align}
  \begin{split}
    c_{1,\zeta} t^{-\frac12}\frac{\exp\left(-\frac{(r-s)^2}{c_{2,\zeta} t}\right)}{(rs+t)^{\zeta}}
    \leq p_\zeta^{(2)}(t,r,s)
    \leq c_{3,\zeta} t^{-\frac12}\frac{\exp\left(-\frac{(r-s)^2}{c_{4,\zeta} t}\right)}{(rs+t)^{\zeta}}.
  \end{split}
\end{align}

For an explicit expressions for $p_\zeta^{(\alpha)}(t,r,s)$ in the physically important case $\alpha=1$, we refer, e.g., to Betancor, Harboure, Nowak, and Viviani \cite[p.~136]{Betancoretal2010}.

Below, we will also use the following convenient estimate.

\begin{lemma}
  \label{lem:pzetaest}
  Let $\zeta\in(-1/2,\infty)$ and $\alpha\in(0,2)$. Then, for every $\sigma,\lambda \in[0,2\zeta+1]$ such that $\sigma+\lambda \in[0,2\zeta+1]$, we have
  \begin{align}
    \label{eq:pzetaest}
    p_\zeta^{(\alpha)}(\tau,z,s) \lesssim  \frac{\tau^{1/\alpha}+z+s}{\tau^{1/\alpha}+|z-s|}   \tau^{-\lambda/\alpha} s^{-\sigma} z^{\sigma+\lambda-2\zeta-1}.
  \end{align}
  In particular if $z<s/2$ or $s  \lesssim \tau^{1/\alpha} \land z $, then
  \begin{align}
    \label{eq:pzetaest2}
    p_\zeta^{(\alpha)}(\tau,z,s) \lesssim  \tau^{-\lambda/\alpha} s^{-\sigma} z^{\sigma+\lambda-2\zeta-1}.
  \end{align}
\end{lemma}

\begin{proof}
  Since $0<\sigma\leq 2\zeta+1$, by \eqref{eq:pzetafact} and \eqref{eq:stabledensity}, we have
  \begin{align*}
    p_\zeta^{(\alpha)}(\tau,z,s) 
    & \sim \frac{\tau}{(\tau^{1/\alpha}+|z-s|)^{1+\alpha}} (\tau^{1/\alpha}+z+s)^{-2\zeta}  \\
    & = \frac{\tau(\tau^{1/\alpha}+z+s)}{(\tau^{1/\alpha}+|z-s|)^{1+\alpha}}  (\tau^{1/\alpha}+z+s)^{-\lambda}(\tau^{1/\alpha}+z+s)^{-\sigma}(\tau^{1/\alpha}+z+s)^{\sigma+\lambda-2\zeta-1} \\
    & \le \frac{\tau^{1/\alpha}+z+s}{\tau^{1/\alpha}+|z-s|}  \tau^{-\lambda/\alpha} s^{-\sigma} z^{\sigma+\lambda-2\zeta-1}.
  \end{align*}
  For $z<s/2$ or $s/2 \le \tau^{1/\alpha} \land z$, we have $\frac{\tau^{1/\alpha}+z+s}{\tau^{1/\alpha}+|z-s|} \lesssim 1$ and consequently $p_\zeta^{(\alpha)}(\tau,z,s) \lesssim \tau^{-\lambda/\alpha} s^{-\sigma} z^{\sigma+\lambda-2\zeta-1}$.
\end{proof}

\subsection{Comparability results for $p_\zeta^{(\alpha)}$}

The following lemma is proved in \cite{BogdanMerz2023S} and allows to compare two kernels $p_\zeta^{(\alpha)}$ at different times and positions with each other.

\begin{lemma}[{\cite[Theorem~4.1]{BogdanMerz2023S}}]
  \label{comparablealpha}
  Let $\zeta\in(-1/2,\infty)$ and $\alpha\in(0,2]$.
  \begin{enumerate}
  \item Let $z,s>0$, $0<C\leq 1$, and $\tau\in[C,C^{-1}]$. Then,  there exist $c_j=c_j(\zeta,C)\in(0,\infty)$, $j\in\{1,2\}$, such that
    \begin{align}
      \label{eq:comparablealpha1}
      \begin{split}
        p_\zeta^{(2)}(1,c_1 z,c_1 s) & \lesssim_{C,\zeta} p_\zeta^{(2)}(\tau,z,s) \lesssim_{C,\zeta} p_\zeta^{(2)}(1,c_2 z,c_2 s), \\
        p_\zeta^{(\alpha)}(\tau,z,s) & \sim_{C,\zeta,\alpha} p_\zeta^{(\alpha)}(1,z,s), \quad \alpha<2.
      \end{split}
    \end{align}

  \item Let $C,\tau>0$ and $0<z\leq s/2<\infty$. Then,  there is $c=c(\zeta,C)\in(0,\infty)$ with
    \begin{subequations}
      \label{eq:comparablealpha2}
      \begin{align}
        \label{eq:comparablealpha2a}
        \begin{split}
          p_\zeta^{(\alpha)}(\tau,z,s)
          & \lesssim_{C,\zeta,\alpha} p_\zeta^{(\alpha)}(\tau,c,c s)\one_{\{\tau>C\}} \\
          & \quad + \left(\tau^{-\frac12}\frac{\me{-cs^2/\tau}}{(\tau+s^2)^{\zeta}}\one_{\alpha=2} + \frac{\tau}{s^{2\zeta+1+\alpha}+\tau^{\frac{2\zeta+1+\alpha}{\alpha}}} \one_{\alpha\in(0,2)}\right)\one_{\{\tau<C\}},
        \end{split}
        \\
        \label{eq:comparablealpha2c}
        p_\zeta^{(\alpha)}(\tau,z,s)
        & \lesssim_{\zeta,\alpha} s^{-(2\zeta+1)}.
      \end{align}
    \end{subequations}

  \item Let $0<\tau\leq1$, $0<z\leq s/2$, and $s\geq C>0$. Then, there is $c=c(\zeta,C)\in(0,\infty)$ with
    \begin{align}
      \label{eq:comparablealpha3}
      p_\zeta^{(\alpha)}(\tau,z,s) \lesssim_{C,\zeta,\alpha} p_\zeta^{(\alpha)}(1,c,c s).
    \end{align}

  \item There is $c=c(\zeta)\in(0,\infty)$ such that for all $r,s>0$,
    \begin{align}
      \label{eq:comparablealpha4}
      \min\{p_\zeta^{(\alpha)}(1,1,r),p_\zeta^{(\alpha)}(1,1,s)\}
      \lesssim_{\zeta,\alpha} p_\zeta^{(\alpha)}(1,c r,c s).
    \end{align}

  \item Let $r,s,z,t>0$ with $|z-s|>|r-s|/2$. Then,  there is $c=c(\zeta)\in(0,\infty)$ with
    \begin{align}
      \label{eq:comparablealpha6}
      p_\zeta^{(\alpha)}(t/2,z,s) \lesssim_{\zeta,\alpha} p_\zeta^{(\alpha)}(t,cr,cs).
    \end{align}
  \end{enumerate}
\end{lemma}

The constants in exponential factors in Lemma~\ref{comparablealpha} may change from place to place when $\alpha=2$.

The following lemma describes the concentration of mass of $p_\zeta^{(\alpha)}$ at the origin. It is similar to \cite[p.~10]{Bogdanetal2019}.
\begin{lemma}
  \label{posmassheatkernel}
  Let $\zeta\in(-1/2,\infty)$ and $\alpha\in(0,2]$. Then, there is $c=c_{\zeta,\alpha}\in(0,1)$ such that for any $R>1$,
  \begin{align}
    \label{eq:posmassheatkernel}
    c < \int_0^{2R} dw\, w^{2\zeta} \, p_\zeta^{(\alpha)}(\tau,z,w) \leq 1, \quad \tau\in(0,1],\ z<R.
  \end{align}
\end{lemma}

\begin{proof}
  The upper bound follows from the normalization of $p_\zeta^{(\alpha)}(\tau,z,w)$.
  We now prove the lower bound in \eqref{eq:posmassheatkernel}.
  First, consider $\alpha \in(0,2)$. Let $x=z \vee \tau^{1/\alpha}$. Then, for $w \in (x,x+\tau^{1/\alpha}) \subset (0,2R)$, we have $\tau^{1/\alpha} + |z-w| \sim \tau^{1/\alpha}$  and $\tau^{1/\alpha}+z+w \sim w \sim x$. Hence, by \eqref{eq:stabledensity}  and \eqref{eq:pzetafact}, we get
   \begin{align*}
      & \int_0^{2R} dw\, w^{2\zeta} \, p_{\zeta}^{(\alpha)}(\tau,z,w) \gtrsim \int_x^{x+\tau^{1/\alpha}} dw\, \frac{\tau}{(\tau^{1/\alpha} + |z-w|)^{1+\alpha}} (\tau^{1/\alpha}+z+w)^{-2\zeta} w^{2\zeta}\\
      & \quad \gtrsim \int_{x}^{x+\tau^{1/\alpha}} dw\, \tau^{-1/\alpha} = 1.
  \end{align*}
  Applying \eqref{eq:easybounds2a}, by the same proof, we get the result for $\alpha=2$.
  \end{proof}

\subsection{A 3G inequality for $p_\zeta^{(\alpha)}$ with $\alpha\in(0,2)$}

For the proof of the continuity of the heat kernel for all $\alpha\in(0,2)$ and $\eta\in(-\alpha,(2\zeta+1-\alpha)/2]$ in Theorem~\ref{mainresultgen}, we will use the following 3G inequality for $p_\zeta^{(\alpha)}(t,r,s)$. 
\begin{lemma}[3G inequality]
  \label{lem:3Gineq}
  Let $\zeta\in(-1/2,\infty)$ and $\alpha\in(0,2)$. For $t,\tau,t,s,r,z>0$, 
\begin{align}
\frac{p_\zeta^{(\alpha)}(t,r,z) p_\zeta^{(\alpha)}(\tau,z,s)}{p^{(\alpha)}(t+\tau,r,s)} \lesssim  \frac{p_\zeta^{(\alpha)}(t,r,z)}{(\tau^{1/\alpha}+z+s)^{2\zeta}} + \frac{p_\zeta^{(\alpha)}(\tau,z,s)}{(t^{1/\alpha}+z+r)^{2\zeta}}.
\end{align}
\begin{proof}
It suffices to use \eqref{eq:pzetafact} and the $3G$ inequality in \cite[(7)--(9)]{BogdanJakubowski2007}
\begin{align*}
p^{(\alpha)}(t,r,z) p^{(\alpha)}(\tau,z,s) \lesssim p^{(\alpha)}(t+\tau,r,s) \left[p^{(\alpha)}(t,r,z) + p^{(\alpha)}(\tau,z,s) \right].
\end{align*}
\end{proof}
\end{lemma}

There is a similar 3G inequality in \cite[Theorem~3.1]{BogdanMerz2023S}. However, the present 3G inequality in Lemma~\ref{lem:3Gineq} and that in \cite[Theorem~3.1]{BogdanMerz2023S} do not imply each other.

\section{Integral analysis of $p_{\zeta,\eta}^{(\alpha)}$}
\label{s:integralana}

The \textit{integral analysis} in this section is a potential-theoretic study based on suitable supermedian functions. In particular, we show that the function $s^{-\eta}$ is invariant for the transition density $p^{(\alpha)}_{\zeta,\eta}$,
see \eqref{eq:heatkernelharmonic1}.

The main results of this section are Theorems \ref{heatkernelharmonic} and \ref{jakthm24}, which are key elements in the proof of Theorem \ref{mainresultgen}. Let us now sketch the main ideas behind the integral analysis presented in this section. The method, developed in \cite{Bogdanetal2019} and further improved in \cite{JakubowskiMaciocha2023, Jakubowskietal2024}, builds on the results of \cite[Theorem 1]{Bogdanetal2016}. 

For flexibility in what follows, we introduce an extra parameter $\beta$ acting as a proxy for $\eta$. To
prove Theorem \ref{heatkernelharmonic}, we essentially show that $\Psi_\zeta(\beta)(\cl_\zeta)^{-\alpha/2} s^{-\alpha-\beta} = s^{-\beta}$, see \eqref{eq:kappainv}. 
As a consequence, we obtain \eqref{eq:heatkernelharmonic2} for \(\eta = 0\); see Lemma \ref{boglem33}. Next, by applying Duhamel's formula and Lemma \ref{boglem33}, we derive \eqref{eq:tpIntpos}, which is almost \eqref{eq:heatkernelharmonic2}. To fully establish \eqref{eq:heatkernelharmonic2}, we prove the convergence of the integrals in \eqref{eq:tpIntpos}; see Corollary \ref{bogcor35} and Lemma \ref{lem:iiptconv}.

We point out Proposition \ref{bogprop32} as a crucial step in our integral analysis, in which we estimate the mass of \(p_{\zeta,\eta}^{(\alpha)}\). It is the key to proving the estimates stated in Corollary~\ref{bogcor38}; the case \(\eta<0\) being simpler since the convergence of the integrals follows from the inequality \(p_{\zeta,\eta}^{(\alpha)} \le p_{\zeta}^{(\alpha)}\).

The result stated in Theorems \ref{heatkernelharmonic} and \ref{jakthm24} may be viewed as follows. Denote $\cl_{\zeta,\eta}^{(\alpha)} := \cl_\zeta^{\alpha/2} - \Psi_\zeta(\eta)s^{-\alpha}$. If we let $t\to\infty$ in \eqref{eq:heatkernelharmonic2}, then we formally get 
\begin{align*}
  (\Psi_\zeta(\beta) -\Psi_\zeta(\eta)) \left(\cl_{\zeta,\eta}^{(\alpha)}\right)^{-1} s^{-\beta-\alpha}=  s^{-\beta},
\end{align*} 
which is another expression of 
$\cl_{\zeta,\eta}^{(\alpha)} s^{-\beta}= (\Psi_\zeta(\beta) -\Psi_\zeta(\eta))  s^{-\beta-\alpha}$. The latter right-hand side changes sign at $\beta=\eta$, so we may informally conclude that for the operator $\cl_{\zeta,\eta}^{(\alpha)}$, the function $s^{-\beta}$ is subharmonic, harmonic and superharmonic if $\beta <\eta$, $\beta=\eta$ and $\beta>\eta$, respectively. In fact, in Theorems~\ref{heatkernelharmonic} and~\ref{jakthm24}, we prove that $s^{-\beta}$ is supermedian for $p_{\zeta,\eta}^{(\alpha)}$ if $\beta \ge\eta$ and $h(s) = s^{-\eta}$ is even invariant for $p_{\zeta,\eta}^{(\alpha)}$, which is crucial to prove the lower bounds. 
As we shall see, the estimates of Bessel and subordinated Bessel heat kernels also play an important role in our development.

When proving ground state representations for \eqref{e.fBo} in \cite{BogdanMerz2024}, we arrived at the invariant function $h(r)$ by integrating $p_{\zeta}^{(\alpha)}(t,r,s)$ against suitable functions in space and time, as suggested by the road map in \cite{Bogdanetal2016}.
In the following, we call two parameters $\beta,\gamma\in\R$ admissible, whenever they satisfy $\gamma\in(-1,\infty)$ and $\beta\in(1,(2\zeta-\gamma)/\alpha)$.
For admissible $\beta,\gamma$, we recall from \cite[Section~3]{BogdanMerz2024} that the function\footnote{The function $h_{\beta,\gamma}$ is also well-defined for $\beta\in(0,(2\zeta-\gamma)/\alpha]$.}

\begin{align}
  \label{eq:defhbetagammaalphatransformed}
  \begin{split}
    h_{\beta,\gamma}(r)
    & := \int_0^\infty \frac{dt}{t}\, t^\beta \int_0^\infty ds\, s^\gamma p_\zeta^{(\alpha)}(t,r,s)
      = C^{(\alpha)}\left(\beta,\gamma,\zeta\right) r^{\alpha\beta+\gamma-2\zeta}, \quad r>0,
  \end{split}
\end{align}
is finite where
\begin{align}
  \label{eq:timeshiftconstant1resultalpha}
  C^{(\alpha)}(\beta,\gamma,\zeta)
  & := \frac{\Gamma(\beta)}{\Gamma(\frac{\alpha\beta}{2})} C\left(\frac{\alpha\beta}{2},\gamma,\zeta\right), \quad \text{with} \\
  \label{eq:timeshiftconstant1result}
  C(\beta,\gamma,\zeta)
  & := \frac{2^{-2\beta} \Gamma (\beta) \Gamma \left(\frac{1}{2} (\gamma+1)\right) \Gamma \left(\frac{1}{2} (2\zeta-2\beta-\gamma)\right)}{\Gamma \left(\frac{1}{2}(2\zeta-\gamma)\right) \Gamma\left(\frac{1}{2}(2\beta+\gamma+1)\right)}.
\end{align}
Moreover, the Fitzsimmons ratio (see \cite{Bogdanetal2016}, \cite{Fitzsimmons2000}) is
\begin{align}
  q_{\beta,\gamma}(r) := \frac{(\beta-1) h_{\beta-1,\gamma}(r)}{h_{\beta,\gamma}(r)}, \quad r>0.
\end{align}
In particular, we recover
\begin{align}
  \label{eq:hardypottransformedalpha}
  q(z) = \frac{(\overline\beta-1)h_{\overline\beta-1,\overline\gamma}(z)}{h_{\overline\beta,\overline\gamma}(z)}, \quad z>0
\end{align}
from \eqref{eq:defq}, whenever $\overline\beta,\overline\gamma$ are admissible, to wit, $\overline\gamma\in(-1,\infty)$ and $\overline\beta\in(1,(2\zeta-\overline\gamma)/\alpha)$, and are chosen so that $\alpha\overline\beta+\overline\gamma-2\zeta=-\eta$. This is, e.g., the case when $\overline\beta=(2\zeta-\overline\gamma-\eta)/\alpha$ with $\zeta\in(-1/2,\infty)$, $\alpha\in(0,2]\cap(0,2\zeta+1)$, $\eta\in[0,\frac{2\zeta+1-\alpha}{2}]$, and $\overline\gamma\in(-1,\zeta-(\alpha+1)/2)$. 
We will use the function $h_{\beta,\gamma}$ and its variant $h_{\beta,\gamma}^{(+)}$ in Section~\ref{ss:integralanalysisnegativeeta} to obtain invariant functions for $p_{\zeta,\eta}^{(\alpha)}$. When doing so, the additional flexibility provided by $\overline\beta$ and $\overline\gamma$ will be helpful.

\subsection{Continuity}

For $\alpha\in(0,2)$, we will use the following lemma to prove that $p_{\zeta,\eta}^{(\alpha)}$ is a strongly continuous contraction semigroup and that $p_{\zeta,\eta}^{(\alpha)}(t,r,s)$ is continuous in $t,r,s>0$.

\begin{lemma}
 \label{integralana1}
 Let $\zeta\in(-1/2,\infty)$, $\alpha\in(0,2]$ and $\delta \in (0,2\zeta+1)$. Then,
  \begin{align}
    \label{eq:hbetagammataurtransformed}
    \begin{split}
      \int_0^\infty ds\, s^{2\zeta} p_\zeta^{(\alpha)}(\tau,r,s) s^{-\delta} \sim r^{-\delta}\one_{\{\tau<r^\alpha\}} + \tau^{-\delta/\alpha}\one_{\{\tau>r^\alpha\}} =r^{-\delta} \land \tau^{-\delta/\alpha}
    \end{split}
  \end{align}
  and
  \begin{align}
    \label{eq:hbetagammataurintegratedtransformed}
    \begin{split}
       \int_0^t d\tau\, \int_0^\infty ds\, s^{2\zeta} p_\zeta^{(\alpha)}(\tau,r,s) s^{-\delta} \; \sim \;
        \begin{cases}
          t (r^\alpha \vee t)^{-\delta/\alpha} & \quad \text{for}\ \delta<\alpha \\
          \log(1+\frac{t}{r^\alpha}) & \quad \text{for}\ \delta=\alpha \\
          r^{-\delta}  (t\wedge r^\alpha)  & \quad \text{for}\ \delta>\alpha
        \end{cases}.
    \end{split}
  \end{align}
\end{lemma}

\begin{proof}
  We only prove this lemma for $\alpha\in(0,2)$; the proof for $\alpha=2$ is analogous and, in fact, easier using the explicit expression for $p_\zeta^{(2)}(t,r,s)$.
  Formula \eqref{eq:hbetagammataurintegratedtransformed} follows from integrating \eqref{eq:hbetagammataurtransformed}. Hence, we only need to prove \eqref{eq:hbetagammataurtransformed}.  By scaling we may and do assume $r=1$. We first show that for $\tau>1$,
  \begin{align}
    \label{eq1:integralana1}
    \int_0^\infty ds\, s^{2\zeta}\, p_\zeta^{(\alpha)}(\tau,1,s)s^{-\delta} \sim \tau^{-\delta/\alpha}.
  \end{align}
  Indeed, since $\tau>1$, we have $|s-1|< s+1 <3\tau^{1/\alpha}$ for $s<2\tau^{1/\alpha}$. Hence, by \eqref{eq:pzetafact},
  \begin{align*}
    \int_0^{2\tau^{1/\alpha}} ds\, s^{2\zeta}\, p_\zeta^{(\alpha)}(\tau,1,s) s^{-\delta} \sim \tau^{-(2\zeta+1)/\alpha}\int_0^{2\tau^{1/\alpha}} ds \,  s^{2\zeta-\delta} \sim \tau^{-\delta/\alpha}.
  \end{align*}
  For $s>2\tau^{1/\alpha}$, we have $s+1 \sim |s-1| \sim s$. Thus,
  \begin{align*}
    \int_{2\tau^{1/\alpha}}^\infty ds\, s^{2\zeta} p_\zeta^{(\alpha)}(\tau,1,s) s^{-\delta} \sim \tau \int_{2\tau^{1/\alpha}}^\infty ds\,  s^{-\delta-1-\alpha} \sim \tau^{-\delta/\alpha},
  \end{align*}
  which yields \eqref{eq1:integralana1}.
  Now, suppose $\tau\le 1$. Note that $s+1+\tau^{1/\alpha} \sim s \vee 1$. Additionally, for $s<1/2$, $p^{(\alpha)}(\tau,1,s) \sim \tau$. Hence, by \eqref{eq:pzetafact},
  \begin{align*}
    \int_{0}^{\infty} ds\, s^{2\zeta}\, p_\zeta^{(\alpha)}(\tau,1,s) s^{-\delta} \lesssim \int_{1/2}^{\infty} ds\, p^{(\alpha)}(\tau,1,s) +\tau \int_0^{1/2} ds\, s^{2\zeta-\delta} \lesssim 1,
  \end{align*}
  where we used that $p^{(\alpha)}(\tau,1,s)$ is the density of a probability distribution.
  Finally, for $1<s<1+\tau^{1/\alpha}$, $p^{(\alpha)}(\tau,1,s) \sim \tau^{-1/\alpha}$. Thus,
  \begin{align*}
    \int_{0}^{\infty} ds\, s^{2\zeta}\, p_\zeta^{(\alpha)}(\tau,1,s) s^{-\delta} \gtrsim \int_{1}^{1+\tau^{1/\alpha}} ds\, \tau^{-1/\alpha} = 1,
  \end{align*}
  which, together with \eqref{eq1:integralana1}, yields \eqref{eq:hbetagammataurtransformed}.
\end{proof}

We now prove that $t\mapsto p_{\zeta,\eta}^{(\alpha)}(t,\cdot,\cdot)$ is a strongly continuous contraction semigroup. This property will also be important to prove the pointwise continuity of $p_{\zeta,\eta}^{(\alpha)}(t,r,s)$ with respect to all $r,s,t>0$.

\begin{proposition}
  \label{bogprop24}
  Let $\zeta\in(-1/2,\infty)$, $\alpha\in(0,2]$, and $\eta\in(-\alpha,\frac{2\zeta+1-\alpha}{2}]$.
  Then $\{p_{\zeta,\eta}^{(\alpha)}(t,\cdot,\cdot)\}_{t\geq0}$ is a strongly continuous contraction semigroup on $L^2(\R_+,r^{2\zeta}dr)$.
\end{proposition}

\begin{proof}
  For $\alpha=2$, Proposition~\ref{bogprop24} follows, e.g., from the explicit expression for $p_{\zeta,\eta}^{(2)}$ in~\eqref{eq:mainresultalpha2}.
  Thus, we consider $\alpha\in(0,2)$ in the following.
  Since we verified the semigroup property in Proposition~\ref{propertiesschrodheatkernel}, it suffices to verify the strong continuity and the contraction property.
  By the symmetry of $p_{\zeta,\eta}^{(\alpha)}$, the supermedian property \eqref{eq:supermediantransformedalphaintro}, and a Schur test, $p_{\zeta,\eta}^{(\alpha)}$ is a contraction on $L^2(\R_+,r^{2\zeta}dr)$ for every $t>0$.
  It remains to prove the strong continuity of $p_{\zeta,\eta}^{(\alpha)}(t,\cdot,\cdot)$. Since $C_c^\infty(\R_+)$ is dense in $L^2(\R_+,r^{2\zeta}dr)$ and $p_\zeta^{(\alpha)}(t,\cdot,\cdot)$ is strongly continuous (see Lemma~\ref{summarypropertiespzeta}), it follows from Duhamel's formula, that $p_{\zeta,\eta}^{(\alpha)}(t,\cdot,\cdot)$ is strongly continuous, if $\lim_{t\searrow0}\|T_t\phi\|_{L^2(\R_+,r^{2\zeta}dr)}=0$ holds for every nonnegative function $\phi\in C_c^\infty(\R_+)$, where $T_t$ is the integral operator acting as
  \begin{align}
    (T_t\phi)(r)
    = \int_0^t d\tau \int_0^\infty dz\, z^{2\zeta} p_\zeta^{(\alpha)}(\tau,r,z)q(z) \int_0^\infty ds\, s^{2\zeta} p_{\zeta,\eta}^{(\alpha)}(t-\tau,z,s)\phi(s).
  \end{align}
  By the definition of $p_{\zeta,\eta}^{(\alpha)}$, the integral kernel of $T_t$ is maximal for $\eta=(2\zeta+1-\alpha)/2$ (this is clear for $\eta\ge 0$ by \eqref{eq:feynmankactransformed}; for $\eta<0$, see Lemma~\ref{lem:1}). Thus, it suffices to consider this case. Since $\phi\in C_c^\infty(\R_+)$, we have $\phi\lesssim_\phi h^*$, with $h^*(r)=r^{-(2\zeta+1-\alpha)/2}$. By the supermedian property
  \eqref{eq:supermediantransformedalphaintro} and Lemma~\ref{integralana1} with $\delta =  \zeta+(1+\alpha)/2 \in (0,2\zeta+1)$,
  \begin{align}
    \begin{split}
      & \|T_t\phi\|_{L^2(\R_+,r^{2\zeta}dr)}^2
      \lesssim_\phi \|T_t h^*\|_{L^2(\R_+,r^{2\zeta}dr)}^2 \\
      & \quad \lesssim \int_0^\infty dr\, r^{2\zeta}\left(\int_0^t d\tau \int_0^\infty dz\, z^{2\zeta} p_\zeta^{(\alpha)}(\tau,r,z)q(z)h^*(z)\right)^2 \\
      & \quad \lesssim \int_0^\infty dr\, r^{-(\alpha+1)}(t\wedge r^\alpha)^2.
    \end{split}
  \end{align}
  Here we used $\alpha\beta+\gamma-2\zeta=-\zeta-(1+\alpha)/2<-\alpha$, which allows us to use the third line in \eqref{eq:hbetagammataurintegratedtransformed}.
  The conclusion follows from an application of the dominated convergence theorem.
\end{proof}

\subsection{The case $\eta\in(0,(2\zeta+1-\alpha)/2]$}\label{s.pc}

One of the main ingredients in the proof of heat kernel estimates is the invariance of the ground state expressed in terms of the heat kernel. 
This is the content of the following important theorem.

\begin{theorem}
  \label{heatkernelharmonic}
  Let $\zeta\in(-1/2,\infty)$, $\alpha\in(0,2]\cap(0,2\zeta+1)$, $\eta\in(0,(2\zeta+1-\alpha)/2]$, and $r,t>0$. Then,
  \begin{align}
    \label{eq:heatkernelharmonic1}
    \int_0^\infty ds\, s^{2\zeta} p_{\zeta,\eta}^{(\alpha)}(t,r,s)s^{-\eta}
    = r^{-\eta}.
  \end{align}
  Moreover, for $\beta\in[0,2\zeta+1-\alpha-\eta)$, we have
  \begin{align}
    \label{eq:heatkernelharmonic2}
    \begin{split}
      & \int_0^\infty ds\, s^{2\zeta} p_{\zeta,\eta}^{(\alpha)}(t,r,s)s^{-\beta}\,ds \\
      & \quad = r^{-\beta} + \left(1-\frac{\Psi_\zeta(\beta)}{\Psi_\zeta(\eta)}\right)\int_0^t d\tau\int_0^\infty ds\, s^{2\zeta} p_{\zeta,\eta}^{(\alpha)}(\tau,r,s)\, q(s) \, s^{-\beta}.
    \end{split}
  \end{align}
\end{theorem}

For $\alpha=2$, Theorem~\ref{heatkernelharmonic} follows, e.g., from the explicit expression for $p_{\zeta,\eta}^{(2)}$ in \eqref{eq:mainresultalpha2}. Thus, we may freely assume $\alpha\in(0,2)$ in the following, although all lemmas, propositions, and theorems in this subsection also hold true for $\alpha=2$.

To prove Theorem~\ref{heatkernelharmonic}, we need several auxiliary statements, which follow almost exactly as in \cite{Bogdanetal2019}.
In Proposition~\ref{bogprop32}, we estimate the mass of $p_{\zeta,\eta}^{(\alpha)}$ by $1+(r/t^{1/\alpha})^{-\eta}$.
Auxiliary bounds to that end are contained in Lemmas~\ref{bogprop32aux} and \ref{chapmankolmogorovinequality}. Lemmas~\ref{boglem34} and~\ref{boglem33} contain the heart of the proof of Theorem~\ref{heatkernelharmonic}. Here, we integrate the zeroth and $n$-th summands of the perturbation series \eqref{eq:feynmankactransformed} against monomials. We note the role of the auxiliary parameter $\beta$, in particular \eqref{eq:heatkernelharmonic2}, in proving \eqref{eq:heatkernelharmonic1}. 
We also note that \eqref{eq:heatkernelharmonic2} plays an important role to prove the upper bound in \eqref{eq:mainresultgen} in Theorem~\ref{mainresultgen} for small distances to the origin; see Lemmas~\ref{boglem44} and \ref{boglem46}.
We apply Lemma~\ref{boglem34} in Corollary~\ref{bogcor35} to obtain a weaker version of Theorem~\ref{heatkernelharmonic}, which will be useful to prove the upper bounds for $p_{\zeta,\eta}^{(\alpha)}$; see, in particular, Lemma~\ref{boglem46}.
These results together with the finiteness of certain space and space-time integrals over $p_{\zeta,\eta}^{(\alpha)}$ in Lemma~\ref{lem:iiptconv} eventually enable us to conclude the proof of Theorem~\ref{heatkernelharmonic}.
In Corollary~\ref{bogcor38}, we collect some convenient consequences of Theorem~\ref{heatkernelharmonic} that will be used to prove the bounds for $p_{\zeta,\eta}^{(\alpha)}$ when $\eta>0$.
We finish this section with Lemma~\ref{pzetaetafinite}, where we show the finiteness of $p_{\zeta,\eta}^{(\alpha)}(t,r,s)>0$ for all $r,s,t>0$.

\smallskip
Let us now give the details and start with the following auxiliary bound.
\begin{lemma}
  \label{bogprop32aux}
  For $\zeta\in(-1/2,\infty)$, $\alpha\in(0,2]$, $\eta\in(0,(2\zeta+1-\alpha)/2]$, and $r,t,R>0$,
  \begin{align}
    \label{eq:bogprop32aux}
    \int_0^1 d\tau\int_0^R dz\, z^{2\zeta} p_{\zeta,\eta}^{(\alpha)}(\tau,r,z)q(z)
    \leq \frac{(2R)^{\eta}}{c} r^{-\eta},
  \end{align}
  where $c=c_R$ is the constant appearing in Lemma~\ref{posmassheatkernel}.
\end{lemma}

\begin{proof}
  By Lemma~\ref{posmassheatkernel}, there is $c=c_R\in(0,1)$ such that
  \begin{align}
    c \leq \int_0^{2R} p_\zeta^{(\alpha)}(1-\tau,z,s)s^{2\zeta}\,ds \leq 1,
    \quad 0<z<R,\,\tau\in(0,1].
  \end{align}
  Thus, by Duhamel \eqref{eq:duhamelclassictransformed} and the supermedian property \eqref{eq:supermediantransformedalphaintro}, we have
  \begin{align}
    \begin{split}
      & \int_0^1 d\tau\int_0^R dz\, z^{2\zeta} p_{\zeta,\eta}^{(\alpha)}(\tau,r,z)q(z) \\
      & \quad \leq \frac1c \int_0^{2R} ds\, s^{2\zeta}\int_0^1 d\tau\int_0^R dz\, z^{2\zeta} p_{\zeta,\eta}^{(\alpha)}(\tau,r,z)q(z)p_\zeta^{(\alpha)}(1-\tau,z,s) \\
      & \quad \leq \frac1c \int_0^{2R}ds\, s^{2\zeta} p_{\zeta,\eta}^{(\alpha)}(1,r,s)
      \leq \frac{(2R)^{\eta}}{c}\int_0^{2R} ds\, s^{2\zeta} p_{\zeta,\eta}^{(\alpha)}(1,r,s) s^{-\eta}
      \leq \frac{(2R)^{\eta}}{c} r^{-\eta}.
    \end{split}
  \end{align}
  This concludes the proof.
\end{proof}

We record the following Chapman--Kolmogorov-type inequality for partial sums of the perturbation series \eqref{eq:feynmankactransformed}.

\begin{lemma}
  \label{chapmankolmogorovinequality}
  Let $\zeta\in(-1/2,\infty)$, $\alpha\in(0,2]$, $\eta\in(0,(2\zeta+1-\alpha)/2]$, and $r,s,t>0$.
  For $n\in\N_0=\{0,1,2,...\}$, let
  \begin{align}
    P_t^{(n,D)}(r,s) := \sum_{k=0}^n p_t^{(k,D)}(r,s),
  \end{align}
  with $p_t^{(n,D)}(r,s)$ defined in \eqref{eq:feynmankactransformed}. Then, for all $0<\tau<t$ and $n\in\N_0$,
  \begin{align}
    \label{eq:chapmankolmogorovinequality}
    \int_0^\infty dz\, z^{2\zeta} P_{\tau}^{(n,D)}(r,z)p_\zeta^{(\alpha)}(t-\tau,z,s)
    \leq P_t^{(n,D)}(r,s).
  \end{align}
\end{lemma}

\begin{proof}
  By scaling, it suffices to consider $0<\tau<1=t$.
  For $n=0$, the statement is just the Chapman--Kolmogorov equality, so let $n\in\N$.
  Using the definition of $P_t^{(n,D)}$ and $p_t^{(n,D)}$, we get
  \begin{align*}
    \begin{split}
      & \int_0^\infty dz\, z^{2\zeta} P_{\tau}^{(n,D)}(r,z) p_\zeta^{(\alpha)}(1-\tau,z,s) \\
      & \quad = \int_0^\infty dz\, z^{2\zeta} p_{\tau}^{(0,D)}(r,z) p_\zeta^{(\alpha)}(1-\tau,z,s) \\
      & \quad + \sum_{k=1}^n \int_0^\infty dz\, z^{2\zeta} \int_0^\tau d\rho \int_0^\infty dw\, w^{2\zeta} p_{\rho}^{(k-1,D)}(r,w) q(w)p_{\zeta}^{(\alpha)}(\tau-\rho,w,z) p_\zeta^{(\alpha)}(1-\tau,z,s) \\
      & \quad = p_\zeta^{(\alpha)}(1,r,s)
      + \sum_{k=1}^n\int_0^\tau d\rho \int_0^\infty dw\, w^{2\zeta} p_{\rho}^{(k-1,D)}(r,w) q(w)p_{\zeta}^{(\alpha)}(1-\rho,w,s) \\
      & \quad \leq p_\zeta^{(\alpha)}(1,r,s)
        + \sum_{k=1}^n\int_0^1 d\rho \int_0^\infty dw\, w^{2\zeta} p_{\rho}^{(k-1,D)}(r,w) q(w)p_{\zeta}^{(\alpha)}(1-\rho,w,s) \\
      & \quad = P_1^{(n,D)}(r,s).
    \end{split}
  \end{align*}
  This concludes the proof.
\end{proof}

\begin{proposition}
  \label{bogprop32}
  Let $\zeta\in(-1/2,\infty)$, $\alpha\in(0,2]$, $\eta\in(0,(2\zeta+1-\alpha)/2]$, and $r,s,t>0$.
  Then  $p_{\zeta,\eta}^{(\alpha)}(t,r,s)<\infty$ for all $r>0$ and almost all $s>0$, and there is a constant $M\geq1$ such that
  \begin{align}
    \label{eq:bogprop32}
    \int_0^\infty p_{\zeta,\eta}^{(\alpha)}(t,r,s)s^{2\zeta}\,ds
    \leq M[1+(t^{-1/\alpha}r)^{-\eta}].
  \end{align}
\end{proposition}

\begin{proof}
  By the scaling \eqref{eq:scalingalpha}, it suffices to consider $t=1$.
  For $n\in\N$, we consider the $n$-th partial sum $P_t^{(n,D)}(r,s) = \sum_{k=0}^n p_t^{(n,D)}(r,s)$, with $p_t^{(n,D)}(r,s)$ defined in \eqref{eq:feynmankactransformed}. Then, by Duhamel's formula \eqref{eq:duhamelclassictransformed} and $P_1^{(n-1,D)}(r,s)\leq P_1^{(n,D)}(r,s)\leq p_{\zeta,\eta}^{(\alpha)}(1,r,s)$,
  \begin{align}
    \label{eq:bogprop32aux5}
    \begin{split}
      P_1^{(n,D)}(r,s)
      & \leq p_\zeta^{(\alpha)}(1,r,s)
        + \int_0^1 d\tau\int_0^R dz\, z^{2\zeta} p_{\zeta,\eta}^{(\alpha)}(\tau,r,z)(r,z)q(z)p_\zeta^{(\alpha)}(1-\tau,z,s) \\
      & \quad + \frac{\Psi_\zeta(\eta)}{R^\alpha}\int_0^1 d\tau\int_R^\infty dz\, z^{2\zeta} P_{\tau}^{(n,D)}(r,z)p_\zeta^{(\alpha)}(1-\tau,z,s).
    \end{split}
  \end{align}
  By Lemma~\ref{chapmankolmogorovinequality},
  \begin{align}
    \label{eq:bogprop32aux6}
    \int_0^1 d\tau\int_R^\infty dz\, z^{2\zeta} P_{\tau}^{(n,D)}(\tau,r,z)p_\zeta^{(\alpha)}(1-\tau,z,s) \leq P_1^{(n,D)}(r,s).
  \end{align}
  By induction and the heat kernel bounds in Proposition~\ref{heatkernelalpha1subordinatedboundsfinal}, for each $n\in\N_0$, we have $p_1^{(n,D)}(r,s)<\infty$. Thus, for $R\geq(2\Psi_\zeta(\eta))^{1/\alpha}\vee2$ and all $r,s>0$, we have, by \eqref{eq:bogprop32aux5},
  \begin{align}
    \label{eq:bogprop32aux2}
    P_{1}^{(n,D)}(r,s) \leq 2p_\zeta^{(\alpha)}(1,r,s) + 2\int_0^1 d\tau\int_0^R dz\, z^{2\zeta} p_{\zeta,\eta}^{(\alpha)}(\tau,r,z)q(z)p_\zeta^{(\alpha)}(1-\tau,z,s).
  \end{align}
  Integrating both sides of \eqref{eq:bogprop32aux2} against $s^{2\zeta}\,ds$ and using $\int_0^\infty p_\zeta^{(\alpha)}(1,r,s)s^{2\zeta}\,ds=1$ for all $r>0$ by \eqref{eq:normalizedalpha}, we get, upon letting $n\to\infty$ on the left-hand side of \eqref{eq:bogprop32aux2},
  \begin{align}
    \label{eq:bogprop32aux4}
    \int_0^\infty ds\, s^{2\zeta} p_{\zeta,\eta}^{(\alpha)}(1,r,s)
    % \sim
    \lesssim 1 + \int_0^1 d\tau\int_0^R dz\, z^{2\zeta} p_{\zeta,\eta}^{(\alpha)}(\tau,r,z)q(z)
    \lesssim 1 + R^\eta r^{-\eta}.
  \end{align}
  Here, we used Lemma~\ref{bogprop32aux} to obtain the final inequality in \eqref{eq:bogprop32aux4}.
\end{proof}

\begin{lemma}  
  \label{boglem34}
  Let $\zeta\in(-1/2,\infty)$, $\alpha\in(0,2]$, $\eta\in[0,(2\zeta+1-\alpha)/2]$,
  $\beta\in(0,2\zeta+1-\alpha)$, $n\in\N_0$, and $r>0$.
  Then,
  \begin{align}
    \label{eq:boglem342inf}
    \int_0^\infty d\tau \int_0^\infty ds\, s^{2\zeta} p_\tau^{(n,D)}(r,s) s^{-\beta-\alpha}
    = \frac{\Psi_\zeta(\eta)^n}{\Psi_\zeta(\beta)^{n+1}} r^{-\beta}.
  \end{align}
\end{lemma}

\begin{proof}
  We prove the statement using induction. By \eqref{eq:defhbetagammaalphatransformed}, we get
  \begin{align}
    \int_0^\infty dt \,\int_0^\infty ds\, s^{2\zeta}  p_\zeta^{(\alpha)}(t,r,s)s^{-\alpha-\beta}
    & = C^{(\alpha)}\left(1,2\zeta-\alpha-\beta,\zeta\right) r^{-\beta} \notag \\
    & = \frac{2^{-\alpha}  \Gamma \left(\frac{2\zeta+1-\alpha-\beta}{2} \right) \Gamma \left(\frac{\beta}{2}\right)}{\Gamma \left(\frac{\alpha+\beta}{2}\right) \Gamma\left(\frac{2\zeta+1-\beta}{2}\right)} r^{-\beta}  = \frac{1}{\Psi_\zeta(\beta)} r^{-\beta},  \label{eq:kappainv} 
  \end{align}
  which gives \eqref{eq:boglem342inf} for $n=0$. Now, suppose \eqref{eq:boglem342inf} holds for some $n\in\N_0$. Then, by the definition of $p_\tau^{(n,D)}(r,s)$ in the formula \eqref{eq:feynmankactransformed}, a time-translation $\tau\mapsto\tau+t$, \eqref{eq:kappainv}, and the induction hypothesis,
  \begin{align}
    \begin{split}
      & \int_0^\infty d\tau \int_0^\infty ds\, s^{2\zeta} p_\tau^{(n+1,D)}(r,s) s^{-\beta-\alpha} \\
      & \quad = \int_0^\infty d\tau \int_0^\infty ds\, s^{2\zeta} \int_0^\tau dt \int_0^\infty dz\, z^{2\zeta} p_{t}^{(n,D)}(r,z)q(z)p_\zeta^{(\alpha)}(\tau-t,z,s) s^{-\beta-\alpha} \\
      & \quad = \int_0^\infty dt \int_0^\infty dz\, z^{2\zeta} p_{t}^{(n,D)}(r,z)q(z) \int_0^\infty d\tau\, \int_0^\infty ds\, s^{2\zeta} p_\zeta^{(\alpha)}(\tau,z,s)s^{-\beta-\alpha} \\
      & \quad = \int_0^\infty d\tau' \int_0^\infty dz\, z^{2\zeta} p_{\tau'}^{(n,D)}(r,z)\Psi_\zeta(\eta)z^{-\alpha} \cdot\frac{z^{-\beta}}{\Psi_\zeta(\beta)} =\frac{\Psi_\zeta(\eta)^{n+1}}{\Psi_\zeta(\beta)^{n+2}}r^{-\beta}.
    \end{split}
  \end{align}
  This concludes the proof.
\end{proof}

\begin{corollary}
  \label{bogcor35}
  Let $\zeta\in(-1/2,\infty)$, $\alpha\in(0,2]\cap(0,2\zeta+1)$, $\eta\in(0,(2\zeta+1-\alpha)/2]$, and $\beta\in(\eta,2\zeta+1-\alpha-\eta)$. Then, for all $r>0$,
  \begin{align} 
    \label{eq:bogcor35inf}
    \int_0^\infty d\tau \int_0^\infty ds\, s^{2\zeta} p_{\zeta,\eta}^{(\alpha)}(\tau,r,s)s^{-\beta-\alpha} 
    = \frac{r^{-\beta}}{\Psi_\zeta(\beta)-\Psi_\zeta(\eta)}.
  \end{align}
\end{corollary}

\begin{proof}
  Since $\Psi_\zeta(\beta)>\Psi_\zeta(\eta)$ (by the symmetry of $\Psi_\zeta(\eta)$ about $\eta=(2\zeta+1-\alpha)/2$), the claim follows from $p_{\zeta,\eta}^{(\alpha)}(t,\cdot,\cdot)=\sum_{n\geq0}p_t^{(n,D)}$, Lemma~\ref{boglem34}, and geometric series.
\end{proof}

For $\eta=0$, i.e., $\Psi_\zeta(\eta)=0$, Theorem~\ref{heatkernelharmonic} is verified by the following lemma.

\begin{lemma}
  \label{boglem33}
  Let $\zeta\in(-1/2,\infty)$, $\alpha\in(0,2]$, and
  $\beta\in[0,2\zeta+1-\alpha)$.
  Then, for all $r,t>0$,
  \begin{align}
    \label{eq:boglem331}
    \int_0^\infty ds\, s^{2\zeta} p_\zeta^{(\alpha)}(t,r,s) s^{-\beta}
    = r^{-\beta}
    - \int_0^t d\tau \int_0^\infty ds\, s^{2\zeta} p_\zeta^{(\alpha)}(\tau,r,s) \, \frac{\Psi_\zeta(\beta)}{s^\alpha}\, s^{-\beta}.
  \end{align}
\end{lemma}

\begin{proof}
  The assertion for $\beta=0$ follows from \eqref{eq:normalizedalpha} and $\Psi_\zeta(0)=0$.
  For $\beta>0$, by \eqref{eq:kappainv} and Chapman-Kolmogorov,
  \begin{align*}
    r^{-\beta}
    & =  \int_0^t d\tau \,\int_0^\infty ds\, s^{2\zeta}  p_\zeta^{(\alpha)}(\tau,r,s) \frac{\Psi_\zeta(\beta)}{s^{\alpha}}s^{-\beta}
      + \int_0^\infty d\tau \,\int_0^\infty ds\, s^{2\zeta}  p_\zeta^{(\alpha)}(t+\tau,r,s) \frac{\Psi_\zeta(\beta)}{s^{\alpha}}s^{-\beta}\\
    & = \int_0^t d\tau \,\int_0^\infty ds\, s^{2\zeta}  p_\zeta^{(\alpha)}(\tau,r,s) \frac{\Psi_\zeta(\beta)}{s^{\alpha}}s^{-\beta}\\
    & + \int_0^\infty d\tau \,\int_0^\infty ds\, s^{2\zeta}  \int_0^\infty dz \, z^{2\zeta}\,p_\zeta^{(\alpha)}(t,r,z) p_\zeta^{(\alpha)}(\tau,z,s)  \frac{\Psi_\zeta(\beta)}{s^{\alpha}}s^{-\beta} \\
    & = \int_0^t d\tau \,\int_0^\infty ds\, s^{2\zeta}  p_\zeta^{(\alpha)}(\tau,r,s) \frac{\Psi_\zeta(\beta)}{s^{\alpha}}s^{-\beta} + \int_0^\infty dz\, z^{2\zeta}\,p_\zeta^{(\alpha)}(t,r,z) z^{-\beta},
  \end{align*}
  where in the last line we used again \eqref{eq:kappainv}.
\end{proof}

\begin{lemma}
  \label{lem:iiptconv}
  Let $\zeta\in(-1/2,\infty)$, $\alpha\in(0,2]$, $\eta\in(0,(2\zeta+1-\alpha)/2)$, and $\beta\in(0,2\zeta+1-\alpha-\eta)$. Then, for all $r,t>0$,
  \begin{align}
    & \int_0^t d\tau \int_0^\infty ds\, s^{2\zeta} p_{\zeta,\eta}^{(\alpha)}(\tau,r,s)s^{-\beta-\alpha} < \infty,  \label{EQ1:estbgd}\\
    & \int_0^\infty ds\, s^{2\zeta} p_{\zeta,\eta}^{(\alpha)}(\tau,r,s)s^{-\beta} < \infty. \label{EQ2:estbgd}
  \end{align}
\end{lemma}

\begin{proof}
  By \eqref{eq:bogcor35inf}, for $\beta \in (\eta,2\zeta+1-\alpha-\eta)$, we have 
  \begin{align}\label{eq:estbgd}
    \int_0^t d\tau \int_0^\infty ds\, s^{2\zeta} p_{\zeta,\eta}^{(\alpha)}(\tau,r,s)s^{-\beta-\alpha} \le \frac{1}{\Psi_\zeta(\beta)- \Psi_\zeta(\eta)} r^{-\beta}.
  \end{align}
  Now let $\beta \in (0,\eta]$. Note that $\beta < (2\zeta+1-\alpha)/2$ and $(2\zeta+1-\alpha)/2\in (\eta,2\zeta+1-\alpha-\eta)$. By \eqref{eq:bogprop32} and \eqref{eq:estbgd},
  \begin{align} \label{eq:pw(-beta-alpha)IntApprox}
    \begin{split}
      & \int_0^t d\tau \int_0^\infty ds\, s^{2\zeta} p_{\zeta,\eta}^{(\alpha)}(\tau,r,s)s^{-\beta-\alpha} \\
      & \le \int_0^t d\tau \int_0^\infty ds\, s^{2\zeta} p_{\zeta,\eta}^{(\alpha)}(\tau,r,s)(1+s^{-(2\zeta+1-\alpha)/2-\alpha}) \\
      & \le M\int_0^t [1+(\tau^{-1/\alpha}r)^{-\eta}] d\tau + \frac{1}{\Psi_\zeta((2\zeta+1-\alpha)/2)- \Psi_\zeta(\eta)} r^{-(2\zeta+1-\alpha)/2} <\infty,
    \end{split}
  \end{align}
  Hence, \eqref{EQ1:estbgd} follows by \eqref{eq:estbgd} and \eqref{eq:pw(-beta-alpha)IntApprox}. By Lemma \ref{boglem33}, we have 
  \begin{equation} 
    \label{eq:pwy-betaApprox}
    \int_0^\infty ds\, s^{2\zeta} p_{\zeta}^{(\alpha)}(\tau,r,s)s^{-\beta} \le r^{-\beta}.
  \end{equation}
  Therefore, by Duhamel's formula, \eqref{eq:pwy-betaApprox}, and \eqref{EQ1:estbgd},
  \begin{align*}
    & \int_0^\infty ds\, s^{2\zeta} p_{\zeta,\eta}^{(\alpha)}(\tau,r,s)s^{-\beta} \\ 
    & \le \int_0^\infty ds\, s^{2\zeta} p_{\zeta}^{(\alpha)}(\tau,r,s)s^{-\beta} + \int_0^t d\tau \int_0^\infty ds\, s^{2\zeta} p_{\zeta,\eta}^{(\alpha)}(\tau,r,s)q(s)s^{-\beta} < \infty, 
  \end{align*}
  which ends the proof of \eqref{EQ2:estbgd}.
\end{proof} 

We are now ready to give the
\begin{proof}[Proof of Theorem \ref{heatkernelharmonic}]
  Let $t>0$ and $r>0$. By Duhamel's formula,
  \begin{align}
    \begin{split}
    \label{eq1:thm:tpInt}
    &\int_0^\infty ds\, s^{2\zeta} p_{\zeta,\eta}^{(\alpha)}(t,r,s) s^{-\beta}  =  \int_0^\infty ds\,s^{2\zeta} p_{\zeta}^{(\alpha)}(t,r,s) s^{-\beta} \\   
    & \qquad\qquad + \int_0^t d\tau 	\int_0^\infty dz\,z^{2\zeta}p_{\zeta,\eta}^{(\alpha)}(\tau,r,z) q(z) \left(\int_0^\infty ds\,s^{2\zeta}p_{\zeta}^{(\alpha)}(t-\tau, z, s) s^{- \beta} \right) .
    \end{split}
  \end{align}
  Next, by Lemma \ref{boglem33},
  \begin{align}
    &\Psi_{\zeta}(\beta) \int_0^t du\int_0^\infty ds\, s^{2\zeta} p_{\zeta,\eta}^{(\alpha)}(u,r,s) s^{-\alpha - \beta } \notag\\
    & = \Psi_{\zeta}(\beta) \int_0^t du \int_0^\infty ds\, s^{2\zeta} p_{\zeta}^{(\alpha)}(u,r,s) s^{-\alpha - \beta } \notag\\ 
    & + \Psi_{\zeta}(\beta)  \int_0^t du \int_0^\infty ds\, s^{2\zeta}  \int_0^u d\tau \int_0^\infty dz\,z^{2\zeta} p_{\zeta,\eta}^{(\alpha)}(\tau,r,z) q(z) p_{\zeta}^{(\alpha)}(u-\tau, z, s) s^{-\alpha - \beta} \notag\\ 
    & = r^{-\beta} - \int_0^\infty s^{2\zeta} p_{\zeta}^{(\alpha)}(t,r,s) s^{-\beta} \label{eq2:thm:tpInt}\\
    & + \int_0^t d\tau \int_0^\infty dz\, z^{2\zeta}   p_{\zeta,\eta}^{(\alpha)}(\tau,r,z) q(z) \left(\Psi_{\zeta}(\beta)\int_0^{t-\tau} du \int_0^\infty ds\, s^{2\zeta} p_{\zeta}^{(\alpha)}(u, z, s) s^{-\alpha - \beta} \right). \notag
  \end{align}
  We add \eqref{eq1:thm:tpInt} and \eqref{eq2:thm:tpInt} and apply Lemma \ref{boglem33} to the terms in parentheses to get
  \begin{align}
    \label{eq:tpIntpos}
    & \int_0^\infty ds\, s^{2\zeta} p_{\zeta,\eta}^{(\alpha)}(t,r,s) s^{-\beta} + \Psi_{\zeta}(\beta) \int_0^t d\tau \int_0^\infty ds\, s^{2\zeta} p_{\zeta,\eta}^{(\alpha)}(\tau,r,s) s^{-\alpha - \beta } \notag \\
    & = r^{-\beta} + \Psi_{\zeta}(\eta) \int_0^t d\tau \int_0^\infty s^{2\zeta} p_{\zeta,\eta}^{(\alpha)}(\tau,r,s) s^{-\alpha - \beta }.
  \end{align}

  Let $\eta \in (0,(2\zeta+1-\alpha)/2)$. Then  \eqref{eq:heatkernelharmonic2} follows by \eqref{eq:tpIntpos} and Lemma \ref{lem:iiptconv}. Furthermore, let $\beta \nearrow\delta$. By \eqref{eq:heatkernelharmonic2}, Lemma \ref{lem:iiptconv}, and the Lebesgue convergence theorem, we get \eqref{eq:heatkernelharmonic1}.

  Now, consider $\eta = (2\zeta+1-\alpha)/2$ and let $\beta \in (0,\eta)$. Note that $p_{\zeta,\eta}^{(\alpha)}$ is the perturbation of $p_{\zeta,\beta}^{(\alpha)}$ by $(\Psi_{\zeta}(\eta) - \Psi_{\zeta}(\beta))s^{-\alpha}$ (see, e.g., \cite{Bogdanetal2008}). Hence, by Duhamel's formula, Fubini's theorem, and \eqref{eq:heatkernelharmonic1} applied to $p_{\zeta,\beta}^{(\alpha)}$, we get
  \begin{align*}
    &(\Psi_{\zeta}(\eta)-\Psi_{\zeta}(\beta)) \int_0^t d\tau \int_0^\infty ds\, s^{2\zeta} p_{\zeta,\eta}^{(\alpha)}(\tau,r,s) s^{-\alpha - \beta } \\
    & =  \int_0^\infty  dz\, z^{2\zeta} \left[(\Psi_{\zeta}(\eta)-\Psi_{\zeta}(\beta)) \int_0^t d\tau \int_0^\infty ds\, s^{2\zeta} p_{\zeta,\eta}^{(\alpha)}(\tau,r,s) s^{-\alpha} p_{\zeta,\beta}^{(\alpha)}(t-\tau,s,z)\right] z^{-\beta} 	\\
    & = \int_0^\infty  dz\, z^{2\zeta} \left[ p_{\zeta,\eta}^{(\alpha)}(t,r,z) - p_{\zeta,\beta}^{(\alpha)}(t,r,z)\right] z^{-\beta} \\
    & = \int_0^\infty  dz\, z^{2\zeta} p_{\zeta,\eta}^{(\alpha)}(t,r,z)  z^{-\beta} - r^{-\beta}, 
  \end{align*}
  which gives \eqref{eq:heatkernelharmonic2}. Next, let $\beta < \eta$. Since $s^{-\beta} < 1+ s^{-\eta}$, by \eqref{eq:supermediantransformedalphaintro} and \eqref{eq:bogprop32}, 
  \begin{align*}
    \int_0^\infty ds\, s^{2\zeta}  p_{\zeta,\eta}^{(\alpha)}(t,r,s) s^{-\beta} \le \int_0^\infty ds\, s^{2\zeta}  p_{\zeta,\eta}^{(\alpha)}(t,r,s) (1+ s^{-\eta}) <\infty.
  \end{align*}
  By \eqref{eq:heatkernelharmonic2},
  \begin{align*}
    \int_0^\infty ds\, s^{2\zeta}  p_{\zeta,\eta}^{(\alpha)}(t,r,s) s^{-\beta} dy \ge r^{-\beta}.
  \end{align*}
  Hence, letting $\beta \nearrow \eta$, by Lebesgue's convergence theorem we get 
  \begin{align*}
    \int_0^\infty ds\, s^{2\zeta}  p_{\zeta,\eta}^{(\alpha)}(t,r,s) s^{-\eta } dy \ge r^{-\eta},
  \end{align*}
  which together with \eqref{eq:supermediantransformedalphaintro} gives \eqref{eq:heatkernelharmonic1}. This concludes the proof of Theorem \ref{heatkernelharmonic}.
\end{proof}

Using Theorem~\ref{heatkernelharmonic} and Proposition~\ref{bogprop32}, we now prove an \textit{approximate supermedian property} for
\begin{align}
  \label{eq:defbighscaled}
  H(t,r) := 1 + (t^{-1/\alpha}r)^{-\eta}, \quad r,t>0
\end{align}\label{$H(t,r)$}
and
\begin{align}
  \label{eq:defbigh}
  H(r) := H(1,r) = 1+r^{-\eta}
\end{align}\label{$H(r)$}
with respect to $p_{\zeta,\eta}^{(\alpha)}(t,\cdot,\cdot)$.
\begin{corollary}
  \label{bogcor38}
  \label{boglem41}
  Let $\zeta\in(-1/2,\infty)$, $\alpha\in(0,2]$, and $\eta\in(0,(2\zeta+1-\alpha)/2]$. Then,  there is a constant $M>0$ such that for all $r,t>0$,
  \begin{align}
    \label{eq:bogcor38}
    \int_0^\infty ds\, s^{2\zeta} p_{\zeta,\eta}^{(\alpha)}(t,r,s)H(t,s)
    \leq (M+1) H(t,r).
  \end{align}
  Furthermore, for all $\beta\in[0,\eta]$, 
  \begin{align}
    \label{eq:boglem41}
    \int_0^\infty ds\, s^{2\zeta} p_{\zeta,\eta}^{(\alpha)}(t,r,s) \left(\frac{s}{t^{1/\alpha}}\right)^{-\beta}
    \leq M H(t,r).
  \end{align}
  Finally,
  \begin{align}
    \label{eq:boglem42aux4}
    \int_0^1 d\tau \int_0^\infty dz\, z^{2\zeta } p_{\zeta,\eta}^{(\alpha)}(1-\tau,r,z)q(z)
    \leq M H(r).
  \end{align}
\end{corollary}

\begin{proof}
  Formula~\eqref{eq:bogcor38} follows from \eqref{eq:heatkernelharmonic1} and \eqref{eq:bogprop32}.
  Formula~\eqref{eq:boglem41} follows from $s^{-\beta}\leq 1+s^{-\eta}=H(s)$, Formula~\eqref{eq:bogcor38} and scaling. To prove Formula~\eqref{eq:boglem42aux4}, we use~\eqref{eq:bogprop32} and Duhamel's formula to obtain
  \begin{align*}
    \begin{split}
      MH(r)
      & \ge \int_0^\infty ds\, s^{2\zeta} p_{\zeta,\eta}^{(\alpha)}(1,r,s) \\
      & = \int_0^\infty ds\, s^{2\zeta} \left(p_\zeta^{(\alpha)}(1,r,s) + \int_0^1 d\tau \int_0^\infty dz\, z^{2\zeta} p_{\zeta,\eta}^{(\alpha)}(1-\tau,r,z)q(z) p_\zeta^{(\alpha)}(\tau,z,s) \right) \\
      & = 1 + \int_0^1 d\tau \int_0^\infty dz\, z^{2\zeta} p_{\zeta,\eta}^{(\alpha)}(1-\tau,r,z)q(z).
    \end{split}
  \end{align*}
  This concludes the proof.
\end{proof}

Finally, we show the finiteness of $p_{\zeta,\eta}^{(\alpha)}(t,r,s)$ for all $r,s,t>0$.

\begin{lemma}
  \label{pzetaetafinite}
  Let $\zeta\in(-1/2,\infty)$, $\alpha\in(0,2]$, $\eta\in(0,(2\zeta+1-\alpha)/2]$. Then $p_{\zeta,\eta}^{(\alpha)}(t,r,s)<\infty$ for all $r,s,t>0$.  
\end{lemma}

\begin{proof}
  By the scaling \eqref{eq:scalingalpha}, it suffices to consider $t=1$.
  By \eqref{eq:supermediantransformedalphaintro}, we know $p_{\zeta,\eta}^{(\alpha)}(1,r,s)<\infty$ for all $r>0$ and almost all $s>0$. We now show finiteness for a given $r>0$ and all $s>0$. By \eqref{eq:bogprop32aux2}, i.e.,
  \begin{align}
    \label{eq:pzetaetafiniteaux1}
    P_{1}^{(n,D)}(1,r,s) \leq 2p_\zeta^{(\alpha)}(1,r,s) + 2\int_0^1 d\tau\int_0^R dz\, z^{2\zeta} p_{\zeta,\eta}^{(\alpha)}(\tau,r,z)q(z)p_\zeta^{(\alpha)}(1-\tau,z,s),
  \end{align}
  with $R=2\wedge(2\Psi_\zeta(\eta))^{1/\alpha}$ and $P_t^{(n,D)}(r,s) = \sum_{k=0}^n p_t^{(n,D)}(r,s)$, it suffices to show
  \begin{align}
    \label{eq:pzetaetafiniteaux2}
    \int_0^1 d\tau\int_0^R dz\, z^{2\zeta} p_{\zeta,\eta}^{(\alpha)}(\tau,r,z)q(z)p_\zeta^{(\alpha)}(1-\tau,z,s) < \infty.
  \end{align}
  If $s\geq2R$, we use the uniform boundedness of $p_\zeta^{(\alpha)}(1-\tau,z,s)$ and \eqref{eq:boglem42aux4}, which allows to bound the remaining integral by
  \begin{align}
    \label{eq:pzetaetafiniteaux3}
    \int_0^1d\tau\int_0^\infty ds\, z^{2\zeta} p_{\zeta,\eta}^{(\alpha)}(\tau,r,z)q(z)\lesssim H(r)<\infty.
  \end{align}
  We assume $s<2R$ from now on.
  We first consider the integral over $\tau\in(0,1-\frac{1}{2\Psi_\zeta(\eta)}(\frac{s}{2})^\alpha)$. Then  the left-hand side of \eqref{eq:pzetaetafiniteaux2} is again finite by the uniform boundedness of $p_\zeta^{(\alpha)}(1-\tau,z,s)$ and by \eqref{eq:boglem42aux4}.
  It remains to consider the integral over $\tau\in(1-\frac{1}{2\Psi_\zeta(\eta)}(\frac{s}{2})^\alpha,1)$. We now distinguish between $|z-s|\lessgtr s/2$. For $|z-s|>s/2$, we bound
  \begin{align*}
    p_\zeta^{(\alpha)}(1-\tau,z,s) \lesssim \frac{1}{s^{2\zeta+1+\alpha}}
  \end{align*}
  and use \eqref{eq:pzetaetafiniteaux3}. If $|z-s|\leq s/2$, then $z\geq s/2$. Thus, by \eqref{eq:trivialupperbound} and the Chapman--Kolmogorov equation, we can bound the integral in question by
  \begin{align*}
    & \Psi_\zeta(\eta)\left(\frac{2}{s}\right)^\alpha \int_{1-\frac{1}{2\Psi_\zeta(\eta)}(\frac{s}{2})^\alpha}^1 d\tau \int_{s/2}^R dz\, p_{\zeta,\eta}^{(\alpha)}(\tau,r,z)p_{\zeta,\eta}^{(\alpha)}(1-\tau,z,s) \\
    & \quad \leq \Psi_\zeta(\eta)\left(\frac{2}{s}\right)^\alpha \cdot \frac{1}{2\Psi_\zeta(\eta)}\left(\frac{s}{2}\right)^\alpha \cdot p_{\zeta,\eta}^{(\alpha)}(1,r,s)
      = \frac12 p_{\zeta,\eta}^{(\alpha)}(1,r,s),
  \end{align*}
  which can be absorbed by the left-hand side of \eqref{eq:pzetaetafiniteaux1} upon taking $n\to\infty$. This concludes the proof.
\end{proof}

\subsection{The case $\eta<0$}
\label{ss:integralanalysisnegativeeta}

In this section, we extend Theorem~\ref{heatkernelharmonic} to \(\eta < 0\). Given the explicit expression for \(p_{\zeta,\eta}^{(2)}\), this is straightforward for \(\alpha = 2\). Hence, the main goal of this section is to generalize the previous integral analysis for \(\alpha \in (0, 2)\) to those coupling constants \(\Psi_\zeta(\eta)\) with \(\eta \in (-\alpha, 0)\). Our approach relies on compensating the kernels \(p_\zeta^{(\alpha)}(t,r,s)\) to add extra integrability for large \(t\); see \eqref{eq:defhbetagammaposcouplingsalpha}.

\smallskip
We begin the integral analysis by constructing generalized ground states for $\alpha=2$ and then use subordination to study $\alpha\in(0,2)$ in the spirit of \cite{Bogdanetal2016}. 
We proceed similarly as in \cite[Section~3]{BogdanMerz2024} by integrating $p_\zeta^{(\alpha)}$ against suitable functions of space and time. This gave us the functions $h_{\beta,\gamma}$ in \eqref{eq:defhbetagammaalphatransformed} discussed at the beginning of Section~\ref{s:integralana} above. However, since we wish to consider a wider range of parameters, we modify our construction slightly by compensating the kernel $p_\zeta^{(\alpha)}(t,r,s)$ with $p_\zeta^{(\alpha)}(t,r,0)$ when we integrate against monomials growing too fast for large $t$. The following calculation is similar but more delicate than \cite[Proposition~3.1]{BogdanMerz2024}. 

\begin{lemma}
  \label{hbetagammaposcouplingsalpha}
  Let $\zeta\in(-1/2,\infty)$, $\alpha\in(0,2]$, $\gamma\in(-1,\infty)$, and $0<\beta<(2\zeta-\gamma+2)/\alpha$ with $\beta\neq (2\zeta-\gamma)/\alpha$.
  Then,
  \begin{align}
    \label{eq:defhbetagammaposcouplingsalpha}
    \begin{split}
      h_{\beta,\gamma}^{(+)}(s)
      & := \int_0^\infty \frac{dt}{t}\, t^\beta \int_0^\infty dr\, r^\gamma \left( p_\zeta^{(\alpha)}(t,r,0) - p_\zeta^{(\alpha)}(t,r,s) \right) \\
      & \ = C^{(\alpha)}(\beta,\gamma,\zeta) s^{\alpha\beta+\gamma-2\zeta},
    \end{split}
  \end{align}
  with $C^{(\alpha)}(\beta,\gamma,\zeta)$ as in \eqref{eq:timeshiftconstant1resultalpha}.
  
  If, additionally, $\beta=(2\zeta-\gamma-\eta)/\alpha$ for some $\eta\in(-\alpha,0)$, then
  \begin{align}
    \label{eq:hardypotclassictransformedposcouplingsalpha}
    q(s)
    = \frac{\Psi_\zeta(\eta)}{s^\alpha}
    = \frac{(\beta-1) h_{\beta-1,\gamma}(s)}{h_{\beta,\gamma}^{(+)}(s)}
  \end{align}
  where $h_{\beta-1,\gamma}(s)$ is as in \eqref{eq:defhbetagammaalphatransformed}.
\end{lemma}

Note that the case $\alpha\beta=2\zeta-\gamma$, which is excluded in the above lemma, corresponds to $\eta=0$, as can be seen by considering the right-hand sides of \eqref{eq:defhbetagammaposcouplingsalpha} and \eqref{eq:hardypotclassictransformedposcouplingsalpha}.
The constant on the right-hand side of \eqref{eq:defhbetagammaposcouplingsalpha} agrees with that in \eqref{eq:defhbetagammaalphatransformed}. The subtraction of $p_\zeta^{(\alpha)}(t,r,0)$ in \eqref{eq:defhbetagammaposcouplingsalpha}, however, allows us to extend the range of admissible $\beta$, compare with~\eqref{eq:defhbetagammaalphatransformed}.

\begin{proof}[Proof of Lemma~\ref{hbetagammaposcouplingsalpha}]
  Formula \eqref{eq:hardypotclassictransformedposcouplingsalpha} immediately follows from \eqref{eq:defhbetagammaposcouplingsalpha}, which we prove now. Consider first $\alpha=2$. For $\beta\in(0,(2\zeta-\gamma)/2)$, the calculations were already carried out in \cite[Proposition~3.1]{BogdanMerz2024}. Thus, consider $\beta\in((2\zeta-\gamma)/2,(2\zeta-\gamma+2)/2)$ from now on. It is for this range that we require the compensation by $p_\zeta^{(2)}(t,r,0)$ to make the $t$-integral convergent at $t=\infty$; at $t=0$, no compensation is necessary in view of the exponential factor $\me{-c(r-s)^2/t}$ in $p_\zeta^{(2)}(t,r,s)$ in \eqref{eq:easybounds2} and similarly for $p_\zeta^{(2)}(t,r,0)$ in~\eqref{eq:pzetatzero}. We now compute  $h_{\beta,\gamma}^{(+)}(s)$. By scaling,
  $h_{\beta,\gamma}^{(+)}(s)\cdot s^{-\gamma-2\beta+2\zeta} = C^{(+)}(\beta,\gamma,\zeta)$
  with
  \begin{align}
    \label{eq:timeshiftconstantposcouplings1}
    \begin{split}
      C^{(+)}(\beta,\gamma,\zeta)
      & = \int_0^\infty \frac{dt}{t}\, t^\beta \int_0^\infty dr\, r^\gamma \left(p_\zeta^{(2)}(t,r,0)-p_\zeta^{(2)}(t,r,1)\right).
    \end{split}
  \end{align}
  Since $\gamma>-1$, the $r$-integral gives, as in \cite{BogdanMerz2024},
  \begin{align}
  \label{eq:derivativeheatkernelpre}
    \begin{split}
      & \int_0^\infty dr\, r^\gamma \left(p_\zeta^{(2)}(t,r,0)-p_\zeta^{(2)}(t,r,1)\right) \\
      & \quad = -\frac{2^{\gamma -2 \zeta } \Gamma \left(\frac{\gamma +1}{2}\right) t^{\frac{1}{2} (\gamma -2 \zeta )} \left(\, _1F_1\left(\zeta -\frac{\gamma }{2};\zeta +\frac{1}{2};-\frac{1}{4 t}\right)-1\right)}{\Gamma \left(\zeta +\frac{1}{2}\right)}.
    \end{split}
  \end{align}
  Here, $_1F_1(a;b;z)$ with $a,b\in\C$, $b\notin\{0,-1,-2\}$, and $z\in\{w\in\C:\,|w|<1\}$ denotes Kummer's confluent hypergeometric function \cite[(13.2.2)]{NIST:DLMF}.\index{$_1F_1(a;b;z)$}
  The right-hand side of \eqref{eq:derivativeheatkernelpre} behaves like $t^{-\zeta+\gamma/2-1}$ as $t\to\infty$ by \cite[(13.2.2)]{NIST:DLMF}. To study the behavior at $t=0$, we use \cite[(13.2.39) and (13.2.23)]{NIST:DLMF}, i.e.,
  \begin{align}
    _1F_1(a;b;z) = \me{z}\, _1F_1(b-a;b;-z)
  \end{align}
  and
  \begin{align}
    \lim_{z\to\infty}\, {_1}{F}_1(a;b;z) \cdot \frac{1}{\me{z}z^{a-b}} 
    = \frac{1}{\Gamma(a)\Gamma(b)}
  \end{align}
  to infer the existence of some $c_{\gamma,\zeta}\in\R$ such that
  \begin{align}
  \, _1F_1\left(\zeta -\frac{\gamma }{2};\zeta +\frac{1}{2};-\frac{1}{4 t}\right) =
   c_{\gamma,\zeta} t^{-\gamma/2+\zeta} + \mathcal{O}(t^{-\gamma/2+\zeta+1}) & \quad \text{at} \ t=0.
  \end{align}
  Thus, the right-hand side behaves like $1+t^{-\zeta+\gamma/2}$ as $t\to0$.
  These asymptotics and similar computations as in \cite{BogdanMerz2024} thus give
  \begin{align}
    C^{(+)}(\beta,\gamma,\zeta)
    = -\frac{2^{-2 \beta -1} \Gamma (\beta) \Gamma \left(\frac{\gamma +1}{2}\right) \Gamma \left(-\beta +\zeta -\frac{\gamma }{2}\right)}{\Gamma \left(\beta +\frac{\gamma }{2}+\frac{1}{2}\right) \Gamma \left(-\frac{\gamma }{2}+\zeta\right)}
  \end{align}
  for $\beta\in(\zeta-\gamma/2,\zeta-\gamma/2+1)$.
  This proves \eqref{eq:defhbetagammaposcouplingsalpha} for $\alpha=2$.

  \smallskip
  Now consider $\alpha\in(0,2)$. By subordination and the previous computations,
  \begin{align}
    \begin{split}
      h_{\beta,\gamma}^{(+)}(s)
      & = \int_0^\infty dr\, r^\gamma \int_0^\infty d\tau \left[p_\zeta^{(2)}(\tau,r,s)-p_\zeta^{(2)}(\tau,r,0)\right] \int_0^\infty \frac{dt}{t}\, \sigma_t^{(\alpha/2)}(\tau) \cdot t^{\beta} \\
      & = \frac{\Gamma(\beta)}{\Gamma(\frac{\alpha\beta}{2})} \int_0^\infty dr\, r^\gamma \int_0^\infty d\tau\, \left[p_\zeta^{(2)}(\tau,r,s)-p_\zeta^{(2)}(\tau,r,0)\right] \cdot \tau^{-1+\frac{\alpha\beta}{2}} \\
      & = \frac{\Gamma(\beta)}{\Gamma(\frac{\alpha\beta}{2})} C\left(\frac{\alpha\beta}{2},\gamma,\zeta\right) s^{\alpha\beta+\gamma-2\zeta}
        = C^{(\alpha)}(\beta,\gamma,\zeta) s^{\alpha\beta+\gamma-2\zeta},
    \end{split}
  \end{align}
  by the definition of $C^{(\alpha)}(\beta,\gamma,\zeta)$ in \eqref{eq:timeshiftconstant1resultalpha}.
\end{proof}

Here and in the following, we could also work with $h_{\beta,\gamma}^{(+)}$ defined by \eqref{eq:defhbetagammaalphatransformed} when $\beta\in(0,(2\zeta-\gamma)/\alpha)$. Using the new ansatz \eqref{eq:defhbetagammaposcouplingsalpha} involving the difference of heat kernels allows us to increase the range of admissible $\beta$ and, importantly, work under the less restrictive assumption $\beta<(2\zeta-\gamma+2)/\alpha$.

\smallskip
We are now ready to start the actual integral analysis. We begin with a variant of Lemma~\ref{boglem33}.

\begin{lemma}
  \label{jaklem23}
  Let $\zeta\in(-1/2,\infty)$, $\alpha\in(0,2]$, $\beta\in(-\alpha,0)$, and $r,t>0$. Then,
  \begin{align}
    \label{eq:jaklem232}
    \begin{split}
      \int_0^\infty ds\, s^{2\zeta} p_\zeta^{(\alpha)}(t,r,s) s^{-\beta}
      = r^{-\beta} - \int_0^t d\tau \int_0^\infty ds\, s^{2\zeta} p_\zeta^{(\alpha)}(\tau,r,s) \frac{\Psi_\zeta(\beta)}{s^\alpha} s^{-\beta}.
    \end{split}
  \end{align}
\end{lemma}

\begin{proof}
  Let $\overline\beta,\overline\gamma$ be such that $h_{\overline\beta,\overline\gamma}^{(+)}(r) = C^{(\alpha)}(\overline\beta,\overline\gamma,\zeta) r^{-\beta} = h_{\overline\beta,\overline\gamma}$ with $h_{\overline\beta,\overline\gamma}^{(+)}(r)$ as in \eqref{eq:defhbetagammaposcouplingsalpha} and $h_{\overline\beta,\overline\gamma}(r)$ as in \eqref{eq:defhbetagammaalphatransformed}.
  We start with the $\Psi_\zeta(\beta)$ term. By the expression \eqref{eq:hardypottransformedalpha} for the Hardy potential, and \eqref{eq:hardypotclassictransformedposcouplingsalpha}, we have
  \begin{align}
    \label{eq:jaklem232aux1}
    \begin{split}
      \frac{\Psi_\zeta(\beta)}{s^\alpha} \cdot h_{\overline\beta,\overline\gamma}(s)
      = (\overline\beta-1) \int_0^\infty d\vartheta \int_0^\infty dz\, \vartheta^{\overline\beta-2} z^{\overline\gamma} p_\zeta^{(\alpha)}(\vartheta,s,z).
    \end{split}
  \end{align}
  Integrating \eqref{eq:jaklem232aux1} against $s^{2\zeta}p_\zeta^{(\alpha)}(\tau,r,s)\one_{[0,t]}(\tau)\,ds\,d\tau$ yields
  \begin{align}
    \label{eq:jaklem232aux2}
    \begin{split}
      & \int_0^t d\tau \int_0^\infty ds\, s^{2\zeta} p_\zeta^{(\alpha)}(\tau,r,s) \frac{\Psi_\zeta(\beta)}{s^\alpha} \cdot h_{\overline\beta,\overline\gamma}(s) \\
      & \quad = (\overline\beta-1) \int_0^t d\tau \int_0^\infty ds\, s^{2\zeta} p_\zeta^{(\alpha)}(\tau,r,s) \int_0^\infty d\vartheta \int_0^\infty dz\, \vartheta^{\overline\beta-2} z^{\overline\gamma} p_\zeta^{(\alpha)}(\vartheta,s,z) \\
      & \quad = \int_0^t d\tau \int_0^\infty d\vartheta \int_0^\infty dz\, p_\zeta^{(\alpha)}(\tau+\vartheta,r,z) \partial_{\vartheta}\vartheta^{\overline\beta-1} z^{\overline\gamma} \\
      & \quad = -\int_0^t d\tau \int_0^\infty d\vartheta \int_0^\infty dz\, [\partial_\tau p_\zeta^{(\alpha)}(\tau+\vartheta,r,z)]\, \vartheta^{\overline\beta-1} z^{\overline\gamma} \\
      & \quad = \int_0^\infty d\vartheta \int_0^\infty dz\, \left[p_\zeta^{(\alpha)}(\vartheta,r,z) - p_\zeta^{(\alpha)}(\vartheta+t,r,z)\right]\, \vartheta^{\overline\beta-1} z^{\overline\gamma},
    \end{split}
  \end{align}
  where we used the semigroup property and $(\overline\beta-1)\vartheta^{\overline\beta-2}=(\vartheta^{\overline\beta-1})'$ in the second step, integrated by parts in the third step\footnote{The boundary term at $\vartheta=\infty$ vanishes due to the decay of order $\vartheta^{-(2\zeta+1)/\alpha}$ for the heat kernel (see \eqref{eq:easybounds2} and \eqref{eq:heatkernelalpha1weightedsubordinatedboundsfinal}), which suppresses $\vartheta^{\overline\beta-1}$ at infinity, since $\overline\beta<(2\zeta-\gamma)/\alpha<(2\zeta+1)/\alpha$ for $\gamma>-1$. The boundary term at $\vartheta=0$ vanishes since $\overline\beta>1$ and the heat kernels are non-singular at the temporal origin.}, and integrated with respect to $d\tau$ using the fundamental theorem of calculus. Now we add and subtract $p_\zeta^{(\alpha)}(\vartheta,0,z)$ in the integral on the right-hand side of \eqref{eq:jaklem232aux2}, and use \eqref{eq:hardypotclassictransformedposcouplingsalpha} to obtain
  \begin{align}
    \footnotesize
    \begin{split}
      & \int_0^t d\tau \int_0^\infty ds\, s^{2\zeta} p_\zeta^{(\alpha)}(\tau,r,s)\, \frac{\Psi_\zeta(\beta)}{s^\alpha} \cdot h_{\overline\beta,\overline\gamma}(s) \\
      & \quad = \int_0^\infty d\vartheta \int_0^\infty dz\, \left[\left(p_\zeta^{(\alpha)}(\vartheta,r,z) - p_\zeta^{(\alpha)}(\vartheta,0,z)\right) + \left(p_\zeta^{(\alpha)}(\vartheta,0,z) - p_\zeta^{(\alpha)}(\vartheta+t,r,z)\right)\right]\, \vartheta^{\overline\beta-1} z^{\overline\gamma} \\
      & \quad = h_{\overline\beta,\overline\gamma}(r) - \int_0^\infty d\vartheta\int_0^\infty dz \int_0^\infty dw\, w^{2\zeta} p_\zeta^{(\alpha)}(t,r,w) \left(p_\zeta^{(\alpha)}(\vartheta,w,z)-p_\zeta^{(\alpha)}(\vartheta,0,z)\right) \, \vartheta^{\overline\beta-1} z^{\overline\gamma} \\
      & \quad = h_{\overline\beta,\overline\gamma}(r) - \int_0^\infty dw\, w^{2\zeta} p_\zeta^{(\alpha)}(t,r,w) h_{\overline\beta,\overline\gamma}^{(+)}(w).
    \end{split}
  \end{align}
  In the last two steps, we used the definition \eqref{eq:hardypotclassictransformedposcouplingsalpha} of $h_{\overline\beta,\overline\gamma}^{(+)}$, the normalization~\eqref{eq:normalizedalpha}, and the semigroup property for $p_\zeta^{(\alpha)}(t+\vartheta,\cdot,\cdot)$.
\end{proof}

We finish off with a complement of Theorem \ref{heatkernelharmonic} for repulsive Hardy potentials
Recall that $M=\infty$ if $\alpha=2$ and $M=\alpha$ if $0<\alpha<2$.
\begin{theorem}
  \label{jakthm24}
  Let $\zeta\in(-1/2,\infty)$, $\alpha\in(0,2]\cap(0,2\zeta+1)$, $\eta\in(-M,0)$, $\beta\in(-M,0)$, and $r,t>0$. Then
  \begin{align}
    \label{eq:jakthm241}
    \begin{split}
      & \int_0^\infty ds\, s^{2\zeta} p_{\zeta,\eta}^{(\alpha)}(t,r,s)s^{-\beta} \\
      & \quad = r^{-\beta} + (\Psi_\zeta(\eta)-\Psi_\zeta(\beta)) \int_0^t d\tau \int_0^\infty ds\, s^{2\zeta} p_{\zeta,\eta}^{(\alpha)}(\tau,r,s)s^{-\beta-\alpha}.
    \end{split}
  \end{align}
  In particular, for any $r,t>0$,
  \begin{align}
    \label{eq:jakthm242}
    \begin{split}
      \int_0^\infty ds\, s^{2\zeta} p_{\zeta,\eta}^{(\alpha)}(t,r,s)s^{-\eta} = r^{-\eta}.
    \end{split}
  \end{align}
\end{theorem}

\begin{proof}
  For $\alpha=2$, the statements follow from explicit calculations using \eqref{eq:mainresultalpha2}. Thus, suppose $\alpha\in(0,2)$.
  By Duhamel's formula \eqref{eq:duhamelclassictransformed} and \eqref{eq:jaklem232},
  \begin{align}
    \begin{split}
      & \int_0^t d\tau \int_0^\infty ds\, s^{2\zeta} p_{\zeta,\eta}^{(\alpha)}(\tau,r,s)\, \frac{\Psi_\zeta(\beta)}{s^\alpha} s^{-\beta} \\
      & \quad = \int_0^t d\tau \int_0^\infty ds\, s^{2\zeta} p_\zeta^{(\alpha)}(\tau,r,s)\, \frac{\Psi_\zeta(\beta)}{s^\alpha}\, s^{-\beta} \\
      & \qquad + \int_0^t d\tau \int_0^\infty ds\, s^{2\zeta-\beta} \int_0^\tau d\vartheta \int_0^\infty dz\, z^{2\zeta} p_{\zeta,\eta}^{(\alpha)}(\vartheta,r,z)q(z) p_\zeta^{(\alpha)}(\tau-\vartheta,z,s)\,\frac{\Psi_\zeta(\beta)}{s^\alpha} \\
      & \quad = r^{-\beta} - \int_0^\infty ds\, s^{2\zeta} p_\zeta^{(\alpha)}(t,r,s)s^{-\beta} \\
      & \qquad + \int_0^t d\vartheta \int_0^\infty dz\, z^{2\zeta} p_{\zeta,\eta}^{(\alpha)}(\vartheta,r,z)q(z)\left[z^{-\beta} - \int_0^\infty ds\, s^{2\zeta} p_\zeta^{(\alpha)}(t-\vartheta,z,s)s^{-\beta}\right].
    \end{split}
  \end{align}
  Here we applied a time-shift $\tau\mapsto\tau+\vartheta$ and \eqref{eq:jaklem232} to both lines after the first equality. Bringing the $z^{-\beta}$ term to the left-hand side and using
  \begin{align*}
    & \int_0^\infty ds\, s^{2\zeta} \int_0^t d\vartheta \int_0^\infty dz\, z^{2\zeta} p_{\zeta,\eta}^{(\alpha)}(\vartheta,r,z)q(z) p_\zeta^{(\alpha)}(t-\vartheta,z,s) s^{-\beta} \\
    & \quad = \int_0^\infty ds\, s^{2\zeta} (p_{\zeta,\eta}^{(\alpha)}(t,r,s) - p_\zeta^{(\alpha)}(t,r,s))s^{-\beta},
  \end{align*}
  which follows from Duhamel's formula, we obtain
  \begin{align}
    \begin{split}
      & \left(\Psi_\zeta(\beta)-\Psi_\zeta(\eta)\right)\int_0^t d\tau \int_0^\infty ds\, s^{2\zeta} p_{\zeta,\eta}^{(\alpha)}(\tau,r,s)s^{-\beta-\alpha} \\
      & \quad = r^{-\beta} - \int_0^\infty ds\, s^{2\zeta} p_\zeta^{(\alpha)}(t,r,s)s^{-\beta} \\
      & \qquad - \int_0^\infty ds\, s^{2\zeta} \left[p_{\zeta,\eta}^{(\alpha)}(t,r,s) - p_\zeta^{(\alpha)}(t,r,s)\right] s^{-\beta} \\
      & \quad = r^{-\beta} - \int_0^\infty ds\, s^{2\zeta} p_{\zeta,\eta}^{(\alpha)}(t,r,s) s^{-\beta}.
    \end{split}
  \end{align}
  This proves \eqref{eq:jakthm241}. Finally, \eqref{eq:jakthm242} follows from \eqref{eq:jakthm241} by taking $\beta=\eta$.
\end{proof}

\section{Proof of Theorem~\ref{mainresultgen}}
\label{s:proofmainresultgen}

In this section, we prove the upper and lower bounds for, and continuity of $p_{\zeta,\eta}^{(\alpha)}(t,r,s)$. Moreover, we prove that $p_{\zeta,\eta}^{(\alpha)}$ blows up for supercritical coupling constants. 
Before we start the proof of Theorem~\ref{mainresultgen}, we sketch the proof of \eqref{eq:mainresultalpha2}, which
essentially follows from a change of variables. It was carried out, e.g., by Metafune, Negro, and Spina~\cite[Proposition~4.12, Theorem~4.14]{Metafuneetal2018}. 

\begin{proof}[Proof of \eqref{eq:mainresultalpha2}]
  Recall that $p_\zeta^{(2)}$ is the heat kernel of the nonnegative Bessel operator $\cl_\zeta$ in \eqref{eq:deflzeta}; see \eqref{eq:heatkernellzeta}. Using the unitary operator $U_\zeta:L^2(\R_+,r^{2\zeta}dr)\to L^2(\R_+,dr)$ defined by
  \begin{align}
    \label{eq:defdoob}
    L^2(\R_+,r^{2\zeta}dr) \ni u \mapsto (U_\zeta u)(r)=r^{\zeta}u(r) \in L^2(\R_+,dr),
  \end{align}
  we have
  \begin{align}
  \label{eq:changeofvariables}
    \begin{split}
     \cl_\zeta - \frac{\Psi_\zeta(\eta)}{r^2}
       & = U_\zeta^*\left(-\frac{d^2}{dr^2} + \frac{(2\zeta-1)^2-1-4\Psi_\zeta(\eta)}{4r^2}\right)U_\zeta \\
      & = U_\eta^* U_{\zeta-\eta}^*\left(-\frac{d^2}{dr^2} + \frac{(2(\zeta-\eta)-1)^2-1}{4r^2}\right)U_{\zeta-\eta}U_\eta \\
      & = U_\eta^* \cl_{\zeta-\eta} U_\eta  \quad \text{in} \ L^2(\R_+,r^{2\zeta}dr).
    \end{split}
  \end{align}
  Here we used $(2\vartheta-1)^2-1=(2\zeta-1)^2-1-4\Psi_\zeta(\eta)$ for $\vartheta\in\{\zeta-\eta,1-\zeta+\eta\}$ since $\Psi_\zeta(\eta)=\eta(2\zeta-1-\eta)$ for $\alpha=2$. Thus, the change of variables in \eqref{eq:changeofvariables} yields
  \begin{align}
    \exp(-t(\cl_\zeta-\Psi_\zeta(\eta)/r^2))
    = \exp\left(-tU_\eta^*\cl_{\zeta-\eta}U_\eta\right)
    = U_\eta^* \exp(-t\cl_{\zeta-\eta}) U_\eta.
  \end{align}
  Hence, since $p_{\zeta-\eta}^{(2)}$ is the heat kernel of $\cl_{\zeta-\eta}$, we see that
  \begin{align}
    p_{\zeta,\eta}^{(2)}(t,r,s) = (rs)^{-\eta} p_{\zeta-\eta}^{(2)}(t,r,s)
  \end{align}
  is the heat kernel of $\cl_\zeta-\Psi_\zeta(\eta)/r^2$ in $L^2(\R_+,r^{2\zeta}dr)$, where $\cl_\zeta$ is defined in \eqref{eq:deflzeta}.
\end{proof}

In the remainder of this section, we focus on $\alpha\in(0,2)$. 
To prove Theorem \ref{mainresultgen}, we build on the ideas of \cite{Bogdanetal2019} and \cite{JakubowskiWang2020} for the cases $\eta\in(-\alpha,0)$ and $\eta\in(0,(2\zeta+1-\alpha)/2]$, respectively. Let us first discuss positive $\eta$. The main tools we use are Duhamel's formula, the Chapman--Kolmogorov equation, and the scaling given in Proposition \ref{propertiesschrodheatkernel}. The first step is to get estimates for $p_{\zeta,\eta}^{(\alpha)}(1,r,s)$ with arbitrary $r>0$ and $s\gtrsim1$. In Lemma~\ref{boglem43} below, we show $p_{\zeta,\eta}^{(\alpha)}(1,r,s) \lesssim H(r)p_\zeta^{(\alpha)}(1,r,s)$ by properly splitting the integrals in Duhamel's formula and applying the estimates \eqref{eq:heatkernelalpha1weightedsubordinatedboundsfinal} and \eqref{eq:bogcor38}.
For $0<r\vee s\lesssim1$, we show $p_{\zeta,\eta}^{(\alpha)}(1,r,s)\lesssim H(r) H(s)$ using $p_{\zeta,\eta}^{(\alpha)}(2,r,s) = \int_{\R_+} dz\, z^{2\zeta} p_{\zeta,\eta}^{(\alpha)}(1,r,z) p_{\zeta,\eta}^{(\alpha)}(1,z,s) \left(\one_{z<1} + \one_{z>1}\right)$ (Chapman--Kolmogorov). This requires bounds for $p_{\zeta,\eta}^{(\alpha)}(1,r,z)$ for $r\lesssim1$ and $z\lessgtr1$.
For $r\lesssim1\lesssim z$, we use Lemma~\ref{boglem43} below.
For $r\vee z\lesssim1$, we use Lemmas~\ref{boglem44}--\ref{boglem47} below. More precisely, in Lemma~\ref{boglem44}, we prove the preliminary estimate $p_{\zeta,\eta}^{(\alpha)}(1,r,s) \lesssim r^{-\eta} s^{\mu-2\zeta-1}$, which we systematically improve in Lemma~\ref{boglem45} and in the crucial Lemma~\ref{boglem46} to get the upper bounds stated in Lemma~\eqref{boglem48}.
To get the lower bounds in Theorem~\ref{heatkernelharmonic}, we use the upper bounds stated in Lemma~\ref{boglem48}, the Chapman--Kolmogorov equation, and the invariance of the function $h(s) = s^{-\eta}$. 

\smallskip
For $\eta<0$, the Duhamel formula is almost useless. Our proofs are generally based on the Chapman--Kolmogorov equation and the method called "self-improving estimates". This method is applied in the key estimates of the mass $M(t,r)$ of $p_{\zeta,\eta}^{(\alpha)}$, defined in~\eqref{eq:defbighscaledposcouplings0}, which are contained in Proposition~\ref{jakprop31}. To that end, we proceed as follows. In the first step, we show $M(1,r) < \delta M(1,3^{1/\alpha}r) + A r^{-\eta}$ for sufficiently small $\delta$ and some $A>0$. Next, we iterate this inequality to obtain $M(1,r) < \delta^n M(1,3^{n/\alpha}r) + A_n r^{-\eta}$ for some bounded sequence $A_n$, thereby giving the desired estimates for $M(1,r)$. Proposition~\ref{jakprop31} yields the upper estimates of $p_{\zeta,\eta}^{(\alpha)}(1,r,s)$ for $r,s \lesssim 1$. The case $r,s \gtrsim 1$ is quite obvious because of the bound $p_{\zeta,\eta}^{(\alpha)} \le p_{\zeta}^{(\alpha)}$. In contrary to positive perturbations, the most challenging case is $r \lesssim 1 \lesssim s$. The key estimate is given in Lemma~\ref{jaklem33}, which together with \ref{jaklem34} allows to apply the method of "self-improving estimates" in the proof of the inequality~ \eqref{eq:jakthm35aux3}. The proof of the lower bounds is based on Lemma~\ref{jaklem38} and the estimates of the first summand in the perturbation series.

\subsection{The case $\eta\in(0,(2\zeta+1-\alpha)/2]$}

\subsubsection{Upper bound}

The proof of the upper bound in \eqref{eq:mainresultgen} will be concluded in Lemma~\ref{boglem48} after a series of preparatory lemmas involving $H(r) = H(1,r) = 1+r^{-\eta}$ defined in \eqref{eq:defbigh} above.

In the proofs below, we use the scaling symmetry of $p_{\zeta,\eta}^{(\alpha)}(t,r,s)$ to reduce the analysis to the case $t=1$. Our first step to study $p_{\zeta,\eta}^{(\alpha)}(1,r,s)$ is to use Duhamel's formula. In order to discuss "small" and "large" distances to the origin, we will, motivated by the uncertainty principle, introduce a cut-off in the time parameter, defined as
\begin{align}\index{$g(s)$}
  \label{eq:defg}
  g(s) := \frac{s^\alpha}{2^{\alpha+1}\Psi_\zeta(\eta)}, \quad s>0.
\end{align}
The prefactor $[2^{\alpha+1}\Psi_\zeta(\eta)]^{-1}$ is chosen such that certain parts of the integrals appearing in Duhamel's formula are absorbed by $p_{\zeta,\eta}^{(\alpha)}$. 
In turn, the time cut-off \eqref{eq:defg} suggests to distinguish distances to the origin that are smaller or larger than $2(2\Psi_\zeta(\eta))^{1/\alpha}$.
By the scaling symmetry of $p_{\zeta,\eta}^{(\alpha)}(t,r,s)$, we could have also considered values of $t$ different from one, which would then require to consider different time and spatial cut-offs.

\begin{remark}
  All theorems, propositions, and lemmas of this subsection continue to hold for $\alpha=2$. The only necessary change is that the spatial arguments $r,s>0$ of the heat kernel $p_\zeta^{(2)}(t,r,s)$ for $\alpha=2$ should be, in bounds, replaced with $cr$ and $cs$ for some $c=c(\zeta,\alpha,\eta)\in(0,\infty)$.
\end{remark}

In the following lemma, we examine the region $r<2(2\Psi_\zeta(\eta))^{1/\alpha}$, $s>0$.

\begin{lemma}
  \label{boglem42}
  Let $\zeta\in(-1/2,\infty)$,
  $\alpha\in(0,2)$,
  $\eta\in(0,(2\zeta+1-\alpha)/2]$, $s>0$, and $0<r\leq 2(2\Psi_\zeta(\eta))^{1/\alpha}$. Then,  there is $M>0$ such that
  \begin{align}
    \label{eq:boglem42}
    \begin{split}
      p_{\zeta,\eta}^{(\alpha)}(1,r,s)
      & \leq 2 (p_\zeta^{(\alpha)}(1,r,s) + M p_\zeta^{(\alpha)}(1,1,s)) H(r) \\
      & \quad + 2\int_{g(s)\wedge 1}^1 d\tau \int_{s/2}^\infty dz\, z^{2\zeta } p_{\zeta,\eta}^{(\alpha)}(1-\tau,r,z) q(z) p_\zeta^{(\alpha)}(\tau,z,s) \\
      & \quad + \int_0^{1/2}d\tau \int_0^{s/2} dz\, z^{2\zeta } p_{\zeta,\eta}^{(\alpha)}(1-\tau,r,z)q(z) p_\zeta^{(\alpha)}(\tau,z,s).
    \end{split}
  \end{align}
\end{lemma}

\begin{proof}
  By Duhamel's formula \eqref{eq:duhamelclassictransformed},
  \begin{align}
    \label{eq:boglem42aux1}
    p_{\zeta,\eta}^{(\alpha)}(1,r,s) = p_\zeta^{(\alpha)}(1,r,s) + I_1 + I_2 + I_3,
  \end{align}
  where
  \begin{align}
    I_1 & := \int_0^{1\wedge g(s)} d\tau \int_{s/2}^\infty dz\, z^{2\zeta} p_{\zeta,\eta}^{(\alpha)}(1-\tau,r,z)q(z) p_\zeta^{(\alpha)}(\tau,z,s), \\
    I_2 & := \int_{1\wedge g(s)}^1 d\tau \int_{s/2}^\infty dz\, z^{2\zeta } \,p_{\zeta,\eta}^{(\alpha)}(1-\tau,r,z)q(z) p_\zeta^{(\alpha)}(\tau,z,s), \\
    I_3 & := \int_0^{1} d\tau \int_0^{s/2} dz\, z^{2\zeta } \,p_{\zeta,\eta}^{(\alpha)}(1-\tau,r,z)q(z) p_\zeta^{(\alpha)}(\tau,z,s).
  \end{align}
  We leave $I_2$ untouched. To bound $I_1$, we use $q(z)\leq q(s/2)$, $p_\zeta^{(\alpha)}(\tau,z,s)\leq \,p_{\zeta,\eta}^{(\alpha)}(\tau,z,s)$, and the semigroup property to obtain
  \begin{align}
    \begin{split}
      I_1
      & \leq q\left(\frac s2\right) \int_0^{1\wedge g(s)}d\tau \int_0^\infty dz\, z^{2\zeta } \,p_{\zeta,\eta}^{(\alpha)}(1-\tau,r,z) \,p_{\zeta,\eta}^{(\alpha)}(\tau,z,s) \\
      & = \,p_{\zeta,\eta}^{(\alpha)}(1,r,s) q\left(\frac s2\right) \int_0^{1\wedge g(s)}d\tau \leq \frac12\,p_{\zeta,\eta}^{(\alpha)}(1,r,s).
    \end{split}
  \end{align}
  Since $p_{\zeta,\eta}^{(\alpha)}(1,r,s)<\infty$ by Lemma~\ref{pzetaetafinite}, this term can be absorbed by the left-hand side of \eqref{eq:boglem42aux1}. To treat $I_3$, we split
  \begin{align}
    \label{eq:boglem42aux2}
    \begin{split}
      I_3
      & = \int_0^{1/2}d\tau \int_0^{s/2} dz\, z^{2\zeta } \,p_{\zeta,\eta}^{(\alpha)}(1-\tau,r,z)q(z) p_\zeta^{(\alpha)}(\tau,z,s) \\
      & \quad + \int_{1/2}^1 d\tau \int_0^{s/2} dz\, z^{2\zeta } \,p_{\zeta,\eta}^{(\alpha)}(1-\tau,r,z)q(z) p_\zeta^{(\alpha)}(\tau,z,s).
    \end{split}
  \end{align}
  We leave the first summand untouched. To estimate the second summand, we use $p_\zeta^{(\alpha)}(\tau,z,s)\lesssim p_\zeta^{(\alpha)}(1,1,s)$ for $1/2<\tau<1$ and $0<z<s/2$ (by \eqref{eq:comparablealpha1} and \eqref{eq:comparablealpha2a} in Lemma~\ref{comparablealpha}), \eqref{eq:normalizedalpha}, Duhamel's formula \eqref{eq:duhamelclassictransformed}, and Proposition~\ref{bogprop32}, and obtain
  \begin{align}
    \label{eq:boglem42aux3}
    \begin{split}
      & \int_{1/2}^1 d\tau \int_0^{s/2} dz\, z^{2\zeta } \,p_{\zeta,\eta}^{(\alpha)}(1-\tau,r,z)q(z) p_\zeta^{(\alpha)}(\tau,z,s) \\
      & \quad \lesssim p_\zeta^{(\alpha)}(1,1,s) \int_{0}^1 d\tau \int_0^{\infty} dz\, z^{2\zeta } \,p_{\zeta,\eta}^{(\alpha)}(1-\tau,r,z)q(z) \\
      & \quad = p_\zeta^{(\alpha)}(1,1,s)\int_{0}^1d\tau \int_0^{\infty} dz\, z^{2\zeta } \,p_{\zeta,\eta}^{(\alpha)}(1-\tau,r,z)q(z)\int_0^\infty dy\, y^{2\zeta } p_\zeta^{(\alpha)}(\tau,z,y) \\
      & \quad = p_\zeta^{(\alpha)}(1,1,s)\int_0^\infty dy\, y^{2\zeta } \int_0^1d\tau \int_0^\infty dz\, z^{2\zeta }\, \,p_{\zeta,\eta}^{(\alpha)}(1-\tau,r,z)q(z)p_\zeta^{(\alpha)}(\tau,z,y) \\
      & \quad = p_\zeta^{(\alpha)}(1,1,s) \int_0^\infty dy\, y^{2\zeta } (\,p_{\zeta,\eta}^{(\alpha)}(1,r,y) - p_\zeta^{(\alpha)}(1,r,y))
      \leq M p_\zeta^{(\alpha)}(1,1,s) H(r).
    \end{split}
  \end{align}
  This concludes the proof of \eqref{eq:boglem42}.
\end{proof}

Lemma~\ref{boglem42} will be particularly important to study the region $r\vee s\leq 2(2\Psi_\zeta(\eta))^{1/\alpha}$. Although it could also be used to study $r\leq2(2\Psi_\zeta(\eta))^{1/\alpha}\leq s$, we will use different methods here.

\begin{lemma}
  \label{boglem43}
  Let $\zeta\in(-1/2,\infty)$, $\alpha\in(0,2)$, $\eta\in(0,(2\zeta+1-\alpha)/2]$. Then,  there is $C>0$ such that for all $s\geq2(2\Psi_\zeta(\eta))^{1/\alpha}$ and $r>0$,
  \begin{align}
    \label{eq:boglem43}
    p_{\zeta,\eta}^{(\alpha)}(1,r,s) \leq  C H(r) \,p_\zeta^{(\alpha)}(1,r,s).
  \end{align}
\end{lemma}

\begin{proof}
  We use \eqref{eq:bogprop32aux2} for $R=(2\Psi_\zeta(\eta))^{1/\alpha}$, to get 
  \begin{align}
  \label{eq:boglem43aux1}
    p_{\zeta,\eta}^{(\alpha)}(1,r,s)
    \leq 2\,p_\zeta^{(\alpha)}(1,r,s)
    + 2 \int\limits_0^1d\tau\int\limits_0^R dz\, z^{2\zeta}\,p_{\zeta,\eta}^{(\alpha)}(\tau,r,z) q(z) p_\zeta^{(\alpha)}(1-\tau,z,s).
  \end{align}
  It suffices to estimate the second summand by $p_\zeta^{(\alpha)}(1,r,s)H(r)$.
  For $0<\tau<1$, $0<z\leq s/2$, $s\geq 2R>0$, we estimate $p_\zeta^{(\alpha)}(1-\tau,z,s)\lesssim_\eta p_\zeta^{(\alpha)}(1,1,s)$ (\eqref{eq:comparablealpha3} in Lemma~\ref{comparablealpha}) and use \eqref{eq:boglem42aux4} to bound
  \begin{align}
  \label{eq:boglem43aux2}
    \int_0^1d\tau\int_0^R dz\, z^{2\zeta}\,p_{\zeta,\eta}^{(\alpha)}(\tau,r,z) q(z) p_\zeta^{(\alpha)}(1-\tau,z,s)
    \lesssim p_\zeta^{(\alpha)}(1,1,s)\, H(r).
  \end{align}
  Since for $r\leq2R$, we have $p_\zeta^{(\alpha)}(1,1,s) \sim p_\zeta^{(\alpha)}(1,r,s)$, we obtain \eqref{eq:boglem43} for $0\leq r\le 2R\leq s$. 
  On the other hand, by the symmetry of $p_{\zeta,\eta}^{(\alpha)}(1,r,s)$ in $r$ and $s$, \eqref{eq:boglem43aux1} and \eqref{eq:boglem43aux2} give
  \begin{align}
    p_{\zeta,\eta}^{(\alpha)}(1,r,s) \lesssim p_\zeta^{(\alpha)}(1,r,s) + \min\Big\{p_\zeta^{(\alpha)}(1,1,s)\, H(r),p_\zeta^{(\alpha)}(1,1,r)\, H(s)\Big\}, \quad r,s\geq 2R.
  \end{align}
  Since $H(s)\sim1\sim H(r)$ for all $r,s\geq 2R$, and $p_\zeta^{(\alpha)}(1,r,1)\wedge p_\zeta^{(\alpha)}(1,1,s)\lesssim_\eta \,p_\zeta^{(\alpha)}(1,r,s)$ for all $r,s>0$ (\eqref{eq:comparablealpha4} in Lemma~\ref{comparablealpha}), we get \eqref{eq:boglem43} for $r,s\geq 2R$.
  This ends the proof.
\end{proof}

Lemma \ref{boglem42}, together with \eqref{eq:heatkernelharmonic2} in Theorem \ref{heatkernelharmonic} and \eqref{eq:boglem41} in Corollary \ref{boglem41}, allows us to understand the remaining region $r,s\leq 2(2\Psi_\zeta(\eta))^{1/\alpha}$.

\begin{lemma}
  \label{boglem44}
  Let $\zeta\in(-1/2,\infty)$, $\alpha\in(0,2)$, $\eta\in(0,(2\zeta+1-\alpha)/2]$. Then, for every $\mu\in(0,\eta)$, there is a constant $C_\mu>0$ such that
  \begin{align}
    \label{eq:boglem44}
    p_{\zeta,\eta}^{(\alpha)}(1,r,s)
    \leq C_\mu \left( H(r)s^{\mu-(2\zeta+1)} \wedge H(s)r^{\mu-(2\zeta+1)}\right),
    \quad 0<r,s\leq2(2\Psi_\zeta(\eta))^{1/\alpha}.
  \end{align}
\end{lemma}

\begin{proof}
  By symmetry of $p_{\zeta,\eta}^{(\alpha)}(1,r,s)$, we may assume $r<s$ without loss of generality. We use \eqref{eq:boglem42} in Lemma \ref{boglem42}. 
  First note that $p_\zeta^{(\alpha)}(1,r,s)\vee p_\zeta^{(\alpha)}(1,1,s)\lesssim1$ for $r,s\lesssim 1$.
  Since $0<\mu\leq 2\zeta+1$, by \eqref{eq:pzetaest2} with $\lambda=0$ and $\sigma = 2\zeta+1-\mu$, we get $p_\zeta^{(\alpha)}(\tau,z,s) \lesssim s^{\mu-2\zeta-1}z^{-\mu}$. Therefore, by Lemma~\ref{boglem42},  \eqref{eq:heatkernelharmonic2} in Theorem~\ref{heatkernelharmonic} and~\eqref{eq:boglem41} in Corollary~\ref{boglem41}, we get  
  \begin{align}
    \begin{split}
     p_{\zeta,\eta}^{(\alpha)}(1,r,s)  & \quad \lesssim H(r) + s^{\mu-(2\zeta+1)} \int_0^{1}d\tau \int_0^{\infty}dz\, z^{2\zeta} \,p_{\zeta,\eta}^{(\alpha)}(1-\tau,r,z) z^{-\mu-\alpha} \\
     & \quad \leq H(r) + s^{\mu-(2\zeta+1)} \frac{1}{\Psi_\zeta(\eta)-\Psi_\zeta(\mu)} \int_0^\infty dz\, z^{2\zeta} \,p_{\zeta,\eta}^{(\alpha)}(1,r,z)z^{-\mu} \\
     & \quad \lesssim H(r) + s^{\mu-(2\zeta+1)} \cdot \frac{H(r)}{\Psi_\zeta(\eta)-\Psi_\zeta(\mu)} \lesssim H(r)s^{\mu-(2\zeta+1)},
    \end{split}
  \end{align} 
  where in the last line we used that $s\lesssim 1$ and $\mu<2\zeta+1$.
\end{proof}

\bigskip

Lemma \ref{boglem44} is not yet precise enough. The following two lemmas, together with Proposition~\ref{bogprop32} and~\eqref{eq:heatkernelharmonic1} in Theorem~\ref{heatkernelharmonic} will help us sharpen it. The improved estimates for $r,s\leq 2(2\Psi_\zeta(\eta))^{1/\alpha}$ are contained in Lemma~\ref{boglem47}.

\begin{lemma}
  \label{boglem45}
  Let $\zeta\in(-1/2,\infty)$,
  $\alpha\in(0,2)$,
  $\eta\in(0,(2\zeta+1-\alpha)/2]$, and $r>0$.
  Then, for all $\nu\in(\eta,2\zeta+1-\eta)$ and $\delta\in(0,\eta)$,
  \begin{align}
    \label{eq:boglem45scaled}
    \int_0^\infty dz\, z^{2\zeta} \,p_{\zeta,\eta}^{(\alpha)}(t,r,z)z^{-\nu}  
    \lesssim_{\nu,\delta} r^{-\nu}\left[1+\left(\frac{r}{t^{1/\alpha}}\right)^{-\delta}\right].
  \end{align}
  In particular,
  \begin{align}
    \label{eq:boglem45scaledcor}
    \int_0^\infty dz\, z^{2\zeta} \,p_{\zeta,\eta}^{(\alpha)}(t,r,z)z^{-\nu}
    \lesssim_{\nu,\delta} r^{-\nu}\left[1+r^{-\delta}\right], \quad t\in(0,1].
  \end{align}
\end{lemma}

This lemma can be seen as an extension of \eqref{eq:boglem41} in Corollary~\ref{bogcor38}, which treated all $\nu\in(0,\eta]$.

\begin{proof}
  By scaling, it suffices to consider $t=1$.
  We write $R=2(2\Psi_\zeta(\eta))^{1/\alpha}$ and distinguish between $r\leq R$ and $r>R$. We first consider $r\leq R$ and divide the integral in question into two parts.
  Let $\mu=\eta-\delta$ so that $\mu\in(0,\eta)$ and $\mu-\nu<0$.
  By Lemma~\ref{boglem44}, for $r,z\leq R$, we have
  \begin{align}
    \label{eq:boglem44cor}
    \begin{split}
      p_{\zeta,\eta}^{(\alpha)}(1,r,z)
      \lesssim r^{-\eta}z^{\mu-(2\zeta+1)} \one_{\{r\leq z\}} + r^{\mu-(2\zeta+1)}z^{-\eta}\one_{\{z\leq r\}}.
    \end{split}
  \end{align}
  
  Then, by Lemma \ref{boglem44} and \eqref{eq:boglem44cor},
  \begin{align}
    \begin{split}
      & \int_0^R dz\, z^{2\zeta} \,p_{\zeta,\eta}^{(\alpha)}(1,r,z) z^{-\nu} \\
      & \quad \lesssim_R \int_0^r dz\, z^{2\zeta-\nu}\cdot r^{\mu-(2\zeta+1)} \cdot z^{-\eta} + \int_r^R dz\, z^{2\zeta-\nu} \cdot r^{-\eta} \cdot z^{\mu-(2\zeta+1)} \\
      & \quad \lesssim r^{-\nu-\delta}, \quad r\leq R.
    \end{split}
  \end{align}
  For the complementary integral, we use that Proposition \ref{bogprop32} and $\nu+\delta>\eta$ imply
  \begin{align}
    \int_R^\infty dz\, z^{2\zeta}\,p_{\zeta,\eta}^{(\alpha)}(1,r,z)z^{-\nu}
    \leq R^{-\nu} \int_0^\infty dz\, z^{2\zeta}\,p_{\zeta,\eta}^{(\alpha)}(1,r,z)
    \lesssim_R H(r) \lesssim_R r^{-\nu-\delta}, \quad r\leq R.
  \end{align}
  This concludes the proof for $r\leq R$.

  For $r\geq R$ we split the $z$-integration at $z=r/2$. For $z>r/2$, we get
  \begin{align}
    \begin{split}
      \int_{r/2}^\infty dz\, z^{2\zeta }\,p_{\zeta,\eta}^{(\alpha)}(1,r,z)z^{-\nu}
      \leq r^{-\nu} \int_0^\infty dz\, z^{2\zeta }\,p_{\zeta,\eta}^{(\alpha)}(1,r,z)
      \lesssim_R r^{-\nu},\quad r\geq R.
    \end{split}
  \end{align}
  For $z<r/2$, we get, using Lemma~\ref{boglem43} and $p_\zeta^{(\alpha)}(1,r,z)\lesssim r^{-(2\zeta+1+\alpha)}$ (Proposition~\ref{heatkernelalpha1subordinatedboundsfinal}),
  \begin{align}
    \begin{split}
      & \int_0^{r/2} dz\, z^{2\zeta}\,p_{\zeta,\eta}^{(\alpha)}(1,r,z)z^{-\nu}
        \lesssim \int_0^{r/2}  dz\, z^{2\zeta-\nu} p_\zeta^{(\alpha)}(1,r,z) H(z)
        \lesssim r^{-\nu-\alpha}.
    \end{split}
  \end{align}
  The proof is concluded.
\end{proof}

Let us introduce
\begin{align}\index{$H_\nu(r)$}
  \label{eq:defbighbeta}
  H_\nu(r) = 1 + r^{-\nu}, \quad r, \nu>0.
\end{align}
Thus, $H_\eta(r)=H(r)$, with $H(r)$ as in \eqref{eq:defbigh}.
By Proposition~\ref{bogprop32}, \eqref{eq:boglem41} in Corollary~\ref{boglem41}, and Lemma~\ref{boglem45}, we conclude that for all $\nu\in[0,2\zeta+1-\eta)$ and $\delta\in(0,\eta)$,
\begin{align}
  \label{eq:boglem45consequence}
  \int_0^\infty ds\, s^{2\zeta } \,p_{\zeta,\eta}^{(\alpha)}(t,r,s) H_\nu(s)
  \lesssim_\delta H_{(\nu+\delta)\vee\eta}(r)\,, \quad t\in(0,1]\,, \quad r>0\,.
\end{align}

The following result improves and extends this estimate to all $t>0$.

\begin{lemma}
  \label{boglem46}
  Let $\zeta\in(-1/2,\infty)$,
  $\alpha\in(0,2)$,
  $\eta\in(0,(2\zeta+1-\alpha)/2]$. Let $\nu\in[0,2\zeta+1-\eta)$. Then,
  \begin{align}
  \label{eq:boglem46pre}
    \int_0^\infty ds\, s^{2\zeta} \,p_{\zeta,\eta}^{(\alpha)}(t,r,s) H_\nu(s/t^{1/\alpha})
    \lesssim_\nu H(r/t^{1/\alpha}), \quad t>0, \quad r>0.
  \end{align}
  In particular,
  \begin{align}
    \label{eq:boglem46}
    \int_0^\infty ds\, s^{2\zeta} \,p_{\zeta,\eta}^{(\alpha)}(t,r,s) H_\nu(s)
    \lesssim_\nu t^{-\nu/\alpha} H(r/t^{1/\alpha}), \quad t\in(0,1], \quad r>0.
  \end{align}
\end{lemma}

\begin{proof}
  By scaling \eqref{eq:scalingalphahardy} of $p_{\zeta,\eta}^{(\alpha)}$, it suffices to prove \eqref{eq:boglem46pre} for $t=1$.
  If $\nu\leq\eta$, then the claim follows from \eqref{eq:boglem41} in Corollary~\ref{boglem41}. 

  Thus, suppose $\nu=\eta+\xi\alpha$ with $\xi\in(0,(2\zeta+1-2\eta)/\alpha)$. We distinguish between $\xi\in(0,1)$ and $\xi\geq1$.

  If $\xi<1$, then \eqref{eq:heatkernelharmonic2} in Theorem \ref{heatkernelharmonic} and \eqref{eq:boglem41} in Corollary \ref{boglem41} imply for $\eta<\mu\leq\eta+\alpha$,
  \begin{align}
    \label{eq:boglem46aux1}
    \begin{split}
      \int_0^1 d\tau \int_0^\infty dz\, z^{2\zeta} \,p_{\zeta,\eta}^{(\alpha)}(\tau,r,z) H_\mu(z)
      & \lesssim H(r) + \int_0^\infty ds\, s^{2\zeta} \,p_{\zeta,\eta}^{(\alpha)}(1,r,s)s^{-\mu+\alpha} \\
      & \lesssim H(r), \quad r>0.
    \end{split}
  \end{align}
  Let $0<\epsilon<(1-\xi)\alpha\wedge\eta$ so that $\nu+\epsilon=\eta+\xi\alpha+\epsilon<\eta+\alpha$. By the semigroup property, $\int_0^1 d\tau=1$, and Formulae~\eqref{eq:boglem45consequence} and \eqref{eq:boglem46aux1}, if $\nu+\epsilon\in(\eta,\eta+\alpha)$ then there is $C>0$ such that
  \begin{align}
    \label{eq:boglem46aux2}
    \begin{split}
      & \int_0^\infty ds\, s^{2\zeta} \,p_{\zeta,\eta}^{(\alpha)}(1,r,s) H_\nu(s) \\
      & \quad = \int_0^1 d\tau \int_0^\infty dz\, z^{2\zeta} \,p_{\zeta,\eta}^{(\alpha)}(\tau,r,z)\int_0^\infty ds\, s^{2\zeta} \,p_{\zeta,\eta}^{(\alpha)}(1-\tau,z,s) H_\nu(s)\\
      & \quad \lesssim \int_0^1 d\tau \int_0^\infty dz\, z^{2\zeta} \,p_{\zeta,\eta}^{(\alpha)}(\tau,r,z) H_{\nu+\epsilon}(z)
        \leq CH(r), \quad r>0,
    \end{split}
  \end{align}
  as desired. This concludes the case $\xi\in(0,1)$.

  Now let $\xi\in[1,(2\zeta+1-2\eta)/\alpha)$. By Corollary \ref{bogcor35}, if $\eta+\alpha<\mu<2\zeta+1-\eta$ then
  \begin{align}
    \label{eq:boglem46aux3}
    \int_0^1 d\tau \int_0^\infty dz\, z^{2\zeta} \,p_{\zeta,\eta}^{(\alpha)}(\tau,r,z) H_\mu(z)
    \lesssim H_{\mu-\alpha}(r), \quad r>0.
  \end{align}
  We fix $\epsilon\in(0,1)$ such that $\delta:=(1-\epsilon)\alpha<2\zeta+1-\nu-\eta$, and $\delta<\eta$. By the semigroup property, $\int_0^1d\tau=1$, \eqref{eq:boglem45consequence}, and \eqref{eq:boglem46aux3}, for $\eta+\alpha\leq\mu\leq\nu$, we have, if $\eta+\alpha\leq\mu\leq\nu$,
  \begin{align}
    \label{eq:boglem46aux4}
    \begin{split}
      & \int_0^\infty ds\, s^{2\zeta} \,p_{\zeta,\eta}^{(\alpha)}(1,r,s) H_\mu(s) \\
      & \quad = \int_0^1 d\tau \int_0^\infty dz\, z^{2\zeta} \,p_{\zeta,\eta}^{(\alpha)}(\tau,r,z) \int_0^\infty ds\, s^{2\zeta} \,p_{\zeta,\eta}^{(\alpha)}(1-\tau,z,s) H_\mu(s) \\
      & \quad \lesssim \int_0^1 d\tau \int_0^\infty dz\, z^{2\zeta} \,p_{\zeta,\eta}^{(\alpha)}(\tau,r,z) H_{\mu+\delta}(z)
      \lesssim H_{\mu-\epsilon\alpha}(r), \quad r>0.
    \end{split}
  \end{align}
  We choose $n\in\N$ such that $\xi\alpha\in[\epsilon(n\alpha-1),\epsilon(n\alpha-1)+\alpha]\cap[\alpha,2\zeta+1-2\eta)$.
  By the semigroup property, \eqref{eq:boglem46aux4}, and \eqref{eq:boglem46aux2}, if $\eta+(\xi-n\epsilon)\alpha+\epsilon \in (\eta,\eta+\alpha)$, then
  \begin{align}
    \begin{split}
      & \int_0^\infty ds\, s^{2\zeta} p_{\zeta,\eta}^{(\alpha)}(n+1,r,s) H_\nu(s) \\
      & \quad = \int_0^\infty dz_1\, z_1^{2\zeta}\,...\, \int_0^\infty dz_n\, z_n^{2\zeta} \int_0^\infty ds\, s^{2\zeta} p_{\zeta,\eta}^{(\alpha)}(1,r,z_1)\, ...\, p_{\zeta,\eta}^{(\alpha)}(1,z_n,s) H_\nu(s) \\
      & \quad \leq C^{n} \int_0^\infty dz_1\, z_1^{2\zeta} p_{\zeta,\eta}^{(\alpha)}(1,r,z_1) H_{\nu-n\epsilon\alpha}(z_1) \\
      & \quad = C^n \int_0^\infty dz_1 z_1^{2\zeta} p_{\zeta,\eta}^{(\alpha)}(1,r,z_1) H_{\eta+(\xi-n\epsilon)\alpha}(z_1) \\
      & \quad \leq C^{n+1} H(r),
      \quad r>0,
    \end{split}
  \end{align}
  recalling $\nu=\eta+\xi\alpha$. This ends the proof in view of $p_{\zeta,\eta}^{(\alpha)}(n+1,r,s) \sim_n \,p_{\zeta,\eta}^{(\alpha)}(1,r,s)$.
\end{proof}

Lemmas \ref{boglem42} and \ref{boglem46} allow us to improve the estimate for $r,s\leq 2(2\Psi_\zeta(\eta))^{1/\alpha}$ from Lemma~\ref{boglem44}.

\begin{lemma}
  \label{boglem47}
  Let $\zeta\in(-1/2,\infty)$, $\alpha\in(0,2)$, $\eta\in(0,(2\zeta+1-\alpha)/2]$. Let $0<r,s\leq2(2\Psi_\zeta(\eta))^{1/\alpha}$. Then, for each $\delta\in(0,\eta)$,
  \begin{align}
    p_{\zeta,\eta}^{(\alpha)}(1,r,s)
    \lesssim_\delta \left(H(r)s^{-\eta-\delta} \wedge H(s) r^{-\eta-\delta} \right).
  \end{align}
\end{lemma}

\begin{proof} 
  We will use Lemma~\ref{boglem42} and the symmetry of $p_{\zeta,\eta}^{(\alpha)}$. The first line on the right-hand side of~\eqref{eq:boglem42} is clearly bounded by $H(r)s^{-\eta-\delta}$. Now, we estimate the second line on the right-hand side of ~\eqref{eq:boglem42}. Note that for $\tau\in(1/2,1)$, we have $p_\zeta^{(\alpha)}(\tau,z,s)\lesssim 1$ (see~\eqref{eq:heatkernelalpha1weightedsubordinatedboundsfinal}). Hence, by \eqref{eq:boglem42aux4}, we have
 \begin{align}
    \label{eq:boglem47aux3}
    \begin{split}
      & \int_{1/2}^1 d\tau \int_{s/2}^\infty dz\, z^{2\zeta} \,p_{\zeta,\eta}^{(\alpha)}(1-\tau,r,z)q(z) p_\zeta^{(\alpha)}(\tau,z,s) \\
      & \quad \lesssim \int_0^1 d\tau \int_{s/2}^\infty dz\, z^{2\zeta} \,p_{\zeta,\eta}^{(\alpha)}(1-\tau,r,z)q(z) 
        \lesssim H(r)
        \lesssim H(r)s^{-\eta-\delta}.
    \end{split}
  \end{align}
Next, let $\varepsilon \in (0,\alpha \land \delta)$. By \eqref{eq:pzetaest2} with $\sigma = \eta+\delta$ and $\lambda =\alpha-\varepsilon$ and Lemma~\ref{boglem46}, we get 
  \begin{align}
    \label{eq:boglem47aux2}
    \begin{split}
      & \int_0^{1/2}d\tau \int_0^{s/2} dz\, z^{2\zeta} \,p_{\zeta,\eta}^{(\alpha)}(1-\tau,r,z)q(z) p_\zeta^{(\alpha)}(\tau,z,s) \\
			& \qquad + \int_{g(s)}^{1/2} d\tau \int_{s/2}^\infty dz\, z^{2\zeta} \,p_{\zeta,\eta}^{(\alpha)}(1-\tau,r,z)q(z) p_\zeta^{(\alpha)}(\tau,z,s) \\
      & \quad \lesssim s^{-\eta-\delta}\int_0^{1/2}d\tau\, \tau^{(\varepsilon-\alpha)/\alpha} \int_0^{\infty} dz\, z^{2\zeta} \,p_{\zeta,\eta}^{(\alpha)}(1-\tau,r,z) z^{-(2\zeta+1 -\eta-\delta+\varepsilon)}  \\
      & \quad \lesssim s^{-\eta-\delta}\int_0^{1/2} d\tau\, \tau^{(\varepsilon-\alpha)/\alpha}H(r/(1-\tau)^{1/\alpha}) \cdot (1-\tau)^{-(2\zeta+1 -\eta-\delta+\varepsilon)/\alpha} \\ 
      & \quad \lesssim H(r) s^{-\eta-\delta}.
    \end{split}
  \end{align}
  This concludes the proof. 
\end{proof}

Combining the previous results yields the desired upper bounds in Theorem~\ref{mainresultgen}.
\begin{lemma}
  \label{boglem48}
  Let $\zeta\in(-1/2,\infty)$, $\alpha\in(0,2)$, $\eta\in(0,(2\zeta+1-\alpha)/2]$. Then,
  \begin{align}
    \label{eq:boglem48}
    p_{\zeta,\eta}^{(\alpha)}(t,r,s) \lesssim p_\zeta^{(\alpha)}(t,r,s) H(t,r)H(t,s), \quad r,s,t>0.
  \end{align}
\end{lemma}

\begin{proof}
  By scaling (Formula \eqref{eq:scalingalphahardy}), it suffices to consider $t=1$. Let $R=2(2\Psi_\zeta(\eta))^{1/\alpha}$. Fix $\delta\in(0,(2\zeta+1)/2-\eta)$. By the semigroup property for $p_{\zeta,\eta}^{(\alpha)}$ and $p_\zeta^{(\alpha)}$, Lemmas~\ref{boglem47} and~\ref{boglem43}, we have for $0<r,s\leq R$,
  \begin{align}
    \label{eq:boglem48aux1}
    \begin{split}
      p_{\zeta,\eta}^{(\alpha)}(2,r,s)
      & = \int_0^\infty dz\, z^{2\zeta} \,p_{\zeta,\eta}^{(\alpha)}(1,r,z) \,p_{\zeta,\eta}^{(\alpha)}(1,z,s) \\
      & \lesssim H(r)H(s) \int_0^R dz\, z^{2\zeta}\cdot z^{-2(\eta+\delta)} \\
      & \qquad + H(r)H(s) \int_R^\infty dz\, z^{2\zeta} p_\zeta^{(\alpha)}(1,r,z) p_\zeta^{(\alpha)}(1,z,s) \\
      & \lesssim_R H(r)H(s) (1+ p_\zeta^{(\alpha)}(2,r,s)), \quad r,s\leq R.
    \end{split}
  \end{align}
  Since $p_\zeta^{(\alpha)}(2,r,s)\sim \,p_\zeta^{(\alpha)}(1,r,s)\sim_R 1$ by \eqref{eq:heatkernelalpha1weightedsubordinatedboundsfinal}, this concludes the proof of $r,s\leq R$.
  Now let $r>0$ and $s\geq R$. By Lemma~\ref{boglem43},
  \begin{align}
    \label{eq:boglem48aux2}
    p_{\zeta,\eta}^{(\alpha)}(1,r,s)
    \lesssim \,p_\zeta^{(\alpha)}(1,r,s)H(r)
    \sim \,p_\zeta^{(\alpha)}(1,r,s) H(r) H(s), \quad s>R.
  \end{align}
  Combining \eqref{eq:boglem48aux1} and \eqref{eq:boglem48aux2}, and using the symmetry of $p_{\zeta,\eta}^{(\alpha)}$ yields the claim.
\end{proof}

\subsubsection{Lower bound}

We prove the lower bound in \eqref{eq:mainresultgen}. By scaling (Formula \eqref{eq:scalingalphahardy}), it suffices to consider $t=1$. Let $q(t,r,s) := p_{\zeta,\eta}^{(\alpha)}(t,r,s)/(H(r) H(s))$ and $\mu(ds) = H(s)^2s^{2\zeta}\,ds$. By Doob conditioning, $q(t,r,s)$ is an integral kernel of a semigroup in $L^2(\R_+,\mu)$. By \eqref{eq:heatkernelharmonic1} in Theorem~\ref{heatkernelharmonic}, \eqref{eq:trivialupperbound}, and Corollary~\ref{bogcor38},
\begin{align}
  \label{eq:boglowbd0}
  1 \leq \int_0^\infty q(1,r,s) \mu(ds) \leq M+1, \quad r>0.
\end{align}
By the upper bound for $p_{\zeta,\eta}^{(\alpha)}$ in Lemma \ref{boglem48},
\begin{align}
  q(1,r,s) \lesssim p_\zeta^{(\alpha)}(1,r,s), \quad r,s>0.
\end{align}
This estimate and the bounds \eqref{eq:heatkernelalpha1weightedsubordinatedboundsfinal} show that there is $R>2$ such that
\begin{align}
  \int_R^\infty q(1,r,s)\mu(ds) \leq \frac14, \quad 0<r\leq1.
\end{align}
Furthermore, there is $0<\rho<1$ such that
\begin{align}
  \int_0^\rho q(1,r,s)\mu(ds) \leq \frac14, \quad r>0.
\end{align}
Combining the last two estimates shows
\begin{align}
  \label{eq:boglowbd1}
  \int_\rho^R q(1,r,s)\mu(ds) \geq \frac12, \quad 0<r<1.
\end{align}

We are now ready to prove the lower bound in \eqref{eq:mainresultgen}. We consider the regions $0<r<1<s$, and $r,s\lessgtr1$ separately and start with the former. We first note that
\begin{align}
  \label{eq:boglowbd2aux2}
  \begin{split}
    p_\zeta^{(\alpha)}(1,z,s)
    & \sim \frac{1}{|z-s|^{1+\alpha}(z+s)^{2\zeta} + (1+z+s)^{2\zeta}} \\
    & \geq \frac{1}{R^{2\zeta+1+\alpha} + s^{2\zeta+1+\alpha} + (1+s+R)^{2\zeta}} \\
    & \sim p_\zeta^{(\alpha)}(1,1,s),
    \quad z\in(\rho,R), \ s>\rho, 
  \end{split}
\end{align}
which follows from \eqref{eq:heatkernelalpha1weightedsubordinatedboundsfinal}. By the semigroup property, \eqref{eq:boglowbd1}, and \eqref{eq:boglowbd2aux2}, we have for $0<r\leq1$ and $s\geq\rho$,
\begin{align}
  \label{eq:boglowbd2}
  \begin{split}
    q(2,r,s)
    & \geq \int_\rho^R q(1,r,z)q(1,z,s) \mu(dz)
    \geq \int_\rho^R q(1,r,z)\frac{p_\zeta^{(\alpha)}(1,z,s)}{H(\rho)^2}\mu(dz) \\
    & \gtrsim p_\zeta^{(\alpha)}(1,1,s) \int_\rho^R q(1,r,z)\mu(dz) \\
    & \gtrsim p_\zeta^{(\alpha)}(1,1,s),
    \quad 0<r\leq1, \ s\geq\rho.
  \end{split}
\end{align}
Similarly, we get
\begin{align}
  q(3,r,s)
  \gtrsim p_\zeta^{(\alpha)}(1,1,s)
  \gtrsim p_\zeta^{(\alpha)}(3,r,s), \quad 0<r<1<s,
\end{align}
where the second estimate follows from~\eqref{eq:comparablealpha3} if $s\geq2$ and from~\eqref{eq:easybounds2} and~\eqref{eq:heatkernelalpha1weightedsubordinatedboundsfinal} if $s\leq2$.

Now let $r,s\leq1$. By the semigroup property, we obtain
\begin{align}
  \label{eq:boglowbd3}
  \begin{split}
    q(3,r,s)
    & \geq \int_\rho^R q(1,r,z)q(2,z,s)\mu(dz)
    \gtrsim \int_\rho^R q(1,r,z) p_\zeta^{(\alpha)}(1,z,1) \mu(dz) \\
    & \gtrsim_{\rho,R} p_\zeta^{(\alpha)}(1,1,s)\int_\rho^R q(1,r,z) \mu(dz)
    \sim p_\zeta^{(\alpha)}(1,1,s)
    \sim p_\zeta^{(\alpha)}(1,r,s), \quad r,s\leq1,
  \end{split}
\end{align}
where we used~\eqref{eq:boglowbd2} in the second estimate for $q(2,z,s)$. The third and the final estimates in~\eqref{eq:boglowbd3} follow from~\eqref{eq:boglowbd1} and~\eqref{eq:easybounds2} and~\eqref{eq:heatkernelalpha1weightedsubordinatedboundsfinal}, since $r,s,z\lesssim1$.

Finally, if $r,s\geq1$, then $q(3,r,s)\geq H(1)^2\,p_{\zeta,\eta}^{(\alpha)}(3,r,s) \gtrsim p_\zeta^{(\alpha)}(3,r,s)$. Combining all the cases shows
\begin{align}
  q(3,r,s) \gtrsim p_\zeta^{(\alpha)}(3,r,s), \quad r,s>0.
\end{align}
By the definition of $q(t,r,s) = p_{\zeta,\eta}^{(\alpha)}(t,r,s)/(H(r) H(s))$, the claimed lower bound in Theorem~\ref{mainresultgen} is proved.\qed

\subsubsection{Continuity}

To show that $p_{\zeta,\eta}^{(\alpha)}(t,r,s)$ is jointly continuous, we will use Duhamel's formula and the known joint continuity of $p_{\zeta}^{(\alpha)}(t,r,s)$.
To that end, we first prepare the following estimate.

\begin{lemma}
  \label{lem:GTest}
  Let $\zeta\in(-1/2,\infty)$, $\alpha\in(0,2)$, $\eta\in[0,(2\zeta+1-\alpha)/2]$, $T>0$, and 
  \begin{align}\index{$G_\eta^{(T)}(t,r,s)$}
  \label{eq:defGtauest}
   G_{\eta}^{(T)}(t,r,\tau,s) := \int_0^T dz\, z^{2\zeta} \cdot z^{-\alpha-\eta}  p_\zeta^{(\alpha)}(t,r,z)\left(\tau^{1/\alpha}+s+z \right)^{-2\zeta}.
\end{align}
Then, 
  \begin{align}\label{eq:GTest1}
      G_{\eta}^{(T)}(t,r,\tau,s) &\lesssim r^{-\eta-\alpha} s^{-2\zeta}, &&\quad \zeta\ge0,\; r,s,t,\tau>0, \\
      G_{\eta}^{(T)}(t,r,\tau,s) &\lesssim r^{-\eta-\alpha} (\tau^{1/\alpha}+T)^{-2\zeta},& & \quad \zeta\in(-1/2,0),\; r,t,\tau>0, s\in(0,T). \label{eq:GTest2}
  \end{align}
\end{lemma}

\begin{proof}
  Let $\zeta \ge 0$. By Lemma \ref{integralana1},
  \begin{align*}
    G_{ \eta}^{(T)}(t,r,\tau,s) \le s^{-2\zeta} \int_0^\infty dz\, z^{2\zeta} \cdot z^{-\alpha-\eta}  p_\zeta^{(\alpha)}(t,r,z) \lesssim s^{-2\zeta} r^{-\eta-\alpha}.
  \end{align*}
  If $\zeta \in (-1/2,0)$, then
  \begin{align*}
  G_{\eta}^{(T)}(t,r,\tau,s) \le \int_0^\infty dz\, z^{2\zeta} \cdot z^{-\alpha-\eta}  p_\zeta^{(\alpha)}(t,r,z) (\tau^{1/\alpha} +2T)^{-2\zeta} \lesssim (\tau^{1/\alpha} +2T)^{-2\zeta} r^{-\eta-\alpha}.
  \end{align*}
  This concludes the proof.
\end{proof}

\begin{lemma}
  \label{boglem49}
  Let $\zeta\in (-1/2,\infty)$, $\alpha\in(0,2)$, and $\eta\in[0,(2\zeta+1-\alpha)/2]$. Then, for every $r,t>0$, the function $(0,\infty)\ni s\mapsto p_{\zeta,\eta}^{(\alpha)}(t,r,s)$ is continuous.
\end{lemma}

\begin{proof}
  The continuity of $p_\zeta^{(\alpha)}(t,r,s)$ is well known, hence let $\eta>0$ from now on. By scaling, it suffices to consider $t=1$. Fix $r,s>0$ and let $w>0$. Then,
  \begin{align}
    \label{eq:boglem49aux0}
    \begin{split}
      & p_{\zeta,\eta}^{(\alpha)}(1,r,s) - p_{\zeta,\eta}^{(\alpha)}(1,r,w) \\
      & \quad = p_\zeta^{(\alpha)}(1,r,s) - p_\zeta^{(\alpha)}(1,r,w) \\
      & \qquad + \int_0^1 d\tau \int_0^\infty dz\, z^{2\zeta} \,p_{\zeta,\eta}^{(\alpha)}(1-\tau,r,z)q(z)\left[p_\zeta^{(\alpha)}(\tau,z,s)- p_\zeta^{(\alpha)}(\tau,z,w)\right].
    \end{split}
  \end{align}
  We claim that the integral vanishes as $w$ converges to $s$. Indeed, let $\epsilon\in(0,1/2)$.
  Then, for $\tau\in(\epsilon,1)$ and $|w-s|$ sufficiently small, we have $p_\zeta^{(\alpha)}(\tau,z,w)\sim p_\zeta^{(\alpha)}(\tau,z,s)$ by \eqref{eq:heatkernelalpha1weightedsubordinatedboundsfinal}. Thus, by dominated convergence,
  \begin{align}
    \lim_{w\to s}\int_\epsilon^1 d\tau \int_0^\infty dz\, z^{2\zeta} \,p_{\zeta,\eta}^{(\alpha)}(1-\tau,r,z)q(z) \left[p_\zeta^{(\alpha)}(\tau,z,s) - p_\zeta^{(\alpha)}(\tau,z,w)\right] = 0.
  \end{align}
  Furthermore, by the upper heat kernel bounds in Theorem~\ref{mainresultgen} and the 3G inequality in Lemma~\ref{lem:3Gineq}, together with Lemma \ref{lem:GTest}, we have, for any $\xi\in(s/2,2s)$ and $\epsilon\in(0,1/2)$,
  \begin{align}
    \label{eq:boglem49aux1}
    \begin{split}
      & \int_0^\epsilon d\tau \int_0^\infty dz\, z^{2\zeta} \,p_{\zeta,\eta}^{(\alpha)}(1-\tau,r,z)q(z) p_\zeta^{(\alpha)}(\tau,z,\xi) \\
      & \quad \lesssim \int_0^\epsilon d\tau \int_0^\infty dz\, z^{2\zeta} p_\zeta^{(\alpha)}(1-\tau,r,z) H(1-\tau,r)H(1-\tau,z)q(z) p_\zeta^{(\alpha)}(\tau,z,\xi) \\
     & \quad \lesssim p^{(\alpha)}(1,r,s) H(r) \int_0^\epsilon d\tau \left[G_\eta^{2(r\vee\xi\vee1)}(1-\tau,r,\tau,\xi) + G_0^{2(r\vee\xi\vee1)}(1-\tau,r,\tau,\xi) \right. \\ & \qquad\qquad\qquad\qquad\qquad\qquad\quad \left. + G_\eta^{2(r\vee\xi\vee1)}(\tau,\xi,1-\tau,r) + G_0^{2(r\vee\xi\vee1)}(\tau,\xi,1-\tau,r)\right] \\
      & \qquad + H(r) \int_0^\epsilon d\tau \int_{2(r\vee \xi \vee 1)}^\infty dz\, z^{-2\zeta -3\alpha-2} \cdot (1+z^{-\eta}) \\
      & \quad \lesssim \epsilon H(r)p^{(\alpha)}(1,r,\xi)\left[(r^{-\eta-\alpha}+r^{-\alpha})\xi^{-2\zeta} + (\xi^{-\eta-\alpha}+\xi^{-\alpha})r^{-2\zeta}\right]\one_{\zeta\geq0} \\
      & \qquad + \epsilon H(r) p^{(\alpha)}(1,r,\xi) \left[\frac{r^{-\eta-\alpha}+r^{-\alpha}}{(1+r+\xi)^{2\zeta}} + \frac{\xi^{-\eta-\alpha}+\xi^{-\alpha}}{((1-\tau)^{1/\alpha}+1+r+\xi)^{2\zeta}}\right]\one_{\zeta\in(-\frac12,0)} \\
      & \qquad + \epsilon H(r) (r+\xi+1)^{-2\zeta-3\alpha-1} \\
      & \quad \lesssim \epsilon H(r) c(r,s),
    \end{split}
  \end{align}
  where $c(r,s)=c_{\zeta,\eta,\alpha}(r,s)$ depends only on $r,s>0$ and the fixed parameters $\zeta,\alpha$, and $\eta$. Using \eqref{eq:boglem49aux1}, we see that
  $$
  \lim_{\epsilon\to0}\int_0^\epsilon d\tau \int_0^\infty dz\, z^{2\zeta} \,p_{\zeta,\eta}^{(\alpha)}(1-\tau,r,z)q(z)\left[p_\zeta^{(\alpha)}(\tau,z,s)- p_\zeta^{(\alpha)}(\tau,z,w)\right] = 0,
  $$
  uniformly in $w\in(s/2,2s)$.
  This yields the result.
\end{proof}

\begin{lemma}
  \label{boglem410}
  Let $\zeta\in (-1/2,\infty)$, $\alpha\in(0,2)$, and $\eta\in(0,(2\zeta+1-\alpha)/2]$. Then, $p_{\zeta,\eta}^{(\alpha)}(t,r,s)$ is jointly continuous in $t,r,s>0$.
\end{lemma}

\begin{proof}
  By scaling, it suffices to consider $t=1$. Fix $r,s>0$ and let $\tilde r$ and $\tilde s$ be close to $r$ and $s$ respectively. As in Lemma \ref{boglem49}, it suffices to show
  \begin{align}
    \int_0^1 d\tau \int_0^\infty dz\, z^{2\zeta} \left| p_{\zeta,\eta}^{(\alpha)}(1-\tau,\tilde r,z) p_\zeta^{(\alpha)}(\tau,z,\tilde s) - p_{\zeta,\eta}^{(\alpha)}(1-\tau,r,z) p_\zeta^{(\alpha)}(\tau,z,s)\right| q(z) \to 0
  \end{align}
  as $\tilde r\to r$ and $\tilde s\to s$. In addition to \eqref{eq:boglem49aux1} we get (with the same arguments)
  \begin{align}
    \label{eq:boglem410aux1}
    \begin{split}
      & \int_{1-\epsilon}^1 d\tau \int_0^\infty dz\, z^{2\zeta} \,p_{\zeta,\eta}^{(\alpha)}(1-\tau,r,z)q(z) p_\zeta^{(\alpha)}(\tau,z,s) \\
      & \quad \lesssim \int_0^\epsilon d\tau \int_0^\infty dz\, z^{2\zeta} p_\zeta^{(\alpha)}(\tau,r,z) H(\tau,r)H(\tau,z) q(z) p_\zeta^{(\alpha)}(1-\tau,z,s) \leq \epsilon c(r,s)
    \end{split}
  \end{align}
  for a constant $c(r,s)>0$ independent of $\varepsilon$. 
  By \eqref{eq:boglem49aux1}--\eqref{eq:boglem410aux1}, we get
  \begin{align*}
    \small
    \begin{split}
      & \int_0^1 d\tau \int_0^\infty dz\, z^{2\zeta} \left| p_{\zeta,\eta}^{(\alpha)}(1-\tau,\tilde r,z) p_\zeta^{(\alpha)}(\tau,z,\tilde s) - \,p_{\zeta,\eta}^{(\alpha)}(1-\tau,r,z)p_\zeta^{(\alpha)}(\tau,z,s)\right| q(z) \\
      & \quad \leq \int_0^\epsilon d\tau \int_0^\infty dz\, z^{2\zeta}\left[ p_{\zeta,\eta}^{(\alpha)}(1-\tau,\tilde r,z) p_\zeta^{(\alpha)}(\tau,z,\tilde s) + p_{\zeta,\eta}^{(\alpha)}(1-\tau,r,z) p_\zeta^{(\alpha)}(\tau,z,s)\right] q(z) \\
      & \qquad + \int_{1-\epsilon}^1 d\tau \int_0^\infty dz\, z^{2\zeta} \left[ p_{\zeta,\eta}^{(\alpha)}(1-\tau,\tilde r,z) p_\zeta^{(\alpha)}(\tau,z,\tilde s) + p_{\zeta,\eta}^{(\alpha)}(1-\tau,r,z) p_\zeta^{(\alpha)}(\tau,z,s)\right] q(z) \\
      & \qquad + \int_\epsilon^{1-\epsilon} d\tau \int_0^\infty dz\, z^{2\zeta}\left| p_{\zeta,\eta}^{(\alpha)}(1-\tau,\tilde r,z) p_\zeta^{(\alpha)}(\tau,z,\tilde s) - p_{\zeta,\eta}^{(\alpha)}(1-\tau,r,z) p_\zeta^{(\alpha)}(\tau,z,s)\right| q(z) \\
      & \quad \leq \epsilon c(r,s) + \int_\epsilon^{1-\epsilon} \int_0^\infty dz\, z^{2\zeta} \left| p_{\zeta,\eta}^{(\alpha)}(1-\tau,\tilde r,z) p_\zeta^{(\alpha)}(\tau,z,\tilde s) - p_{\zeta,\eta}^{(\alpha)}(1-\tau,r,z)p_\zeta^{(\alpha)}(\tau,z,s) \right|.
    \end{split}
  \end{align*}
  By the upper and lower heat kernel estimates in Theorem~\ref{mainresultgen}, we have, for $\epsilon<\tau<1-\epsilon$, the estimates
  $p_{\zeta,\eta}^{(\alpha)}(1-\tau,\tilde r,z) \sim p_{\zeta,\eta}^{(\alpha)}(1-\tau,r,z)$ and
  $p_\zeta^{(\alpha)}(\tau,z,\tilde s) \sim p_\zeta^{(\alpha)}(\tau,z,s)$.
  By Lemma~\ref{boglem49} and dominated convergence, the last integral is arbitrarily small. This concludes the proof of Lemma~\ref{boglem410} and in particular that of the continuity statement in Theorem~\ref{mainresultgen}.
\end{proof}

\subsubsection{Blowup}
\label{ss:blowup}

We now prove Corollary~\ref{blowupcor} for all $\alpha\in(0,2]$ using Theorem~\ref{mainresultgen}. Note that the pointwise bounds stated in Theorem~\ref{mainresultgen} extend to $\alpha=2$ in view of Remark~\ref{remarksmainresult1}.

\begin{proof}[Proof of Corollary~\ref{blowupcor}]
  Let $\eta_*:=(2\zeta+1-\alpha)/2$ be the parameter corresponding to the critical coupling constant $\kappa_{\rm c}(\zeta,\alpha)$ and recall our assumption $\kappa>\kappa_{\rm c}(\zeta,\alpha)$. 
According to  the Appendix~\ref{s:constructionnegschrodperturbation},
for $t,r,s>0$, we define the Schr\"odinger perturbation $\tilde p_\kappa(t,r,s)$ 
of $p_\zeta^{(\alpha)}(t,r,s)$ by $\kappa/r^\alpha$:
  \begin{align*}
    \begin{split}
      \tilde p_{\kappa}(t,r,s)
      & := \sum_{n\geq0} p_\kappa^{(n)}(t,r,s), \quad \text{where} \quad
        p_\kappa^{(0)}(t,r,s) := p_\zeta^{(\alpha)}(t,r,s), \quad \text{and} \\
      p_\kappa^{(n)}(t,r,s) & := \int_0^t d\tau \int_0^\infty dz\, z^{2\zeta} p_\zeta^{(\alpha)}(t,r,z) \frac{\kappa}{z^\alpha} p_{\kappa}^{(n-1)}(t-\tau,z,s), \quad n\in\N.
    \end{split}
  \end{align*}
  Our goal is to prove that $\tilde p_\kappa(t,r,s)=\infty$ for $s<t$.
  By the definitions of $\tilde p_\kappa$ and $p_{\zeta,\eta}^{(\alpha)}$, we have $\tilde p_{\kappa} \geq p_{\zeta,\eta_*}^{(\alpha)}$.
  In fact, according to \cite[Lemma 8]{Bogdanetal2008}, $\tilde p_{\kappa}$ may be considered as a Schr\"odinger perturbation of $p_{\zeta,\eta_*}^{(\alpha)}$ by the positive potential $(\kappa-\kappa_{\rm c}(\zeta,\alpha))r^{-\alpha}$. Using the second term in the resulting perturbation series 
and the lower heat kernel bound in Theorem~\ref{mainresultgen}, we get
  \begin{align*}
    \tilde p_{\kappa}(t,r,s)
    & \geq \left(\kappa-\kappa_{\rm c}(\zeta,\alpha)\right) \int_0^t d\tau \int_0^\infty dz\, z^{2\zeta} p_{\zeta,\eta_*}^{(\alpha)}(t-\tau,r,z) z^{-\alpha} p_{\zeta,\eta_*}^{(\alpha)}(\tau,z,s)
    = \infty,
  \end{align*}
for $s<t$.  This concludes the proof.
\end{proof}

\subsection{The case $\eta\in(-\alpha,0)$}

Throughout this section, we assume $\alpha\in(0,2)$.
Similarly as in \eqref{eq:defbighscaled}--\eqref{eq:defbigh}, and inspired by Proposition \ref{bogprop32} and 
Theorem~\ref{heatkernelharmonic}, we consider the total mass of the kernel $p_{\zeta,\eta}^{(\alpha)}$,
\begin{align}
  \label{eq:defbighscaledposcouplings0}
  M(t,r) := \int_0^\infty ds\, s^{2\zeta} p_{\zeta,\eta}^{(\alpha)}(t,r,s).
\end{align}
By the scaling \eqref{eq:scalingalphahardy}, we have the scaling relation
\begin{align}
  \label{eq:defbighscaledposcouplings}
  M(t,r)
  = \int_0^{\infty} ds\, s^{2\zeta} t^{-(2\zeta+1)/\alpha} p_{\zeta,\eta}^{(\alpha)}(1,t^{-1/\alpha}r,t^{-1/\alpha}s)
  = M(1,t^{-1/\alpha}r).
\end{align}
Note that, by Theorem~\ref{thm:1}, $p_{\zeta,\eta}^{(\alpha)}$ satisfies the Chapman--Kolmogorov equations \eqref{eq:Ch-Kolm}, the Duhamel formulae \eqref{e.Df}, and
\begin{align}
  \label{eq:negetasimpleest}  
  0 < p_{\zeta,\eta}^{(\alpha)}(t,r,s) \leq p_\zeta^{(\alpha)}(t,r,s) \quad \text{for all}\ r,s,t>0.
\end{align}
In particular, by the normalization \eqref{eq:normalizedalpha}, we have
\begin{align}
  \label{eq:hposcouplingssimple}
  0 \leq M(t,r) \leq \int_0^\infty ds\, s^{2\zeta} p_\zeta^{(\alpha)}(t,r,s) = 1.
\end{align}

\subsubsection{Upper bound}

\begin{proposition}[{Comp.~\cite[Proposition 3.1]{JakubowskiWang2020}}]
  \label{jakprop31}
  Let $\zeta\in(-1/2,\infty)$, $\alpha\in(0,2)$, and $\eta\in(-\alpha,0)$. Then, there is $C>0$ such that for all $r>0$,
  \begin{align}
    M(1,r) \leq C(1\wedge r^{-\eta}).
  \end{align}
\end{proposition}

\begin{proof}
  By Chapman--Kolmogorov and $p_{\zeta,\eta}^{(\alpha)}(1/3,z,w)\leq c$ for all $z,w>0$ (by~\eqref{eq:heatkernelalpha1weightedsubordinatedboundsfinal}), we have
  \begin{align}
    \label{eq:jakprop31aux1}
    \begin{split}
      p_{\zeta,\eta}^{(\alpha)}(1,r,s)
      & = \int_0^\infty dz\, z^{2\zeta} \int_0^\infty dw\, w^{2\zeta} p_{\zeta,\eta}^{(\alpha)}(1/3,r,z) p_{\zeta,\eta}^{(\alpha)}(1/3,z,w) p_{\zeta,\eta}^{(\alpha)}(1/3,w,s) \\
      & \leq c \int_0^\infty dz\, z^{2\zeta} \int_0^\infty dw\, w^{2\zeta} p_{\zeta,\eta}^{(\alpha)}(1/3,r,z) p_{\zeta,\eta}^{(\alpha)}(1/3,w,s) \\
      & = c M(1/3,r) M(1/3,s) = c M(1,3^{1/\alpha}r) M(1,3^{1/\alpha}s).
    \end{split}
  \end{align}
  Next, let $\rho>0$ be so small that $\delta:=c\cdot (2\zeta+1)^{-1}\rho^{2\zeta+1} < 3^{\eta/\alpha}$. Recall $\eta<0$. Then, by \eqref{eq:jakthm241}, \eqref{eq:jakprop31aux1}, and \eqref{eq:hposcouplingssimple}, we have for any $r>0$,
  \begin{align}
    \label{eq:jakprop31aux2}
    \begin{split}
      M(1,r)
      & \leq \int_0^\rho ds\, s^{2\zeta} \,p_{\zeta,\eta}^{(\alpha)}(1,r,s) + \rho^{\eta} \int_\rho^\infty ds\, s^{2\zeta} \,p_{\zeta,\eta}^{(\alpha)}(1,r,s) s^{-\eta} \\
      & \leq \int_0^\rho ds\, s^{2\zeta} \,p_{\zeta,\eta}^{(\alpha)}(1,r,s) + A r^{-\eta} \\
      & \leq c\int_0^\rho ds\, s^{2\zeta} M(1,3^{1/\alpha}r) M(1,3^{1/\alpha}s) + A r^{-\eta} \\
      & \leq \delta M(1,3^{1/\alpha}r) + A r^{-\eta},
    \end{split}
  \end{align}
  where $A:=\rho^{\eta}$. Iterating \eqref{eq:jakprop31aux2} yields for all $r>0$,
  \begin{align}
    \begin{split}
      M(1,r)
      & \leq \delta M(1,3^{1/\alpha}r) + A r^{-\eta} \\
      & \leq \delta \left[\delta M(1,3^{2/\alpha}r) + A (3^{1/\alpha}r)^{-\eta}\right] + A r^{-\eta} \\
      & \leq \delta^2 \left[\delta M(1,3^{3/\alpha}r) + A (3^{2/\alpha}r)^{-\eta}\right] + A (1+\delta 3^{-\eta/\alpha}) r^{-\eta} \\
      & \leq \cdots \\
      & \leq \delta^n M(1,3^{n/\alpha}r) + A \left [1 + \delta 3^{-\eta/\alpha} + \cdots + (\delta 3^{-\eta/\alpha})^{n-1} \right] r^{-\eta}.
    \end{split}
  \end{align}
  Since $\delta<1$ and $\delta 3^{-\eta/\alpha}<1$, the above geometric series converges. Hence, by \eqref{eq:hposcouplingssimple}, we obtain
  \begin{align}
    M(1,r) \leq \frac{A}{1-\delta 3^{-\eta/\alpha}} r^{-\eta},
  \end{align}
  which concludes the proof.
\end{proof}

Applying Proposition \ref{jakprop31} to \eqref{eq:jakprop31aux1} immediately yields

\begin{corollary}
  \label{jakcor32}
  Let $\zeta\in(-1/2,\infty)$, $\alpha\in(0,2)$, $\eta\in(-\alpha,0)$. Then, there is a constant $C>0$ such that
  \begin{align}
    p_{\zeta,\eta}^{(\alpha)}(1,r,s)
    \leq C \left(1\wedge r^{-\eta}\right) \left(1\wedge s^{-\eta}\right),
    \quad r,s>0.
  \end{align}
\end{corollary}

We now refine these upper bounds for $p_{\zeta,\eta}^{(\alpha)}(t,r,s)$. 

\begin{lemma}
  \label{jaklem33}
  Let $\zeta\in(-1/2,\infty)$, $\alpha\in(0,2)$, and $\eta\in (-\alpha,0)$. Then, for any $r,s,t>0$, we have
  \begin{align}
    \int\limits_{|z-s|<|r-s|/2} dz\, z^{2\zeta} p_\zeta^{(\alpha)}(t,r,z) p_\zeta^{(\alpha)}(t,z,s)
    \leq \frac{p_{\zeta}^{(\alpha)}(2t,r,s)}{2}.
  \end{align}
\end{lemma}

\begin{proof}
  Fix $r,s,t>0$. Then, by $p_\zeta^{(\alpha)}(t,r,s) = p_\zeta^{(\alpha)}(t,s,r)$ for all $r,s,t>0$,
  \begin{align}
    \int\limits_{|z-s|<|r-s|/2} dz\, z^{2\zeta} p_\zeta^{(\alpha)}(t,r,z) p_\zeta^{(\alpha)}(t,z,s)
    = \int\limits_{|z-r|<|r-s|/2} dz\, z^{2\zeta} p_\zeta^{(\alpha)}(t,r,z) p_\zeta^{(\alpha)}(t,z,s).
  \end{align}
  By Chapman--Kolmogorov,
  \begin{align}
    \begin{split}
      & 2\int\limits_{|z-s|<|r-s|/2} dz\, z^{2\zeta} p_\zeta^{(\alpha)}(t,r,z)p_\zeta^{(\alpha)}(t,z,s) \\
      & \quad = \int\limits_{|z-s|<|r-s|/2} dz\, z^{2\zeta} p_\zeta^{(\alpha)}(t,r,z)p_\zeta^{(\alpha)}(t,z,s)
      + \int\limits_{|z-r|<|r-s|/2} dz\, z^{2\zeta} p_\zeta^{(\alpha)}(t,r,z)p_\zeta^{(\alpha)}(t,z,s) \\
      & \quad \leq \int_0^\infty dz\, z^{2\zeta} p_\zeta^{(\alpha)}(t,r,z) p_\zeta^{(\alpha)}(t,z,s)
      = p_{\zeta}^{(\alpha)}(2t,r,s),
    \end{split}
  \end{align}
  which concludes the proof.
\end{proof}

\begin{lemma}
  \label{jaklem34}
  Let $\zeta\in(-1/2,\infty)$, $\alpha\in(0,2)$, and $\eta\in(-\alpha,0)$.
  Let also $h(t,r):=(r t^{-1/\alpha})^{-\eta}$.
  Then, there is a constant $A>0$ such that for all $r,s,t>0$, we have
  \begin{align}
    p_{\zeta,\eta}^{(\alpha)}(t,r,s)
    \leq \int\limits_{|z-s|<|r-s|/2} dz\, z^{2\zeta} p_{\zeta,\eta}^{(\alpha)}(t/2,r,z) p_{\zeta,\eta}^{(\alpha)}(t/2,z,s) + A h(t,r) p_\zeta^{(\alpha)}(t,r,s).
  \end{align}
\end{lemma}

\begin{proof}
  By Chapman--Kolmogorov for $p_{\zeta,\eta}^{(\alpha)}$, we have for any $r,s,t>0$,
  \begin{align}
    \begin{split}
      p_{\zeta,\eta}^{(\alpha)}(t,r,s)
      & = \int\limits_{|z-s|\leq|r-s|/2} dz\, z^{2\zeta} p_{\zeta,\eta}^{(\alpha)}(t/2,r,z) p_{\zeta,\eta}^{(\alpha)}(t/2,z,s) \\
      & \quad + \int\limits_{|z-s|\geq|r-s|/2} dz\, z^{2\zeta} p_{\zeta,\eta}^{(\alpha)}(t/2,r,z) p_{\zeta,\eta}^{(\alpha)}(t/2,z,s).
    \end{split}
  \end{align}
  By \eqref{eq:defbighscaledposcouplings} and Proposition \ref{jakprop31}, we have
  \begin{align}
    \int_0^\infty ds\, s^{2\zeta} \,p_{\zeta,\eta}^{(\alpha)}(t,r,s)
    = M(t,r) = M(1,t^{-1/\alpha}r)
    \leq c\, h(t,r), \quad r,t>0.
  \end{align}
  For $t>0$ and $r,s,z>0$ with $|z-s|>|r-s|/2$, by \eqref{eq:negetasimpleest} and \eqref{eq:comparablealpha6}, we have
  \begin{align}
    p_{\zeta,\eta}^{(\alpha)}(t/2,z,s)
    \leq p_{\zeta}^{(\alpha)}(t/2,z,s)
    \lesssim p_\zeta^{(\alpha)}(t,r,s).
  \end{align}
  This implies for $r,s,t>0$,
  \begin{align}
    \begin{split}
      & \int\limits_{|z-s|\geq|r-s|/2}dz\, z^{2\zeta} p_{\zeta,\eta}^{(\alpha)}(t/2,r,z) p_{\zeta,\eta}^{(\alpha)}(t/2,z,s) \\
      & \quad \lesssim p_\zeta^{(\alpha)}(t,r,s) \int\limits_{|z-s|\geq|r-s|/2}dz\, z^{2\zeta} p_{\zeta,\eta}^{(\alpha)}(t/2,r,z) \\
      & \quad \lesssim p_\zeta^{(\alpha)}(t,r,s) \cdot h(t/2,r)  
      = 2^{-\eta/\alpha} h(t,r)p_\zeta^{(\alpha)}(t,r,s),
    \end{split}
  \end{align}
  which concludes the proof.
\end{proof}

We are now in position to derive the upper bounds in Theorem \ref{mainresultgen} in the following key lemma.

\begin{lemma}
  \label{jakthm35}
  Let $\zeta\in(-1/2,\infty)$, $\alpha\in(0,2)$, and $\eta\in(-\alpha,0)$. Then, for all $r,s,t>0$,
  \begin{align}
    \label{eq:jakthm35}
    p_{\zeta,\eta}^{(\alpha)}(t,r,s)
    \lesssim_{\zeta,\alpha,\eta} (1\wedge (t^{-1/\alpha}r)^{-\eta}) \cdot (1\wedge (t^{-1/\alpha}s)^{-\eta}) \cdot p_\zeta^{(\alpha)}(t,r,s).
  \end{align}
\end{lemma}

\begin{proof}
  By the scaling relation \eqref{eq:scalingalphahardy}, it suffices to consider $t=1$. Let $\delta=1/2$ and $\nu=2^{(-\eta-\alpha)/\alpha}<1$. As in Lemma \ref{jaklem34}, let $h(t,r)=(r/t^{1/\alpha})^{-\eta}$. Note that
  \begin{align}
    \label{eq:jakthm35aux1}
    \delta h(t/2,r)
    = \frac12 ((t/2)^{-1/\alpha}r)^{-\eta}
    = 2^{(-\eta-\alpha)/\alpha} (t^{-1/\alpha}r)^{-\eta}
    = \nu h(t,r), \quad r,t>0.
  \end{align} 
  Let $A$ be the constant from Lemma \ref{jaklem34}. We claim that for any $n\in\N_0$, we have
  \begin{align}
    \label{eq:jakthm35aux2}
    p_{\zeta,\eta}^{(\alpha)}(t,r,s)
    \leq \left[\delta^{n+1} + (1+\nu+\cdots+\nu^n) A h(t,r)\right] p_\zeta^{(\alpha)}(t,r,s),
    \quad r,s,t>0.
  \end{align}
  We prove \eqref{eq:jakthm35aux2} by induction. For $n=0$, $t>0$, and $r,s>0$, Lemmas \ref{jaklem34} (together with
  $p_{\zeta,\eta}^{(\alpha)}(t,r,s) \leq p_\zeta^{(\alpha)}(t,r,s)$) and \ref{jaklem33} imply
  \begin{align}
    p_{\zeta,\eta}^{(\alpha)}(t,r,s)
    \leq \left[ \delta + A h(t,r) \right] p_\zeta^{(\alpha)}(t,r,s).
  \end{align}
  We now make the induction step. Thus, suppose \eqref{eq:jakthm35aux2} was true for fixed $n\in\N$, i.e.,
  \begin{align}
    p_{\zeta,\eta}^{(\alpha)}(t,r,s)
    \leq \left[\delta^n + (1+\nu+\cdots+\nu^{n-1}) A h(t,r) \right] p_\zeta^{(\alpha)}(t,r,s),
    \quad r,s,t>0.
  \end{align}
  Then, for any $r,s,t>0$, Lemmas \ref{jaklem34} and \ref{jaklem33}, together with \eqref{eq:jakthm35aux1} yield
  \begin{align}
    \begin{split}
      & p_{\zeta,\eta}^{(\alpha)}(t,r,s)
       \leq \int\limits_{|z-s|\leq|r-s|/2}dz\, z^{2\zeta} p_{\zeta,\eta}^{(\alpha)}(\tfrac t2,r,z) p_{\zeta,\eta}^{(\alpha)}(\tfrac t2,z,s) + A h(t,r) p_\zeta^{(\alpha)}(t,r,s) \\
      & \ \leq \int\limits_{|z-s|\leq\frac{|r-s|}{2}}dz\, z^{2\zeta} \left[\delta^n + (1+\nu+\cdots+\nu^{n-1}) A h\left(\tfrac t2,r\right)\right] p_\zeta^{(\alpha)}\left(\tfrac t2,r,z\right) p_\zeta^{(\alpha)}\left(\tfrac t2,z,s\right) \\
      & \quad + A h(t,r) p_\zeta^{(\alpha)}(t,r,s) \\
      & \ \leq \left[ \delta^n + (1+\nu+\cdots+\nu^{n-1}) A h(t/2,r) \right] \cdot \delta p_\zeta^{(\alpha)}(t,r,s) + A h(t,r) p_\zeta^{(\alpha)}(t,r,s) \\
      & \ \leq \left[\delta^{n+1} + (\nu+\cdots+\nu^n) A h(t,r) \right] \cdot p_\zeta^{(\alpha)}(t,r,s) + A h(t,r)p_\zeta^{(\alpha)}(t,r,s) \\
      & \ = \left[ \delta^{n+1} + (1+\nu+\cdots+\nu^{n+1}) A h(t,r)\right] \cdot p_\zeta^{(\alpha)}(t,r,s),
    \end{split}
  \end{align}
  which proves \eqref{eq:jakthm35aux2}. Taking $t=1$, observing $h(1,r)=r^{-\eta}$ and letting $n\to\infty$ in~\eqref{eq:jakthm35aux2} thus yields
  \begin{align}
    \label{eq:jakthm35aux3}
    p_{\zeta,\eta}^{(\alpha)}(1,r,s)
    \leq \frac{A}{1-\nu} r^{-\eta} p_\zeta^{(\alpha)}(1,r,s),
    \quad r,s>0.
  \end{align}
  This estimate allows us to conclude the claimed upper bound \eqref{eq:jakthm35} (with $t=1$). By the symmetry of $p_{\zeta,\eta}^{(\alpha)}$ it suffices to assume $0<r<s$.
  For $1<r<s$, the claim \eqref{eq:jakthm35} follows from $p_{\zeta,\eta}^{(\alpha)}\leq p_\zeta^{(\alpha)}(1,r,s)$.
  For $0<r<s<1$ we use Corollary \ref{jakcor32} as well as $p_\zeta^{(\alpha)}(1,r,s)\gtrsim1$.
  Finally, for $0<r<1<s$, the claim \eqref{eq:jakthm35} follows from \eqref{eq:jakthm35aux3}.
  This concludes the proof of Lemma~\ref{jakthm35} and thereby the upper bound in Theorem~\ref{mainresultgen}.
\end{proof}

\subsubsection{Lower bound}

We begin with a lower bound on the function $M(t,r)$, defined in~\eqref{eq:defbighscaledposcouplings0}.

\begin{lemma}
  \label{jaklem36}
  Let $\zeta\in(-1/2,\infty)$, $\alpha\in(0,2)$, and $\eta\in(-\alpha,0)$. Then, there is $C>0$ such that
  \begin{align}
    \label{eq:jaklem36}
    M(1,r) \geq C(1\wedge r^{-\eta}), \quad r>0.
  \end{align}
\end{lemma}

\begin{proof}
  Let $R>0$ and $0<r<R/2$. Then, by \eqref{eq:heatkernelalpha1weightedsubordinatedboundsfinal}, there are $R$-independent numbers $c_1,c_2>0$ such that
  \begin{align}
    \int_R^\infty ds\, s^{2\zeta} p_\zeta^{(\alpha)}(1,r,s) s^{-\eta}
    \leq c_1 \int_R^\infty ds\, s^{2\zeta} \cdot \frac{s^{-\eta}}{s^{2\zeta+1+\alpha}}
    = c_2 R^{-\eta-\alpha},
  \end{align}
  and the right-hand side vanishes as $R\to\infty$. Let $C\equiv C_{\zeta,\alpha,\eta}$ be the constant in the upper heat kernel bounds \eqref{eq:jakthm35} and choose $R\geq1$ so large that $c_2CR^{-\eta-\alpha}\leq1/2$. Then,  by~\eqref{eq:jakthm242} in Theorem~\ref{jakthm24}, for $\rho\geq R$ and $r\in(0,\rho/2)$, we have
  \begin{align}
    \label{eq:jaklem36aux1}
    \begin{split}
      & \int_0^\infty ds\, s^{2\zeta} \,p_{\zeta,\eta}^{(\alpha)}(1,r,s)
        \geq \rho^{\eta} \int_0^\rho ds\, s^{2\zeta} \,p_{\zeta,\eta}^{(\alpha)}(1,r,s) s^{-\eta} \\
      & \quad = \rho^{\eta} \left( r^{-\eta} - \int_\rho^\infty ds\, s^{2\zeta} \,p_{\zeta,\eta}^{(\alpha)}(1,r,s) s^{-\eta}\right) \\
      & \quad \geq \rho^{\eta} \left(r^{-\eta} - C\int_\rho^\infty ds\, s^{2\zeta} \cdot r^{-\eta} p_\zeta^{(\alpha)}(1,r,s) s^{-\eta} \right)
      \geq \frac{r^{-\eta}}{2\rho^{-\eta}}.
    \end{split}
  \end{align}
  Thus, we proved
  \begin{align}
    \label{eq:jaklem36aux2}
    \begin{split}
      \int_0^\infty ds\, s^{2\zeta} \,p_{\zeta,\eta}^{(\alpha)}(1,r,s)
      \geq \frac{r^{-\eta}}{2R^{-\eta}}
      \geq \frac{1\wedge r^{-\eta}}{2R^{-\eta}}, \quad 0<r<\frac{R}{2}.
    \end{split}
  \end{align}
  On the other hand, if $r>R/2$, then, by taking $\rho=2r+1$ in \eqref{eq:jaklem36aux1}, we obtain
  \begin{align}
    \label{eq:jaklem36aux3}
    \int_0^\infty ds\, s^{2\zeta} \,p_{\zeta,\eta}^{(\alpha)}(1,r,s)
    \geq \frac{r^{-\eta}}{2(2r+1)^{-\eta}}
    \geq \frac{r^{-\eta}}{2(4r)^{-\eta}}
    \geq \frac{1}{4^{-\eta+1/2}}, \quad r>\frac{R}{2}.
  \end{align}
  Combining \eqref{eq:jaklem36aux2}--\eqref{eq:jaklem36aux3} yields \eqref{eq:jaklem36} and concludes the proof.
\end{proof}

To get the desired lower bounds in Theorem \ref{mainresultgen}, we compare $p_\zeta^{(\alpha)}(t,r,s)$ with $p_{\zeta,\eta}^{(\alpha)}(t,r,s)$ and distinguish between $r\vee s\lesssim1$ and $r\wedge s\gtrsim1$. We first focus on the case $r\wedge s\gtrsim1$. 
To that end, recall the function $p_t^{(1,D)}(r,s)$ from the perturbation series in \eqref{eq:feynmankactransformed}.
By the scaling \eqref{eq:scalingalpha} of $p_\zeta^{(\alpha)}$, we have the scaling
\begin{align}
  \label{eq:scalingauxsemigroup}
  p_t^{(1,D)}(r,s) = t^{-\frac{2\zeta+1}{\alpha}} p_1^{(1,D)}(t^{-1/\alpha}r,t^{-1/\alpha}s),
  \quad r,s,t>0.
\end{align}

The statement of the following lemma is the bound in \eqref{e.cmc}, proved in Appendix~\ref{s:constructionnegschrodperturbation}. It is closely related to estimates in \cite[Section 6]{Bogdanetal2008} or \cite[Section 6]{JakubowskiWang2020}. It is helpful to obtain lower bounds for $p_{\zeta,\eta}^{(\alpha)}(1,r,s)$ when $r,s\gtrsim1$ from upper bounds for $p_1^{(1,D)}$, since $\Psi_\zeta(\eta)<0$ for $\eta<0$.

\begin{lemma}
  \label{jaklem38}
  Let $\zeta\in(-1/2,\infty)$, $\alpha\in(0,2)$, and $\eta\in(-\alpha,0)$. Then, for all $r,s,t>0$, 
  \begin{align}
    p_{\zeta,\eta}^{(\alpha)}(t,r,s)
    \geq p_\zeta^{(\alpha)}(t,r,s) \exp\left(\frac{p_t^{(1,D)}(r,s)}{p_\zeta^{(\alpha)}(t,r,s)}\right).
  \end{align}
\end{lemma}

Our task thus consists of finding good upper bounds for $p_1^{(1,D)}(r,s)$ when $r,s\gtrsim1$. 
The following lemma, in combination with Lemma \ref{jaklem38}, is important for the derivation of lower bounds for $p_{\zeta,\eta}^{(\alpha)}(1,r,s)$ when $r,s\gtrsim1$. 

\begin{lemma}
  \label{jaklem37}
  Let $\zeta\in(-1/2,\infty)$, $\alpha\in(0,2)$. Then, for all $r,s>1$, we have
  \begin{align}
    \label{eq:jaklem37}
    p_1^{(1,D)}(r,s)
    \lesssim  p_\zeta^{(\alpha)}(1,r,s).
  \end{align}
\end{lemma}

Note that \cite[Lemma 3.7]{JakubowskiWang2020} proved both upper and lower bounds. However, the upper bounds suffice for the subsequent analysis.

\begin{proof}[Proof of Lemma \ref{jaklem37}]
  It suffices, by Chapman--Kolmogorov and \eqref{eq:pzetafact}, to estimate
  \begin{align}
  \label{eq:jaklem37aux}
    \begin{split}
      & \int_0^1 d\tau \int_0^\infty dz\, z^{2\zeta} p_\zeta^{(\alpha)}(1-\tau,r,z) z^{-\alpha} p_\zeta^{(\alpha)}(\tau,z,s) \\
      & \quad \le 2^\alpha p_\zeta^{(\alpha)}(1,r,s) 
      +  \int_0^1 d\tau \int_0^{1/2} dz\, z^{2\zeta} p_\zeta^{(\alpha)}(1-\tau,r,z) z^{-\alpha} p_\zeta^{(\alpha)}(\tau,z,s) \\
      & \quad \lesssim  p_\zeta^{(\alpha)}(1,r,s) 
      +  \int_0^1 d\tau \int_0^{1/2} dz\, z^{2\zeta-\alpha} \frac{1}{(1+r)^{1+\alpha+2\zeta}(1+s)^{1+\alpha+2\zeta}} \\
      &\quad \lesssim  p_\zeta^{(\alpha)}(1,r,s).
    \end{split}
  \end{align}
  In the last estimate we used again \eqref{eq:pzetafact} and the inequality $[(1+r)(1+s)]^{-1-\alpha-2\zeta} \lesssim  (1+|r-s|)^{-1-\alpha}(1+r+s)^{-2\zeta} $ valid for all $\zeta>-1/2$.
\end{proof}

Recall that $\Psi_\zeta(\eta)<0$ for $\eta<0$. Thus, Lemmas \ref{jaklem38} and \ref{jaklem37} lead to the following useful estimate, which suffices to prove the aspired lower heat kernel bounds for $p_{\zeta,\eta}^{(\alpha)}(1,r,s)$ when $r\wedge s\gtrsim1$.

\begin{corollary}
  \label{jakcor39}
  \label{jaklem310}
  Let $\zeta\in(-1/2,\infty)$, $\alpha\in(0,2)$, and $\eta\in(-\alpha,0)$. Then,  
  \begin{align}
    \label{eq:jaklem310}
    p_{\zeta,\eta}^{(\alpha)}(t,r,s)
    \gtrsim p_\zeta^{(\alpha)}(t,r,s), \quad  r,s\gtrsim t^{1/\alpha}>0.
  \end{align}
\end{corollary}

\begin{proof}
Since $\Psi_\zeta(\eta)<0$, the result follows by Lemmas \ref{jaklem38} and \ref{jaklem37}.
\end{proof}

Corollary \ref{jakcor39} suffices for the derivation of the lower heat kernel bounds for $r\wedge s\gtrsim1$. The following lemma will be used to derive the lower heat kernel bounds when $r\vee s\lesssim1$. In combination with Corollary \ref{jakcor39} it will also handle the case $r\lesssim 1 \lesssim s$.

\begin{lemma}
  \label{jaklem311}
  Let $\zeta\in(-1/2,\infty)$, $\alpha\in(0,2)$, and $\eta\in(-\alpha,0)$. 
  Then, there are $R>1$ and $C_R>0$ such that for all $r,s>0$ with $r\vee s\leq R$,
  \begin{align}
    \label{eq:jaklem311}
    p_{\zeta,\eta}^{(\alpha)}(1,r,s) \geq C_R (r\cdot s)^{-\eta}.
  \end{align}
\end{lemma}

\begin{proof}
  By \eqref{eq:jaklem36aux1} there is $R_0>2\cdot 3^{1/\alpha}$ large enough such that for all $0<r<R_0/2$,
  \begin{align}
    \int_0^{R_0} ds\, s^{2\zeta} \,p_{\zeta,\eta}^{(\alpha)}(1,r,s)
    \geq \frac{r^{-\eta}}{2R_0^{-\eta}}.
  \end{align}
  On the other hand, the upper heat kernel bounds \eqref{eq:jakthm35} imply for all $\rho_0>0$ and $r>0$,
  \begin{align}
    \int_0^{\rho_0} ds\, s^{2\zeta} \,p_{\zeta,\eta}^{(\alpha)}(1,r,s)
    \leq C r^{-\eta} \int_0^{\rho_0} ds\, s^{2\zeta} p_\zeta^{(\alpha)}(1,r,s) s^{-\eta}
    \leq C r^{-\eta} \cdot\rho_0^{-\eta}
  \end{align}
  where we used \eqref{eq:normalizedalpha} and $-\eta>0$ in the last step. We now take $\rho_0=(4C)^{1/\eta}R_0^{-1}$. Then, for $r<R_0/2$,
  \begin{align}
    \label{eq:jaklem311aux1}
    \int_{\rho_0}^{R_0} dz\, z^{2\zeta} \,p_{\zeta,\eta}^{(\alpha)}(1,r,z)
    \geq r^{-\eta} \left(\frac{1}{2R_0^{-\eta}} - C\rho_0^{-\eta}\right)
    = \frac{r^{-\eta}}{4R_0^{-\eta}}.
  \end{align}
  Therefore, by Chapman--Kolmogorov for $p_{\zeta,\eta}^{(\alpha)}$ and estimate \eqref{eq:jaklem311aux1},
  \begin{align}
    \begin{split}
      p_{\zeta,\eta}^{(\alpha)}(3,r,s)
      & \geq \int_{\rho_0}^{R_0} dw\, w^{2\zeta} \int_{\rho_0}^{R_0} dz\, z^{2\zeta} \,p_{\zeta,\eta}^{(\alpha)}(1,r,z) \,p_{\zeta,\eta}^{(\alpha)}(1,z,w) \,p_{\zeta,\eta}^{(\alpha)}(1,w,s) \\
      & \geq \frac{(rs)^{-\eta}}{16R_0^{-2\eta}} \inf_{z,w\in(\rho_0,R_0)} \,p_{\zeta,\eta}^{(\alpha)}(1,z,w).
    \end{split}
  \end{align}
  The infimum on the right is estimated with the help of \eqref{eq:jaklem310} by 
  \begin{align}
    \inf_{z,w\in(\rho_0,R_0)} p_{\zeta,\eta}^{(\alpha)}(1,z,w)
    \gtrsim_{\rho_0,R_0}\inf_{z,w\in(\rho_0,R_0)} p_\zeta^{(\alpha)}(1,z,w)
    \gtrsim_{\rho_0,R_0} 1,
  \end{align}
  where we used \eqref{eq:heatkernelalpha1weightedsubordinatedboundsfinal} (together with $R_0\gtrsim1$) to estimate the final infimum. Hence,
  \begin{align}
    p_{\zeta,\eta}^{(\alpha)}(3,r,s) \gtrsim (rs)^{-\eta}, \quad r\vee s<R_0/2.
  \end{align}
  Thus, by the scaling property of $p_{\zeta,\eta}^{(\alpha)}$ we obtain
  \begin{align}
    p_{\zeta,\eta}^{(\alpha)}(1,r,s)
    \gtrsim_{\zeta,\alpha,\eta,\rho_0,R_0} (rs)^{-\eta},
    \quad r\vee s<\frac{R_0}{2\cdot 3^{1/\alpha}}.
  \end{align}
  This concludes the proof of Lemma \ref{jaklem311}.
\end{proof}

We are now in position to prove the lower bounds in Theorem \ref{mainresultgen}.

\begin{lemma}
  \label{jakthm313}
  Let $\zeta\in(-1/2,\infty)$, $\alpha\in(0,2)$, and $\eta\in(-\alpha,0)$. Then,
  \begin{align}
    \label{eq:jakthm313}
    p_{\zeta,\eta}^{(\alpha)}(1,r,s)
    \gtrsim_{\zeta,\alpha,\eta} (1\wedge r^{-\eta})(1\wedge s^{-\eta}) p_\zeta^{(\alpha)}(1,r,s).
  \end{align}
\end{lemma}

\begin{proof}
  By symmetry, it suffices to consider $r\leq s$. For $0<r<1/4$, $s>1$, and $1/4\leq z\leq1/2$, Lemma~\ref{jaklem311} and \eqref{eq:jaklem310} imply
  \begin{align}
    p_{\zeta,\eta}^{(\alpha)}(1,r,z) \gtrsim r^{-\eta}
  \end{align}
  and
  \begin{align}
    p_{\zeta,\eta}^{(\alpha)}(1,s,z)
    \gtrsim p_\zeta^{(\alpha)}(1,s,z)
    \sim p_\zeta^{(\alpha)}(1,s,r)
  \end{align}
  where we used $s\geq 2z\geq 2r$, $s>1$, and \eqref{eq:heatkernelalpha1weightedsubordinatedboundsfinal} in the final step. Hence, for any $0<r<1/4$ and $s>1$, we have
  \begin{align}
    \label{eq:jakthm313aux1}
    \begin{split}
      p_{\zeta,\eta}^{(\alpha)}(1,r,s)
      & \geq \int_{1/4}^{1/2} dz\, z^{2\zeta} p_{\zeta,\eta}^{(\alpha)}(1/2,r,z) p_{\zeta,\eta}^{(\alpha)}(1/2,z,s)
        \gtrsim r^{-\eta} p_\zeta^{(\alpha)}(1/2,r,s) \\
      & \sim r^{-\eta} p_\zeta^{(\alpha)}(1,r,s).
    \end{split}
  \end{align}
  Finally, for $r\wedge s\geq 1/4$, the claim follows from \eqref{eq:jaklem310}, whereas for $r\vee s\leq 1$ the claim follows from Lemma \ref{jaklem311}.
  This concludes the proof of Lemma~\ref{jakthm313} and therefore the proof of the lower bound on Theorem \ref{mainresultgen}.
\end{proof}

\subsubsection{Continuity}

We now prove the continuity statement in Theorem~\ref{mainresultgen} for $\eta<0$.

\begin{lemma}
  \label{jaklem314}
  Let $\zeta\in(-1/2,\infty)$, $\alpha\in(0,2)$, and $\eta\in(-\alpha,0)$. Then, for any fixed $r>0$, the function $(0,\infty)\ni s\mapsto p_{\zeta,\eta}^{(\alpha)}(t,r,s)$ is continuous.
\end{lemma}

\begin{proof}
  Fix $r,s>0$ and let $z>0$ converge to $s$. Recall $q(r)=\Psi_\zeta(\eta)r^{-\alpha}$ with $\Psi_\zeta(\eta)<0$. By Duhamel's formula,
  \begin{align}
    \begin{split}
      & p_{\zeta,\eta}^{(\alpha)}(1,r,s) - p_{\zeta,\eta}^{(\alpha)}(1,r,z)
        = p_\zeta^{(\alpha)}(1,r,s) - p_\zeta^{(\alpha)}(1,r,z) \\
      & \quad + \int_0^1 d\tau \int_0^\infty dw\, w^{2\zeta} \,p_{\zeta,\eta}^{(\alpha)}(1-\tau,r,w)q(w)\left(p_\zeta^{(\alpha)}(\tau,w,s) - p_\zeta^{(\alpha)}(\tau,w,z)\right).
    \end{split}
  \end{align}
  For all sufficiently small $\epsilon$, estimates \eqref{eq:negetasimpleest}, \eqref{eq:heatkernelalpha1weightedsubordinatedboundsfinal} and \eqref{eq:hbetagammataurintegratedtransformed} imply
  \begin{align}
    \label{eq:jaklem314aux1}
    \begin{split}
      & -\int_0^\epsilon d\tau \int_0^\infty dw\, w^{2\zeta} \,p_{\zeta,\eta}^{(\alpha)}(1-\tau,r,w)q(w) p_\zeta^{(\alpha)}(\tau,w,s) \\
      & \quad \leq -\int_0^\epsilon \int_0^\infty dw\, w^{2\zeta} p_\zeta^{(\alpha)}(1-\tau,r,w) q(w) p_\zeta^{(\alpha)}(\tau,w,s) \\
      & \quad \lesssim -\int_0^\epsilon \int_0^\infty dw\, w^{2\zeta} p_\zeta^{(\alpha)}(\tau,w,s)q(w)
        \lesssim \epsilon s^{-\alpha}.
    \end{split}
  \end{align}
  Analogously, we have
  \begin{align}
    -\int_0^\epsilon d\tau \int_0^\infty dw\, w^{2\zeta} \,p_{\zeta,\eta}^{(\alpha)}(1-\tau,r,w)q(q)p_\zeta^{(\alpha)}(\tau,w,z)
    \lesssim \epsilon z^{-\alpha}.
  \end{align}
  Next, for any $\epsilon<\tau\leq1$ and $w,s,z>0$ with $z\to s$, we have
  $p_\zeta^{(\alpha)}(\tau,w,s)\sim p_\zeta^{(\alpha)}(\tau,w,z)$ (e.g., as a consequence of \eqref{eq:heatkernelalpha1weightedsubordinatedboundsfinal}). By dominated convergence,
  \begin{align}
    \int_\epsilon^1 d\tau \int_0^\infty dw\, w^{2\zeta} \,p_{\zeta,\eta}^{(\alpha)}(1-\tau,r,w)q(w)\left(p_\zeta^{(\alpha)}(\tau,w,s)-p_\zeta^{(\alpha)}(\tau,w,z)\right) \to0 \quad \text{as}\ z\to s.
  \end{align}
  Combining all previous estimates establishes the claim.
\end{proof}

\begin{proposition}
  \label{jakprop315}
  Let $\zeta\in(-1/2,\infty)$, $\alpha\in(0,2)$, and $\eta\in(-\alpha,0)$. Then, $p_{\zeta,\eta}^{(\alpha)}(t,r,s)$ is jointly continuous as a function of $r,s,t>0$.
\end{proposition}

\begin{proof}
  By scaling, it suffices to prove the continuity of $p_{\zeta,\eta}^{(\alpha)}(1,r,s)$ with respect to $r,s>0$. As indicated in the proof of Lemma \ref{jaklem314}, we only need to verify
  \begin{align*}
    \begin{split}
      \int_0^1 d\tau \int_0^\infty dw\, w^{2\zeta} \left| p_{\zeta,\eta}^{(\alpha)}(1-\tau,\tilde r,w) p_\zeta^{(\alpha)}(\tau,w,\tilde s) - p_{\zeta,\eta}^{(\alpha)}(1-\tau,r,w) p_\zeta^{(\alpha)}(\tau,w,s) \right| q(w) \to0
    \end{split}
  \end{align*}
  for any $r,s,\tilde r,\tilde s>0$ with $\tilde r\to r$ and $\tilde s\to s$. In addition to \eqref{eq:jaklem314aux1}, we have
  \begin{align}
    \begin{split}
      & -\int_{1-\epsilon}^1 d\tau \int_0^\infty dw\, w^{2\zeta} \,p_{\zeta,\eta}^{(\alpha)}(1-\tau,r,w) q(w) p_\zeta^{(\alpha)}(\tau,w,s) \\
      & \quad = - \int_0^\epsilon d\tau \int_0^\infty dw\, w^{2\zeta} p_{\zeta,\eta}^{(\alpha)}(\tau,r,w) q(w) p_\zeta^{(\alpha)}(1-\tau,w,s) \\
      & \quad \leq -\int_0^\epsilon d\tau \int_0^\infty dw\, w^{2\zeta} p_\zeta^{(\alpha)}(\tau,r,w) q(w) p_\zeta^{(\alpha)}(1-\tau,w,s)
      \lesssim \epsilon r^{-\alpha}
    \end{split}
  \end{align}
  by  the same argument as in the proof of \eqref{eq:jaklem314aux1}.

  For any $\epsilon<\tau<1-\epsilon$ and $r,s,z,\tilde r,\tilde s>0$ with $r\to\tilde r$ and $s\to\tilde s$, we have, by Lemma~\ref{jaklem314} both $p_\zeta^{(\alpha)}(\tau,z,\tilde s) \sim p_\zeta^{(\alpha)}(\tau,z,s)$ and $p_{\zeta,\eta}^{(\alpha)}(1-\tau,\tilde r,z)\sim p_\zeta^{(\alpha)}(1-\tau,r,z)$. Then, by dominated convergence,
  \begin{align*}
  \small
    \begin{split}
      \int_\epsilon^{1-\epsilon} d\tau \int_0^\infty dw\, w^{2\zeta} \left|p_{\zeta,\eta}^{(\alpha)}(1-\tau,\tilde r,w) p_\zeta^{(\alpha)}(\tau,w,\tilde s) - p_{\zeta,\eta}^{(\alpha)}(1-\tau,r,w)p_\zeta^{(\alpha)}(\tau,w,s) \right| q(w) \to 0
    \end{split}
  \end{align*}
  for any $r,s,\tilde r,\tilde s>0$ with $\tilde r\to r$ and $\tilde s\to s$. This concludes the proof of Proposition~\ref{jakprop315} and in particular the continuity in Theorem~\ref{mainresultgen}.
\end{proof}

\appendix
\section{Schr\"odinger perturbations}
\label{s:constructionnegschrodperturbation}

In this Appendix, we discuss Schr\"odinger perturbations of transition densities needed in Theorem~\ref{mainresultgen}.
However, the setting is more general and the results are of independent interest; in particular we allow transition densities inhomogeneous in time.

\subsection{General setting and positive perturbations}
Let $X$ be a locally compact space with a countable base of open sets. Consider the Borel $\sigma$-algebra $\mathcal{M}$ on $X$. Let $m$ be a $\sigma$-finite measure on $(X,\mathcal{M})$. 
We further consider an arbitrary nonempty \textit{interval} $\I$ on the real line $\R$, with the Borel $\sigma$-field $\mathcal B$ and the Lebesgue measure $ds$.   We call $\I\times X$ the \textit{space-time}.
The functions considered below are assumed to be (jointly) Borel measurable.
Let $p$ be a \textit{
transition density} on $X$ with \textit{time} in $\I$. This means that $p$ is a (jointly measurable) function on $\I \times X \times \I \times X$ with values in $[0,\infty]$ such that $p(s,x,t,y) =0$ for $s \ge t$ and the Chapman-Kolmogorov equations hold for $p$, i.e.,
\begin{align*}
p(s,x,t,y)=\int_{X} p(s,x,u,z) p(u,z,t,y)\, m(dz), \qquad s,t, u\in \I, \, s<u<t,\, x,y\in X.
\end{align*}
Below, we say that $p$ is finite and \emph{strictly positive} if $0<p(s,x,t,y)<\infty$ for all $s,t\in \I$, $s<t$ and $x,y\in X$.

We consider a measurable function (potential) $q\ge 0$ on $\I \times X$ and let 
\begin{align}
p_0(s,x,t,y)&:= p(s,x,t,y) \quad \mbox{and, for $n\in\N$,} \nonumber \\
\label{e.dpsn}
p_n(s,x,t,y) &:= \int_s^t \int_X p_{n-1}(s,x,u,z) q(u,z) p(u,z,t,y)(u,z)\,m(dz)\,du.
\end{align}
Here and below, $s,t\in \I$, $s<t$, and $x,y\in X$.
We define 
\begin{align}
  \label{eq:pert_series}
  p^q(s,x,t,y) := \sum_{n=0}^\infty  p_n(s,x,t,y)
\end{align}
and say that $p^q$ is a Schr\"odinger perturbation of $p$ by $q$. Specifically, the perturbation is \textit{positive} since $q\ge 0$ and mass-creating since $p^q\ge p$. 
Of course, $p^q$ is increasing in $q\ge 0$. 
By \cite[Lemma~2]{Bogdanetal2008}, we have the Chapman-Kolmogorov equations
\begin{equation}
  \label{eq:Ch-Kolm}
  p^{q}(s,x,t,y) = \int_X p^{q}(s,x,u,z)  p^{q}(u,z,t,y)\,m(dz),\quad s<u<t.
\end{equation}
By Tonelli,
$p^q(s,x,t,y)$  also satisfies the following perturbation (Duhamel) formulas
\begin{align}
\nonumber
p^q(s,x,t,y) &= p(s,x,t,y) + \int_s^t \int_X p^q(s,x,u,z) q(u,z) p(u,z,t,y)\,m(dz)\,du\\
&=
p(s,x,t,y) + \int_s^t \int_X p(s,x,u,z) q(u,z) p^q(u,z,t,y)\,m(dz)\,du.\label{e.Df}
\end{align}
As straightforward as positive Schr\"odinger perturbations are, $p^q$ may be infinite. For example, this is so for the Gaussian kernel $p(s,x,t,y)=(4\pi (t-s))^{-d/2}e^{-|y-x|^2/(4(t-s))}$ in $\Rd$, $d>2$, and $q(x)=c|x|^{-2}$, $x\in \Rd$, if (and only if) $c >(d-2)^2/4$, which underlies the famous result of Baras and Goldstein \cite{BarasGoldstein1984}. Needless to say, the example is 
\textit{time-homogeneous}, meaning that $p(s,x,s+t,y):=p_t(x,y)$,
$s\in \R$, $t>0$, 
and $q(s,x):=q(x)$,
$s\in \R$, yield 
a (time-homogeneous) perturbation $p_t^q(x,y):=p^q(0,x,t,y)=p^q(s,x,s+t,y)$, $s\in \R$, $t>0$.
It is in such a time-homogeneous setting that the results of this Appendix are used in Subsection~\ref{ss:mainresult}.

\subsection{Negative perturbations}

We now focus on functions $q$ (or $-q$) which may be negative. To facilitate discussion, we assume $p$ is sub-Markov, i.e., we assume $\int_X p(s,x,t,y) m(dy)\le1$.
Then we can extend the definitions
\eqref{e.dpsn}---first to the case
of \textit{bounded}, but not necessarily positive $q:\I\times X\to \R$.
Indeed, in this case we can estimate the integrals by using $|q|$ and $\|q\|_\infty:=\sup_{s\in \I,x\in X} |q(s,x)|$.
In particular, by induction, we get
\begin{align}
  \label{eq:boundpnboundedq}
  |p_n(s,x,t,y)|\le \|q\|_\infty^n \frac{(t-s)^n}{n!} p(s,x,t,y),\quad n=0,1,\ldots.
\end{align} 
We see that the perturbation series converges absolutely and
\begin{align}
  \label{eq:boundtildepboundedq}
  |p^q(s,x,t,y)|\le p^{|q|}(s,x,t,y)\le e^{\|q\|_\infty (t-s)}p(s,x,t,y).
\end{align}
By \cite[Lemma~2]{Bogdanetal2008}, $p^q$ satisfies Chapman-Kolmogorov equations.
By Fubini, $p^q$  also satisfies the Duhamel formulas~\eqref{e.Df}.

Next, we shall define and study rather general negative
perturbations of transition densities. To this end, we consider integral kernels $P(s,x,B)$ on the space-time, where $s\in \I$, $x\in X$, and $B$ is a (measurable) subset of $\I\times X$. Thus, by \cite[p. 38]{BliedtnerHansen1986},
\begin{align*}
  &(s,x)\mapsto P((s,x),B) \mbox{ is measurable for every } B, \\
  & B\mapsto P((s,x),B) \mbox{ is a measure on } \I\times X \mbox{ for every } s \mbox{ and } x.
\end{align*}
Then,
$
Pf(s,x):=\int f(t,y)P(s,x,dt\,dy)\ge 0
$ 
is measurable for every measurable function $f\ge 0$ on $\I\times X$. We will identify the \textit{kernel} $P$ with the \textit{operator} $P$. In fact, every additive, positively homogeneous and monotone operator on nonnegative measurable functions
\textit{is} an integral kernel, see \cite[Section II.1]{BliedtnerHansen1986}.

For instance, our transition density $p$ defines, for functions $f\ge 0$ on $\I\times X$,
\begin{equation}
  \label{e.dP}
  Pf(s,x) := \int\limits_{\I}\int\limits_X p(s,x,t,y) f(t,y)m(dy)dt,\quad s\in \I,\; x\in X.
\end{equation}
Here is a variation on \cite[Appendix]{Bogdanetal2020}, which points out a sub-Markov resolvent \cite[p.~46]{BliedtnerHansen1986} associated with the kernel $P$ in \eqref{e.dP}. The proof applies to $\I$ in the same way as to $\R$, with only minor adjustments; we provide it for the reader's convenience.

\begin{lemma}
  \label{l.CMPP}
  For \eqref{e.dP}, a sub-Markov resolvent $P^\lambda, \lambda>0$, exists with $\sup_{\lambda>0} P^\lambda = P$.
\end{lemma}

\begin{proof}
  For $\lambda\ge 0$, we let
  \begin{equation}\label{e.dr}
  p^{-\lambda}(s,x,t,y) :=  e^{-\lambda(t-s)}p(s,x,t,y),\quad s,t\in \I,\; x,y\in X,
  \end{equation}
  and, for $f\ge 0$, 
  \begin{equation}
    \label{resolvent}
    P^{\lambda} f(s,x) := \int_\I \int_X p^{-\lambda}(s,x,t,y) f(t,y) \, m(dy) \, dt, \quad s\in\I, \, x\in X.
  \end{equation}
  For clarity, note that $p^{-\lambda}(s,x,t,y) = 0$ for $s\ge t$, so in \eqref{resolvent}, we have $\int_\I=\int_s^S$, where $S:=\sup \I$. Of course, $\sup_{\lambda>0} P^\lambda = P$ as kernels. Furthermore, for $s\in \I$, $x\in X$, $\lambda>0$,
  \begin{equation*}
    \lambda P^\lambda (s,X)  = \lambda \int_\I\int_X p^{-\lambda}(s,x,t,y) \, m(dy) \, dt \le \lambda \int_s^\infty e^{-\lambda(t-s)}  dt \leq 1.
  \end{equation*}
  Finally, by the Chapman-Kolmogorov equations, for $\lambda>\mu$, $s\in \I$, $x\in X$, $f\ge 0$, 
  \begin{align}
    &P^\mu P^\lambda f(s,x) +(\lambda-\mu)^{-1}P^{\lambda}f(s,x)
    \nonumber
    \\
    = &  \int_\I\int_X  \int_\I\int_X p^{-\mu}(s,x,u,z)p^{-\lambda}(u,z,t,y)f(t,y) \, m(dy) \, dt \, m(dz) \, du+(\lambda-\mu)^{-1}P^{\lambda}f(s,x)     \nonumber
\\
    = & \int_s^S \int_X p(s,x,t,y) f(t,y) e^{\mu s} e^{-\lambda t} \int_s^t e^{-u(\mu-\lambda)} \, du \, m(dy) \, dt+(\lambda-\mu)^{-1}P^{\lambda}f(s,x)     \nonumber
\\
    = &  (\lambda-\mu)^{-1}  \int_X \int_s^S p(s,x,t,y) f(t,y) \left( e^{-\mu(t-s)} - e^{-\lambda(t-s)} \right)\, m(dy) \, dt    \nonumber
\\
     & + (\lambda-\mu)^{-1}P^{\lambda}f(s,x)
     =  (\lambda-\mu)^{-1} P^{\mu}f(s,x).\label{e.re}
  \end{align}
  Thus, $P^\lambda$, $\lambda>0$, is a sub-Markov resolvent.
\end{proof}

The notation $p^{-\lambda}$  in \eqref{e.dr} agrees with that in~\eqref{eq:pert_series}. In particular, the resolvent equation \eqref{e.re} with $\mu=0$ is a variant of the Duhamel formula for $q\equiv -\lambda$.

\smallskip
The following lemma is crucial for handling negative perturbations.
\begin{lemma}
  \label{l.CMP}
  If $f,g\ge 0$, $Pf\le Pg+1$ on $\{f>0\}$, then $Pf\le Pg+1$ on $\I\times X$. In particular, if $P|h|<\infty$ and $Ph\le 1$ on $\{h>0\}$, then $Ph\le 1$ on $\I\times X$. 
\end{lemma}

Here, of course, $\{f>0\}:=\{(s,x)\in \I\times X: f(s,x)>0\}$. We will call the first implication in Lemma~\ref{l.CMP} the \textit{complete maximum principle}, abbreviated as CMP. It is a variant of \cite[Proposition 7.1]{BliedtnerHansen1986}, but here we do not assume that $f$ is bounded and $Pf$ finite. CMP may be also viewed as a variant of the \textit{domination principle} \cite[Lemma~3.5~(ii)]{HansenBogdan2023}, but discussing this connection would take longer than the proof below.

\begin{proof}[Proof of Lemma~\ref{l.CMP}]
If $f\ge 0$ is bounded, then Lemma~\ref{l.CMPP} and the proof of \cite[Proposition 7.1]{BliedtnerHansen1986} give the implication.
In the case of general $f\ge 0$ with
$Pf\le Pg+1$ on $\{f>0\}$, we let $f_n:=f\wedge n$ for 
$n\in \N$.
On $\{f_n>0\}\subset \{f>0\}$, we have $Pf_n\le P f\le Pg+1$, so by the first part of the proof,  $Pf_n\le Pg+1$ on $\I\times X$. Letting $n\to \infty$, we get the first statement. For the second one, we let $f:=h_+$, $g:=h_-$.
\end{proof}

\begin{lemma}
  \label{lem:1}
  Let $q\ge 0$ be bounded. If $p$ is a strictly positive finite sub-Markov transition density then so is $p^{-q}$. Furthermore, $p^{-q}\le p$.
\end{lemma}

\begin{proof}
  Let $y\in X$, $t\in\I$, and $\I_t:=\{s\in \I: s<t\}$. 
  For $s,u\in \I_t$, $s<u$, $x,z\in X$, let
  $$
  h(s,x,u,z) := \frac{p(s,x,u,z)p(u,z,t,y)}{p(s,x,t,y)}.
  $$
  The function is a transition density on the space-time $\I_t\times X$, so by Lemma~\ref{l.CMP},
  \begin{align}
    \label{con:bridge}
    P_{t,y}f(s,x) &:= \int_{s}^t \int_X \frac{p(s,x,u,z)p(u,z,t,y)}{p(s,x,t,y)}f(u,z)\, m(dz)\, du, \quad s<t, \; x \in X, 
  \end{align}
is an integral kernel that  satisfies CMP. Let
  $$
  \widehat{P} f(s,x) := \int_{s}^t \int_X \frac{p(s,x,u,z)p(u,z,t,y)}{p(s,x,t,y)}f(u,z)q(u,z) m(dz)du.
  $$
  Then $\widehat{P}$ satisfies CMP, too. Indeed, let $f,g\ge 0$ on $\I_t\times X$ and $\widehat{P} f \le \widehat{P} g + 1$ on $\{f >0\}$. Let $f' :=  f q$, $g':=g q$. Since $\{f'>0\} \subset \{f>0\}$, we have $P_{t,y} f'\le P_{t,y}g'+1$ on $\{f'>0\}$. By CMP for $P_{t,y}$, $\widehat{P}f=P_{t,y} f' \le P_{t,y}g'+ 1=\widehat{P}g+1$ on $\I_t\times X$, which gives CMP for $\widehat P$.
  Next, let $h(s,x) := p^{-q}(s,x,t,y)/p(s,x,t,y)$. By \eqref{eq:boundtildepboundedq}, $\widehat{P} |h|<\infty$. By \eqref{e.Df},
  $$
  h = 1 - \widehat{P} h.
  $$
  Of course, if $h(s,x) > 0$, then $ \widehat{P} h(s,x) < 1$. By CMP for $\widehat{P}$, we get $\widehat{P} h \le 1$, so $h \ge 0$ and $0\le p^{-q}\le p$ everywhere.
  Furthermore, we note that for all $s,t\in \I$ and $x,y\in X$,
  \begin{align*}
    p^{-q}(s,x,t,y)
    & \ge p(s,x,t,y)-\sum_{n=1}^\infty p_n(s,x,t,y)\\
    & \ge \left(2-e^{\|q\|_\infty (t-s)}\right)p(s,x,t,y)>0
  \end{align*}
  if $0<t-s< \ln 2/\|q\|_\infty$. By \eqref{eq:Ch-Kolm}, $p^{-q}(s,x,t,y)>0$ for all $s,t\in \I$, $s<t$, $x,y\in X$.
\end{proof}

We say that the perturbation of $p$ in Lemma~\ref{lem:1} is \textit{negative} since $-q\le 0$ and mass-decreasing since $p^{-q}\le p$. The main point of Lemma~\ref{lem:1} is that $p^{-q}\ge 0$, in fact, that $p^{-q}$ is strictly positive for (bounded) $q\ge 0$. This  is a similar phenomenon as $e^{-x}> 0$ for $x\in [0,\infty)$.
In passing, we also note that for bounded functions $q,\rho\ge 0$ on $\I\times X$, by \cite[Lemma 8]{Bogdanetal2008},
\begin{equation}\label{e.pa}
p^{-q-\rho}=(p^{-q})^{-\rho}.
\end{equation}

In the reminder of this Appendix, we analyze $p^{-q}$ for possibly unbounded functions $q\ge 0$ on space-time. 
As we shall see, the following integrability assumption,
\begin{equation}
  \label{e.afp}
  p_1(s,x,t,y)=\int_s^t \int_X p(s,x,u,z)q(u,z)p(u,z,t,y)m(dz)du<\infty
  % \mbox{The perturbation series \eqref{eq:pert_series} is finite.} 
\end{equation}
for all $s,t\in \I$, $s<t$, $x,y\in X$, will suffice for construction of $p^{-q}$. Note that 
\begin{equation}\label{e.pqf}
p^q<\infty 
\end{equation}
implies \eqref{e.afp}, but not conversely; see the above discussion of the result of Baras and Goldstein. 
The advantage of \eqref{e.afp} is that for arbitrary $\lambda>0$, it holds for $q\ge 0$ if and only if it holds for $\lambda q$.
For the sake of
the discussion following Definition~\ref{def:pzetaeta}, we note that \eqref{e.afp} holds for the time-homogeneous semigroup $p(s,x,t,y):=p_\zeta^{(\alpha)}(t-s,x,y)$ and  potential  $q(s,x)=x^{-\alpha}$ there,
thanks to $0<\alpha<2\zeta+1$. This is because $p^{\lambda q}<\infty$ for small $\lambda=\Psi_\zeta(\eta)>0$, i.e., small  $\eta>0$; see Lemma~\ref{pzetaetafinite}.

Here is the main result of this Appendix.
\begin{theorem}
  \label{thm:1}
  Let $q:\I\times X\to [0,\infty]$. Let $p$ be a sub-Markov, strictly positive, finite transition density, and assume \eqref{e.afp}.
  Then there is a strictly positive transition density $p^{-q}\le p$ satisfying \eqref{eq:Ch-Kolm} and \eqref{e.Df}.
\end{theorem}

\begin{proof}
  For $n\in \N$, we let $q_n:=q\wedge n$. Then each $p^{-q_n}$ are sub-Markov, strictly positive, finite transition densities decreasing in $n$. Let 
  \begin{equation}\label{e.dnp}
  p^{-q}:=\inf_n p^{-q_n}.
  \end{equation}
  Of course, $0\le p^{-q}\le p$. By \eqref{e.afp} and the dominated convergence theorem, we get~\eqref{eq:Ch-Kolm} and~\eqref{e.Df}.
  The strict positivity of $p^{-q}$ follows from the lower bound
  \begin{equation}
    \label{e.cmc}
    p^{-q} (s,x,t,y) \ge p(s,x,t,y) \exp \left[-\dfrac{p_1(s,x,t,y)}{p(s,x,t,y)}\right],
  \end{equation}
  where $q\ge 0$, $s,t \in \I$, $s<t$, and $x,y \in X$. 
  For the proof of \eqref{e.cmc}, we first let $q\ge 0$ be \textit{bounded}. Then the function $(0,\infty)\ni \lambda\mapsto h(\lambda):=p^{-\lambda q}(s,x,t,y)$ is completely monotone, meaning that $(-1)^n h^{(n)}(\lambda)\ge 0$, $n=0,1,\ldots$, $\lambda>0$.
  Indeed, $(-1)^n h^{(n)}(\lambda)=n! (p^{-\lambda q})_n(s,x,t,y)$ by \cite[Appendix, formula (2)]{JakubowskiWang2020}, which is nonnegative in view of Lemma~\ref{lem:1} and~\eqref{e.dpsn}.
  Since completely monotone functions are logarithmically convex, we get
  \begin{align*}
    p^{-q}(s,x,t,y)=h(1)
    & = h(0)\exp\left[\int_0^1 (\ln h(\lambda))'d\lambda\right]
      \ge h(0)\exp \left[\int_0^1 \frac{h'(0)}{h(0)}d\lambda\right] \\
    & = p(s,x,t,y)\exp \left[-\frac{p_1(s,x,t,y)}{p(s,x,t,y)}\right].
  \end{align*}
  For unbounded $q\ge 0$, we use this, \eqref{e.dnp}, and \eqref{e.afp} and we let $n\to \infty$.
\end{proof}

\begin{remarks}
  (1) Estimate~\eqref{e.cmc} strengthens \cite[Appendix]{JakubowskiWang2020} and \cite[(41)]{Bogdanetal2008}.
  \\
  (2) In the time-homogeneous setting, \eqref{e.cmc} yields
\begin{equation}\label{e.cmc2}
    p^{-q}_t (x,y) \ge p_t(x,y) \exp \left[\frac{-\int_0^t\int_X p_s(x,z)q(z)p_{t-s}(z,y)m(dz)\,ds}{p_t(x,y)}\right].
  \end{equation}
  Here, of course, $p_t$ is sub-Markov, $q\ge 0$, and the numerator is assumed finite.
\end{remarks}

We conclude this appendix with the following example,
which illustrates the importance of CMP, and therefore the Chapman--Kolmogorov equations, for
the positivity of $p^{-q}$.

\begin{example}
  Let $X$ be the space containing only one point and $m(dz)$ be the Dirac delta at this point. Let $p(s,t) = (t-s)_+$ and $q(s) \equiv 1$ (to simplify notation, we skip the space variables from the notation). Then,
  $$
  p_n(s,t) = \frac{(t-s)_+^{2n+1}}{(2n+1)!},
  $$
  therefore
  \begin{align*}
    p^{q}(s,t) = \sinh(t-s)_+ \qquad \mbox{and}\qquad p^{-\lambda q}(s,t) = \frac{\sin(\sqrt{\lambda}(t-s)_+)}{\sqrt{\lambda}}.
  \end{align*}
  Clearly, $p^{-\lambda q}(s,t)$ takes on negative values, too. This does not contradict our findings because the operator $P f(s) := \int_s^\infty p(s,t) f(t) dt$ does not satisfy CMP, as may be verified by considering the function $f(s) = -{\bf 1}_{[0,1)}(s) + 2 \cdot {\bf 1}_{[1,2]}(s)$.
\end{example}

\printindex

% \bibliographystyle{alpha}
% \bibliography{coulomb}

\newcommand{\etalchar}[1]{$^{#1}$}
\def\cprime{$'$}

\end{document}